\documentclass[12pt]{article}
\usepackage[utf8]{inputenc}
\usepackage{amsthm}
\usepackage{amsmath}
\usepackage{amsfonts}
\usepackage{amssymb}
\usepackage[left=2cm,right=2cm,top=2cm,bottom=2cm]{geometry}

\usepackage{tkz-berge}
\usepackage[justification=centering]{caption}

\usepackage{tocloft}
\usepackage{appendix}

\usepackage{enumerate}

\usepackage{hyperref}

\usepackage{tabto}
\TabPositions{2.2 cm, 4.4 cm, 6.6 cm, 8.8 cm}
 
\newtheorem{theorem}{Theorem}[section]
\newtheorem{corollary}[theorem]{Corollary}
\newtheorem{lemma}[theorem]{Lemma}
\newtheorem{proposition}[theorem]{Proposition}

\theoremstyle{definition}
\newtheorem{definition}[theorem]{Definition}
\newtheorem{example}[theorem]{Example}
\newtheorem*{definition*}{Definition}

\newtheorem*{lemma*}{Lemma}
\newtheorem*{proposition*}{Proposition}
\newtheorem*{theorem*}{Theorem}
\newtheorem*{corollary*}{Corollary}

\theoremstyle{remark}
\newtheorem*{remark}{Remark}

\theoremstyle{remark}

\theoremstyle{definition}
\newtheorem{comment}{Comment}[section]

\begin{document}

\begin{titlepage}
\begin{center}

\vspace{1cm}
\Large{Francesco Fournier Facio}
\vfill
\huge{Comments on \textit{Discrete Groups, Expanding Graphs and Invariant Measures}, by Alexander Lubotzky} \\ [3mm]
\Large{Chapters 1-4} \\ 
\vspace{1cm}
\Large{Semester Paper} \\
\vfill
\large{ETHZ} \\
\large{Mathematics Department} \\
\today

\end{center}
\end{titlepage}

\pagebreak

\vspace*{5cm}

\begin{abstract}
This document is a collection of comments that I wrote down while reading the first four chapters of the book \textit{Discrete Groups, Expanding Graphs and Invariant Measures} by Alexander Lubotzky. Most of them are more detailed versions of proofs. Some imprecisions are pointed out and discussed, and some facts referenced in the book are proven. In the appendix we discuss topics of interest in relation to this book, which are however not necessary for its understanding. The aim of this document, which is not quite complete in that respect, is to provide, together with Lubotzky's book, a self-contained read.
\end{abstract}

\vspace{2cm}

\renewcommand{\abstractname}{Acknowledgements}

\begin{abstract}
First and foremost, I would like to thank my supervisor Konstantin Golubev, for following me weekly through the reading of Lubotzky's book and the writing of this document, as well as asking me to write my comments down to begin with. His remarks about the role that some results have in the rest of the book were especially helpful to keep in mind the big picture while going through details.

Secondly, I thank my advising professor Alessandra Iozzi, for being so encouraging and available.

I also wish to express my appreciation for Nicolas Monod for suggesting that I read this book, and for making me interested in amenability and property $(T)$ in the first place. His master course \textit{Analysis on Groups}, that I took at EPFL, is referenced in many parts of this document \cite{Monod}, and is the main reason I am interested in these topics.
\end{abstract}

\vfill

\pagebreak

\tableofcontents
\restoregeometry

\pagebreak

\section*{Notations and conventions}
\addcontentsline{toc}{section}{Notations and conventions}

Throughout this paper we use the following notation. \\

\noindent Let $X = (V, E)$ be a graph. \\
$V = V(X)$ \tab The set of vertices. \\
$E = E(X)$ \tab The set of edges. \\
$d$ \tab The distance function $V \times V \to \mathbb{Z}_{\geq 0} \cup \{ \infty \}$, where $d(v, w)$ is the length of \tab \tab the shortest path connecting $v$ and $w$. \\
$\partial A$ \tab For a subset $A \subseteq V$, the set of neighbours of $A$, i.e., $\{ v \in V : d(A, v) = 1 \}$. \\

\noindent $[x]$ \tab For $x \in \mathbb{R}$, the integer part of $x$, so $0 \leq x - [x] < 1$. \\

\noindent Let $X$ be a topological space. \\
$K \subseteq X$ \tab $K$ is a subset of $X$. \\
$K \subset X$ \tab $K$ is a proper subset of $X$. \\
$K \subseteq_f X$ \tab $K$ is a finite subset of $X$. \\
$K \subseteq_c X$ \tab $K$ is a compact subset of $X$. \\
$A \Delta B$ \tab For $A, B \subseteq X$, the symmetric difference of $A$ and $B$. \\

\noindent Let $X$ be a space with a measure $\lambda$ and an associated integral $\int$. \\
Two functions from $X$ to any other set are equivalent if they coincide almost everywhere. \\
$L^p(X)$ \tab The space of equivalence classes of functions $f : X \to \mathbb{R}$ such that $||f||_p :=$ \tab \tab $\left( \int_X |f|^p \right)^{\frac{1}{p}} < \infty$. \\
$L^\infty(X)$ \tab The space of equivalence classes functions $f : X \to \mathbb{R}$ that are bounded outside a set \tab \tab of measure 0. The minimal such bound for $f \in L^\infty(X)$, is denoted $||f||_\infty$. \\
$\chi_A$ \tab For a subset $A \subseteq X$, the characteristic function of $A$. \\
When the measure is the counting one (for example in discrete groups), we write $\ell^p, \ell^\infty$. \\

\noindent Let $G, G_1, G_2$ be topological groups. \\
$G_x$ \tab For $X$ a $G$-set and $x \in X$, the stabilizer of $x$. \\
$Z(G)$ \tab The center of $G$. \\
$\langle S \rangle$ \tab For a subset $S \subseteq G$, the subgroup generated by $S$. \\
$\tilde{G}$ \tab The collection of unitary representations of $G$. \\
$\hat{G}$ \tab The collection of irreducible unitary representations of $G$. \\
$L_G$ \tab The left regular representation of $G$ on $L^2(G)$, defined as $(gf)(x) = f(g^{-1}x)$. \\

\noindent Let $R$ be a ring. \\
$M_n(R)$ \tab The ring of $n \times n$ matrices on $R$. \\
$GL_n(R)$ \tab The group of invertible $n \times n$ matrices on $R$. \\
$SL_n(R)$ \tab The group of invertible matrices of determinant 1 on $R$. \\
$O_n(R)$ \tab The group of orthogonal matrices, i.e., $\{A \in GL_n(R) : A^{-1} = A^T\}$. \\
$SO_n(R)$ \tab The group of orthogonal matrices of determinant 1. \\
\tab \tab If $R = \mathbb{R}$, we note $O(n), SO(n)$ etc... \\
$U(n)$ \tab The group of unitary matrices, i.e., $\{A \in GL_n(\mathbb{C}) : A^{-1} = A^* = \overline{A}^T \}$. \\
$SU(n)$ \tab The group of unitary matrices of determinant 1.

\pagebreak

\section{Introduction}

During the Fall Semester of 2018, as part of my Master's degree at ETHZ, I started reading Alexander Lubotzky's book \textit{Discrete Groups, Expanding Graphs and Invariant Measures} for a reading course with Alessandra Iozzi, under the supervision of Konstantin Golubev. During our weekly meetings, I started making some comments on what I was reading, and he told me that they may be valuable because of the popularity and difficulty of the text. My comments were aimed at making the book self-contained from my point of view, as a first-year Master student who has some background in group theory, representation theory and graph theory, and a little less in differential geometry and measure theory. This document is therefore written for a similar audience. \\

The reader will note that some sections in chapters 1 to 4 are skipped completely. More specifically almost all of section 3.2, and sections 4.1 and 4.4. This is because my knowledge of algebraic groups and Riemannian geometry is simply too restrained to be able to understand these parts of the book well-enough to make any meaningful comment. In that sense this document is incomplete, as well as the fact that only the first four chapters are commented. \\

The comments follow the structure of the book, and will probably make little sense if not read together with it. We assume reasonable knowledge of graph theory, topology, group theory (including topological groups), basic representation theory and measure theory. Most of the comments amount to rewriting some of the most complex proofs in the book in complete detail. Some other complete passages that are left without proof, whether the proof is given in a reference or not. These have been chosen according to which results I wanted to understand more in depth, and I felt were the most important in relation to the book. A fair amount of comments are devoted to correcting imprecisions throughout the text. \\

The first appendix answers the following question: when restricting to regular expanding graphs, are we really not losing generality? In more than one instance in the book, the graphs appearing are not regular graphs, although they are regular multigraphs. We provide a general method for constructing a family of regular expanders out of a family of expanders of bounded degree, which essentially provides an affirmative answer to the question.

The second appendix studies the arithmetic of Hurwitz integral quaternions, which are used in the construction of free subgroups of $SO(3)$. Quaternions become fundamental in the construction of Ramanujan graphs later in the book. 

The third appendix is a discussion of amenable actions of discrete groups (in the book the author only references paradoxical actions), as well as a proof of various equivalent definitions of amenability, and a proof of Tarski's theorem, which is referenced in the book. This follows the approach taken in Monod's class \cite{Monod}.

The fourth appendix is a detailed proof of the fact that the Lebesgue measure is the unique countably additive measure of total measure 1 on $S^n$. Since a good part of this book is devoted to providing a positive answer to the Banach-Ruziewicz problem, it seemed suitable to understand this result more in depth.

A final appendix lists some typos. \\

Almost all of the results in this document are not original, although when no further specification is given, the proofs are mine. All that is, to my knowledge, original is comment \ref{latin}, appendix \ref{regularization}, and the approach taken in appendix \ref{iso}.

\pagebreak

\section{Expanding Graphs}

In this section, the author refers to $(n, k, c)$-expanders using two definitions: the first refers to an $n$-vertex graph, the second to a bipartite graph in which both parts have $n$-vertices. For clarity, we will talk about \textbf{bi-expanders} in the second case. We begin this section by introducing a third equivalent definition of expanders, which appears in other works of Lubotzky \cite{Lubotzky} and makes it easier to go from expanders to bi-expanders and back in remark 1.1.2 (ii) (comment \ref{112}). Once again, we use a different term to refer to them for clarity.

\begin{definition*}
Let $X = (V, E)$ be a $k$-regular graph with $n$ vertices. $X$ is an \textbf{$(n, k, c)$-fixed-expander} if for all $A \subseteq V$ with $|A| \leq \frac{n}{2}$, we have $|\partial A| \geq c|A|$.
\end{definition*}

We called them fixed-expanders since the expansion factor is fixed for small enough subsets, instead of varying with the size of the subset, like in the definition of expanders (definition 1.1.1). The advantage of introducing this definition is that it serves as a middle ground between that of expander and bi-expander. That is: to go from expanders to fixed-expanders we only need to change the constant, and to go from fixed-expanders to bi-expanders we only need to apply the constructions described in remark 1.1.2 (ii). \\

More precisely: an $(n, k, c)$-expander is an $(n, k, \frac{c}{2})$-fixed-expander, while an $(n, k, c)$-fixed-expander is an $(n, k, \frac{c}{k})$-expander.

\begin{proof}
Let $X = (V, E)$ be an $(n, k, c)$-expander, and let $A \subseteq V$ be such that $|A| \leq \frac{n}{2}$. Then $$|\partial A| \geq c(1 - \frac{|A|}{n})|A| \geq \frac{c}{2}|A|.$$
Here we see why we only want subsets $|A| \leq \frac{n}{2}$ to verify the expanding condition. If we asked it for every subset, then the quantity $(1 - \frac{|A|}{n})$ would attain a minimum of $\frac{1}{n}$ that goes to 0 as $n$ goes to infinity, so we would have a new constant $c' = o(n)$. Instead, we want $c'$ to be independent of $n$, so as to construct infinite families of fixed-expanders. \\

Let $X = (V, E)$ be an $(n, k, c)$-fixed-expander, and let $A \subseteq V$ and $B := V \, \backslash \, A$, so that $(1 - \frac{|A|}{n}) = \frac{|B|}{n}$. If $|A|  \leq \frac{n}{2}$, then $|\partial A| \geq c|A|$. If $|A| \geq \frac{n}{2}$, then $|B| \leq \frac{n}{2}$, so $|\partial B| \geq c|B|$. Now $|\partial B| \leq k|\partial A|$, since by double-counting the edges connecting $\partial A$ to $\partial B$ we see that there are at least $|\partial B|$ and at most $k |\partial A|$. Therefore $|\partial A| \geq \frac{c}{k}|B|$. In both cases,
$$|\partial A| \geq \frac{c}{k} \frac{|B| \, |A|}{n} = \frac{c}{k}(1 - \frac{|A|}{n})|A|,$$
so $X$ is an $(n, k, \frac{c}{k})$-expander. 
\end{proof}

Therefore any infinite family of expander gives rise to an infinite family of fixed-expanders, and vice-versa.

\subsection{Expanders and their applications}

\begin{comment}[Remark 1.1.2 (ii), PP. 1-2]
\label{112}

Here we prove the relation between expanders and bi-expanders. We will use the definition of fixed-expanders instead (see the previous paragraph). \\

\textit{From fixed-expanders to bi-expanders} \\

Let $X = (V, E)$ be an $(n, k, c)$-fixed-expander. We construct a bipartite graph $X' = (V', E')$: the vertex set is $V' = V_- \sqcup V_+$, where $V_\pm$ is a copy of $V$, and the edge set is as follows. For every $v_- \in V_-$, connect it to $v_+ \in V_+$, and for every $vw \in E$, add the edge $v_-w_+$ to $E'$. In other words, we connect each input to its twin and the twins of its neighbours. Thus $X$ is a $(k+1)$-regular bipartite graph, where each part has $n$ vertices.

We check the expanding condition. Let $A_- \subseteq V_-$ be a subset with at most $\frac{n}{2}$ elements. Let $A$ denote the corresponding subset of $V$, and $\partial, \partial'$ the "neighbours" operator in $X$ and $X'$ respectively. Then: $$|\partial' A| = |A| + |\partial A| \geq |A| + c|A| = (1 + c)|A|.$$
Therefore $X$ is an $(n, k+1, c)$-bi-expander. \\

Note that the degree has augmented, which was necessary to ensure that a small enough subset has at least as many neighbours as it has members. \\

\textit{From bi-expanders to fixed-expanders} \\

Let $X' = (V', E')$ be an $(n, k, c)$-bi-expander, with vertex set $V' = V_- \sqcup V_+$. Since regular bipartite graphs satisfy Hall's condition, there exists a perfect matching, so we can label the vertex set so as to assign to each $v_- \in V_-$ a match $v_+ \in V_+$. We then glue the twins together. The resulting graph is $X = (V, E)$, where $V$ is a copy of $V_-$ (or $V_+$), and $vw \in E$ if and only if $v_-w_+ \in E'$ or $w_-v_+ \in E'$. Thus $X$ is a graph with $n$ vertices. \\

Now in order to fit the book's definition of an expander, $X$ must be a $k'$-regular graph, for some $k'$. However, the degrees of the vertices of $X$ may vary depending on the symmetry of $X'$. More precisely, if $v \in V$, then the neighbours of $v$ are $N(v) = \{w \neq v : v_-w_+ \in E' \} \cup \{w \neq v : w_-v_+ \in E' \}$. Since $v_-$ and $v_+$ both have degree $k$ in $X$ and are connected to each other, both of these sets have cardinality $(k-1)$. Therefore $(k-1) \leq d(v) \leq 2(k-1)$. All of these inequalities are sharp: figure \ref{bi_to_exp} shows an example of a 3-regular bipartite graph with two parts of 7 vertices, such that after this construction all possible degrees appear.

\begin{figure}
\centering
\captionsetup{justification=centering}

\begin{tikzpicture}
\SetVertexNoLabel

\grEmptyLadder[RA = 2.5, RS = 4]{7}

\AssignVertexLabel{a}{$1_+$, $2_+$, $3_+$, $4_+$, $5_+$, $6_+$, $7_+$};
\AssignVertexLabel{b}{$1_-$, $2_-$, $3_-$, $4_-$, $5_-$, $6_-$, $7_-$};

\EdgeFromOneToSel{a}{b}{0}{0, 3, 4}
\EdgeFromOneToSel{a}{b}{1}{0, 1, 5}
\EdgeFromOneToSel{a}{b}{2}{0, 2, 6}
\EdgeFromOneToSel{a}{b}{3}{1, 3, 5}
\EdgeFromOneToSel{a}{b}{4}{2, 4, 6}
\EdgeFromOneToSel{a}{b}{5}{1, 3, 5}
\EdgeFromOneToSel{a}{b}{6}{2, 4, 6}
\end{tikzpicture}

\vspace{1.5cm}

\begin{tikzpicture}
\SetVertexNoLabel

\grEmptyCycle[prefix = b, RA = 2.8]{2}

\AssignVertexLabel{b}{$7$, $6$};

\begin{scope}[rotate=45]
\grStar[RA = 2]{5}
\end{scope}

\AssignVertexLabel{a}{$3$, $2$, $4$, $5$, $1$};

\EdgeFromOneToSel{b}{a}{0}{0, 3}
\EdgeFromOneToSel{b}{a}{1}{1, 2}
\EdgeFromOneToSel{a}{a}{0}{3}
\EdgeFromOneToSel{a}{a}{1}{2}
\end{tikzpicture}

\caption{We start with a bipartite 3-regular graph with two parts of 7 vertices. The resulting graph has vertices of degrees 2, 3 and 4.\label{bi_to_exp}}

\end{figure}

This is the first instance in the book where the graphs appearing have bounded degree (by $2(k - 1)$) but are not regular. Still, the bounded degree is enough for our purposes, for instance, notice that passing from fixed-expanders to expanders, if we only had bounded degree the changes in the constants would be the same. For the rest of this comment, we will use the definition where $k' = 2(k - 1)$ denotes a bound on the degree. However, in appendix \ref{regularization} we explain how to deal with this problem in a general setting, together with a special case that appears later in the book (section 3.3). \\

We check the expanding condition. Let $A \subseteq V$ be such that $|A| \leq \frac{n}{2}$. Let $A_\pm$ be the copy of $A$ in $V_\pm$, so that $|\partial' A_-| \geq (1 + c)|A_-| = (1 + c)|A|$. Then $$|\partial A| \geq |(\partial' A_-) \, \backslash \, A_+| \geq |\partial' A_-| - |A_+| \geq (1 + c)|A| - |A| = c|A|.$$
(As a matter of fact, the second inequality is an equality, since by construction $A_+ \subseteq \partial' A_-$).

Therefore $X$ is an $(n, k', c)$-fixed-expander.

\end{comment}

\begin{comment}[P. 2]
\label{match}

The arguments in this comment come from a discussion on a StackExchange post \cite{stack_Hall}. \\

Note that to establish a correspondence between expanders and bi-expanders, we needed to find a perfect matching in the bi-expander. This works because regular bipartite graphs satisfy Hall's condition, as we said in the previous comment. However, if we treated expanders only as graphs with bounded degree, we would need something more: the expanding property only works on sets of less than $\frac{n}{2}$ vertices, while Hall's condition must be satisfied by all subsets of inputs. \\

To see that we cannot get around this, for all $n$ let $X$ be the bipartite graph constructed as follows: start with the complete bipartite graph $K_{4n, 3n}$, then add $n$ outputs and connect each to an input in such a way that some two of these new outputs are connected to the same input. Then we do have the expanding property, for $c = \frac{1}{2}$, for all subsets of at most $2n$ inputs, but the fact that two outputs are connected to the same input shows that there is no perfect matching. If we wanted a bounded-degree construction, we could use a bounded concentrator instead of a complete bipartite graph. \\

One additional property that we could add to ensure the existence of a perfect matching, is that the expansion happens both ways, and not just from inputs to outputs. This follows from the following fact:

\begin{lemma*}

Let $X = (I \sqcup O, E)$ be a bipartite graph with $n$ inputs and $n$ outputs, such that for all $A \subset I$ and for all $A \subset O$ such that $|A| \leq \frac{n}{2}$, we have $|\partial A| \geq |A|$. Then $X$ has a perfect matching.

\end{lemma*}

\begin{proof}

We only need to check Hall's condition for sets of inputs of more than $\frac{n}{2}$ vertices, so let $A \subseteq I$ be such a set. Let $B := O \, \backslash \, \partial A$. Since $|\partial A| \geq \frac{n}{2}$, we have $|B| \leq \frac{n}{2}$, so the hypothesis applies to $B$ and $|\partial B| \geq |B|$. Then, since $\partial B \cap A = \emptyset$, we have $|B| \leq |\partial B| \leq |I \, \backslash \, A| = n - |A|$. So $|\partial A| = |O \, \backslash \, A| = n - |B| \geq |A|$.
\end{proof}

\end{comment}

\begin{comment}[Definition 1.1.3, P. 2]

Some authors \cite{Cheeger} define the Cheeger constant as
$$h'(X) := \inf\limits_{A \sqcup B = V} |E(A, B)| \left( \frac{1}{|A|} + \frac{1}{|B|} \right).$$
However, asking for $h$ or $h'$ to be bounded for an infinite family of graphs is essentially equivalent, since the two quantities are linearly related. Indeed, for any $A \sqcup B = V$:
$$|E(A, B)| \frac{1}{\min (|A|, |B|)} \leq |E(A, B)| \left( \frac{1}{|A|} + \frac{1}{|B|} \right) \leq |E(A, B)| \frac{2}{\min (|A|, |B|)};$$
so $h(X) \leq h'(X) \leq 2h(X)$.

\end{comment}

\begin{comment}[Theorem 1.1.8, PP. 4-5]

Here we give a more detailed version of the construction of bounded concentrators and superconcentrators, as well as improving on the constant $c$ given in the first case. \\

\textit{From bi-expanders to bounded concentrators} \\

Let $r \in \mathbb{Z}_{>1}$, and suppose we have an $(m, k, c)$-bi-expander, where $c \geq \frac{1}{r-1}$. Note that this bound is lower than the one in the book (that is, $\frac{2r^2}{(r-1)(r^2 + 1)}$) for all possible values of $r$. We will construct an $(n, \theta, k', \frac{1}{2})$-bounded concentrator, where $n = m \cdot \frac{r+1}{r} = m + \frac{m}{r}$; $\theta = \frac{r}{r+1}$ (so the output set will have $m$ vertices); and $k' = (k+1) \cdot \frac{r}{r+1}$. \\

Let $I$ be the input set, of size $n$, and divide it into a large part $L$ with $m$ vertices, and a small one $S$ with $\frac{m}{r}$ vertices. Let $O$ be the output set, of size $m$, and divide it into parts $O_1, O_2, \ldots, O_{\frac{m}{r}}$ of size $r$. Connect $L$ and $O$ using the $(m, k, c)$-bi-expander $X$, and connect each $i \in S$ to all of $O_i$, where we label $S = \{ 1, 2, \ldots, \frac{m}{r} \}$. We claim that the resulting graph $G = (V, E)$ is an $(n, \theta, k', \frac{1}{2})$-bounded concentrator. One can easily verify that the numbers of vertices and edges are the right ones. \\

It remains to verify the expanding condition. Let $A \subseteq V$ be a subset with $s := |A| \leq \alpha n = \frac{n}{2} = \frac{m (r+1)}{2r}$. We need to show that $|\partial_G A| \geq s$. If at least $\frac{s}{r}$ of the elements of $A$ are in $S$, then
$$|\partial_G A| \geq |\partial_G (A \cap S)| = |A \cap S| \cdot r \geq \frac{s}{r} \cdot r \geq s,$$
so we are done. Otherwise, at least $s- \frac{s}{r} = s \cdot \frac{r-1}{r}$ are in $L$. So let $A'$ be a subset of $A$ of this size contained in $L$. Note that $$ |A'| = s \cdot \frac{r-1}{r} \leq \frac{m(r+1)}{2r} \cdot \frac{r-1}{r} = \frac{m(r^2-1)}{2r^2} \leq \frac{m}{2}.$$
Therefore the expanding property of the bi-expander $X$ applies to $|A'|$, which implies that:
$$|\partial_G A| \geq |\partial_X A'| \geq (1+c)|A'| = (1+c)\cdot s \cdot \frac{r-1}{r} \geq (1 + \frac{1}{r-1}) \cdot s \cdot \frac{r-1}{r} = s.$$

\textit{From bounded concentrators to superconcentrators} \\

Here the construction is recursive. Suppose that for some $\theta \in [\frac{1}{2}, 1)$ we have constructed a $(\theta n, l_{\theta, n})$-superconcentrator and an $(n, \theta, k_{\theta, n}, \frac{1}{2})$-bounded concentrator. We construct an $(n, l_n)$-superconcentrator, where $l_n = 1 + 2k_{\theta, n} + l_{\theta, n}$. Let $V = I \sqcup I' \sqcup O' \sqcup O$, where $|I| = |O| = n$ and $|I'| = |O'| = \theta n$. Connect $I$ and $I'$ using the $(n, \theta, k_{\theta, n}, \frac{1}{2})$-bounded concentrator, and direct the edges from $I$ to $I'$. Do the same for $O$ and $O'$, directing the edges from $O'$ to $O$. Then connect $I'$ to $O'$ using the $(\theta n, l_{\theta, n})$-superconcentrator. Finally, pair up the vertices of $I$ and $O$ and add an edge connecting each pair. Let $X = (V, E)$ be the resulting directed graph. \\

We claim that $X$ is an $(n, l_n)$-superconcentrator. It is easy to check that the number of vertices and edges is the right one, and that $X$ is acyclic. Now let $A \subseteq I$, $B \subseteq O$ be of size $r$. We can already connect some inputs to some outputs using paths of length one. Then there are at most $\frac{n}{2}$ unmatched inputs (respectively, outputs), which we match to vertices in $I'$ (respectively, $O'$), using the bounded concentrator. This gives us the first and last edge of the remaining paths. The middle part of these paths can be found in the $(\theta n, l_{\theta, n})$-superconcentrator. Therefore, $X$ is a superconcentrator. \\

To start the induction, we can chose a simple concentrator, such as a complete bipartite graph $K_{n, n}$, which is an $(n, n)$-superconcentrator. \\

\textit{Working out the constants} \\

We will start with the hypothesis of the Theorem 1.1.8: let $r \in \mathbb{Z}_{> 1}$, and suppose that there exists a $k \in \mathbb{Z}_{> 0}$ such that for all $n$ we can construct an $(n, k, \frac{1}{r-1})$-expander. Then, by the above, we can construct an $(n \cdot \frac{r+1}{r}, \frac{r}{r+1}, (k+1) \cdot \frac{r}{r+1}, \frac{1}{2})$-bounded concentrator. Note that for this passage we need the weaker condition $c \geq \frac{1}{r-1}$: the one given in the book is too restrictive for this setting. \\

Assume that we can construct an $(n, l)$-superconcentrator. Then this bounded concentrator allows us to construct an $(n \cdot \frac{r+1}{r}, l')$-superconcentrator, where $l' = 2(k+1) \cdot \frac{r}{r+1} + l \cdot \frac{r}{r+1} + 1$. We would like for the density to be independent of $n$. The easiest way to do this is to set $l = l'$, so that the density does not change at the inductive step. Solving the equation yields $l = (2k + 3)r + 1$. Therefore we must set this as the density of the starting graph. For example, we may chose the complete bipartite graph $K_{l, l}$ to start, which is an $(l, l)$-superconcentrator. We conclude that for all $i \geq 0$, we can construct a $(l \cdot (\frac{r+1}{r})^i, l)$-superconcentrator, where $l = (2k + 3)r + 1$.

\end{comment}

\subsection{Existence of expanders}

\begin{comment}[Proposition 1.2.1, P. 6]

Here we explain how König's theorem implies that every $k$-regular bipartite graph is obtained as described in the proof of proposition 1.2.1. The proof shows that most of the graphs constructed this way are expanders, so to generalize this statement and say that most graphs are expanders, one needs this fact. (However, the proof still shows that expanders do exist). \\

The theorem that the author is referring to is probably König's line-coloring theorem: in a bipartite graph $X$, the edge-coloring number of $X$ is $\Delta(X)$, its maximum degree.

Let $X = (V, E)$ be a $k$-regular bipartite graph, with parts $I$ and $O$. Since $|E| = k|I| = k|O|$, we have $|I| = |O| =: n$. Label $I = \{1, \ldots, n\}$. By König's line-coloring theorem, $X$ is $k$-edge-colorable, so pick such a minimal coloring with colors $c_1, c_2, \ldots, c_k$. For each $i = 1, \ldots, k$; and for each $j = 1, \ldots, n$; let $\pi_i(j)$ be the unique neighbour $v$ of $j$ such that $jv$ has color $i$. Then the neighbours of $j$ are exactly $\pi_1(j), \ldots, \pi_k(j)$. To conclude, we need to show that each $\pi_i$ is a permutation of $\{1, \ldots, n\}$. By cardinality, it is enough to show that $\pi_i$ is injective. So suppose that $v = \pi_i(j) = \pi_i(j')$. Then we have edges $jv$ and $j'v$, which touch the common vertex $v$ and are of the same color $i$. Since this is a proper edge-coloring, necessarily $j = j'$.

\end{comment}

\begin{comment}[Proposition 1.2.1, P. 6]
\label{latin}

Notice that here we are not considering graphs, but rather multigraphs. Indeed, when we choose an arbitrary $k$-tuple of permutations, we are not demanding that $\pi_a(i) \neq \pi_b(i)$ whenever $a \neq b$. If for some $a \neq b$ we had $\pi_a(i) = \pi_b(i)$, then $i \pi_a(i) = i \pi_b(i)$ would be a double edge. We could just delete the multiple edges, but then the resulting graph would not be regular. Here we show that the statement is still true if one only considers $k$-tuples of permutations giving rise to ordinary graphs. \\

Without going over the whole proof, we remind how the argument works. We denote by $\beta_n$ the number of "bad" $k$-tuples of permutations, i.e., those which do not give rise to an expander. Then we do some estimates and show that $\beta_n = o((n!)^k)$. Since there are in total $(n!)^k$ $k$-tuples of permutations, this implies that as $n$ goes to infinity, the proportion of "bad" $k$-tuples of permutations goes to 0, and therefore so does the proportion of non-expanders. \\

Now let us restrict ourselves to only the permutations that give rise to graphs without multiple edges. Let $\beta'_n$ be the number of "bad", and $N_n$ the total number of, $k$-tuples of permutations of this type. We need to show that $\beta'_n = o(N_n)$. Clearly $\beta'_n \leq \beta_n$, so it is enough to show that $\beta_n = o(N_n)$.

Given a $k$-tuple of permutations $\pi = (\pi_1, \pi_2, \ldots, \pi_k)$, consider the $k \times n$ matrix $A_\pi = (\pi_i(j))_{i, j}$. The entries of the rows are all different since the permutations are injective. The entries of the columns are all different if and only if the graph arising from $\pi$ does not contain multiple edges. Therefore, $N_n$ is exactly the number of $k \times n$ latin squares, which we denote by $L_{k, n}$. In \cite{latrect}, Erdős and Kaplansky prove the following asymptotic formula: $$L_{k, n} \sim (n!)^k \cdot e^{- \frac{k(k-1)}{2}}; \, \text{when} \, k = O((\log n)^{\frac{3}{2} - \epsilon}).$$ (This is actually the way that it is formulated in \cite{latrect-}). In our case, $k$ is fixed, and not a function of $n$, which means that $N_n$ is proportional to $(n!)^k$. But $\beta_n = o((n!)^k)$, so $\beta_n = o(N_n)$, which is what we wanted to prove.

\end{comment}

\pagebreak

\section{The Banach-Ruziewicz Problem}

\subsection{The Hausdorff-Banach-Tarski paradox}

\begin{comment}[P. 9]

Here we exhibit a free subgroup of $SO(3)$ of rank 2. The group constructed in the book is used later on, so the reader is encouraged to read that construction, together with the comments on it. However, for those who are only interested in the existence of such a group, the following proof is shorter and simpler, and it is all that we need for chapter 2 of the book. We will show that the two following matrices generate a free subgroup of $SO(3)$.

$$
A = 
\left(
\begin{array}{ccc}
1 & 0 & 0 \\
0 & \frac{1}{3} & -\frac{2 \sqrt{2}}{3} \\
0 & \frac{2 \sqrt{2}}{3} & \frac{1}{3}
\end{array}
\right)
; \hspace{1cm}
B = 
\left(
\begin{array}{ccc}
\frac{1}{3} & -\frac{2 \sqrt{2}}{3} & 0 \\
\frac{2 \sqrt{2}}{3} & \frac{1}{3} & 0 \\
0 & 0 & 1
\end{array}
\right).
$$

Proving this was an exercise in Monod's class \cite{Monod}.

\begin{proof}
We need to prove that for any reduced word $w$ of length $k > 0$ on $\{A^{\pm 1}, B^{\pm 1} \}$, we have $w \neq I$, the identity matrix. We may assume that $w$ ends with $B^{\pm 1}$: if this is not the case, then we can replace it with $BwB^{-1}$, which does not change it being the identity or not. We will prove by induction on $k$ that $w (e_1) \neq e_1$, where $e_1 = (1, 0, 0)^T$. \\

We start by showing that if $w$ is of length $k$, then $w(e_1) = 3^{-k} (x, y \sqrt{2}, z)^T$, for some $x, y, z \in \mathbb{Z}$. This is true for $k = 1$, since in this case $w = B^{\pm 1}$ and $B^{\pm 1} (e_1) = 3^{-1}(1, 2 \sqrt{2}, 0)^T$. Then by induction: if $w'(e_1) = 3^{1-k}(x, y\sqrt{2}, z)^T$, we get
$$A^{\pm 1} w' (e_1) = 3^{1-k}(x, y \frac{\sqrt{2}}{3} \mp z \frac{2 \sqrt{2}}{3}, z \frac{1}{3} \pm y \frac{4}{3})^T = 3^{-k}(3x, (y \mp 2z) \sqrt{2}, z \pm 4y)^T;$$
$$B^{\pm 1} w' (e_1) = 3^{1-k}(x \frac{1}{3} \mp y \frac{4}{3}, y \frac{\sqrt{2}}{3} \pm x\frac{2 \sqrt{2}}{3}, z)^T = 3^{-k}(x \mp 4y, (y \pm 4x) \sqrt{2}, 3z)^T.$$

Next, we will prove that $3 \nmid y$. This implies in particular that $y \neq 0$, and so $w(e_1) \neq e_1$. Note that the previous calculation also shows that if the first letter is $A^{\pm 1}$, then $3 | x$, and if it is $B^{\pm 1}$, then $3 | z$. Once again we proceed by induction on $k$, the base case $k = 1$ being already covered.

Suppose that the first letter of $w$ is $C$ and the second one is $A^{\pm 1}$. Then $w(e_1) = C \cdot 3^{1-k}(x, y \sqrt{2}, z)$, where $C$ is the first letter, $3 \nmid y$ (by induction), and $3|x$ (by the above). If $C = B^{\pm 1}$, then the new $y$ will be $y \pm 4x \equiv y \mod 3$, since $3 | x$. We have the same argument if the first letter is $A^{\pm 1}$ and the second is $B^{\pm 1}$.

Suppose instead that $w$ starts with $A^{\pm 2}$. Then $$w(e_1) = A^{\pm 2} 3^{2-k} (x, y \sqrt{2}, z)^T = A^{\pm 1} 3^{1-k}(3x, (y \mp 2z) \sqrt{2}, z \pm 4y)^T,$$
for which the new $y$ is $$(y \mp 2z) \mp 2 (z \pm 4y) \equiv y \mp 2z \mp 2z - 8y \equiv 2y \mp 4z \equiv 2(y \mp 2z) \mod 3 .$$
Since by induction $3 \nmid (y \mp 2z)$, we also get $3 \nmid 2(y \mp 2z)$. Once again, we have the same argument if $w$ starts with $B^{\pm 2}$.
\end{proof}

\end{comment}

\begin{comment}[Theorem 2.1.8, P. 9]
\label{218}

Here we fill in the details of the proof of Jacobi's theorem (theorem 2.1.8). For an analytic proof simpler than the one referenced in the book, see \cite{jacobi}. We also mention a short proof using only elementary arithmetic, in section 2.4 of \cite{Jacobi_odd}. The authors only prove the theorem for $n$ odd, however the general case follows directly from the first two lemmas in that section. \\

By integral quaternions we mean quaternions with integer coordinates. The ring of integral quaternions is noted $H(\mathbb{Z}) := \{a_0 + a_1i + a_2j + a_3k : a_i \in \mathbb{Z} \}$. We will rather work with the ring of Hurwitz integral quaternions $\tilde{H}(\mathbb{Z})$, whose elements are exactly quaternions with coordinates either all integers or all half an odd integer. The norm of an element $\alpha = (a_0 + a_1i + a_2j + a_3k)$ is $N(\alpha) = \alpha \overline{\alpha} = \overline{\alpha}\alpha = a_0^2 + a_1^2 + a_2^2 + a_3^2$. The norm is multiplicative, since it is multiplicative in the division ring of quaternions. We will need the following facts about $\tilde{H}(\mathbb{Z})$:
\begin{enumerate}
\item $\tilde{H}(\mathbb{Z})$ is a ring.
\item The norm of an element of $\tilde{H}(\mathbb{Z})$ is always an integer.
\item $\tilde{H}(\mathbb{Z})$ has 24 units, which are exactly the elements with norm 1: the 8 integral units $\{ \pm 1, \pm i, \pm j, \pm k \}$ and the 16 non-integral ones, which are the ones of the form $\frac{1}{2}((\pm 1) + (\pm i) + (\pm j) + (\pm k))$.
\item $\tilde{H}(\mathbb{Z})$ is a left PID, i.e., every left ideal is principal.
\item If $p$ is an odd prime, then $\tilde{H}(\mathbb{Z}/p\mathbb{Z}) \cong M_2(\mathbb{F}_p)$, the ring of $2 \times 2$ matrices over $\mathbb{F}_p$.
\end{enumerate}
These facts are all proven in appendix \ref{quaternions}. \\

Let $p$ be an odd prime. We want to show that the number of integer solutions to $a_0^2 + a_1^2 + a_2^2 + a_3^2 = p$ is $8(p + 1)$. Notice that there is a bijection between these solutions and the integral quaternions or norm $p$, so we will count those instead. If $x \in \tilde{H}(\mathbb{Z})$, denote by $(x)$ the \textit{left} ideal generated by $x$. If $x$ has norm $p$, then $p = \overline{x}x \in (x)$, so $(p) \subseteq (x)$. Also, since the norm is multiplicative, if $y \in (p)$, then $p^2 = N(p) \, | \, N(y)$, so in particular $x \notin (p)$. Therefore $(p) \subset (x)$. Vice-versa, if $x$ is an element which is not a unit such that $(p) \subset (x)$, then $N(x) \, | \, N(p) = p^2$, so $N(x) = 1, p$ or $p^2$. We exclude $N(x) = 1$ since $x$ is not a unit. Suppose $N(x) = p^2$. Since $p \in (x)$, there exists $y \in \tilde{H}(\mathbb{Z})$ such that $p = yx$. But then $N(p) = N(y)N(x) = N(x)$, so $y$ is a unit and $(p) = (x)$, which contradicts our assumption. So $N(x) = p$.

Therefore the elements of norm $p$ in $\tilde{H}(\mathbb{Z})$ are exactly those who generate a proper left ideal strictly containing $(p)$. For any $x, y \in \tilde{H}(\mathbb{Z})$, we have $(x) = (y)$ if and only if $x = \epsilon y$, where $\epsilon$ is a unit. Therefore there are exactly 24 generators for each left ideal. \\

Now we ask ourselves how many of these generators are integral quaternions. First of all, if $\alpha = \frac{1}{2}(a_0 + a_1i + a_2j + a_3k)$, where all the $a_i$ are odd integers, then we can multiply on the left by the unit $\frac{1}{2}(u_0 + u_1i + u_2j + u_3k)$, where all the $u_i$ are $\pm 1$. The resulting quaternion will have real part $\frac{1}{4}(a_0u_0 - a_1u_1 - a_2u_2 - a_3u_3)$. By choosing the $u_i$ appropriately, since all the $a_i$ are odd, we can make sure that each of the $a_iu_i$ is $1 \mod 4$, and thus get a quaternion with integral real part. But if an element in $\tilde{H}(\mathbb{Z})$ has integral real part, then it is integral. This proves that any left ideal is generated by an integral quaternion.

Let $\alpha = a_0 + a_1i + a_2j + a_3k$ be an integral quaternion. Then multiplying it on the left by a unit quaternion simply permutes the coordinates up to sign, so $\alpha$ has 8 associate integral quaternions. If instead we multiply it by a non-integral unit, the resulting quaternion will have real part $\frac{1}{2}(a_0u_0 - a_1u_1 - a_2u_2 - a_3u_3)$, where $u_i = \pm 1$. But if $\alpha$ has odd norm, since $\pm a_i$ has the same parity as $a_i^2$, this number will be half an odd integer, and so the resulting quaternion will be a non-integral one. In particular this proves that if $N(x) = p$, then $(x)$ has 8 integral and 16 non-integral generators.

We conclude that the number of integral quaternions of norm $p$ is 8 times the number of proper left principal ideals containing $(p)$ properly. Since $\tilde{H}(\mathbb{Z})$ is a left PID, this is the number of all proper left ideals containing $(p)$ properly. \\

So now we want to count the number of these ideals. By the correspondence theorem, and since any quotient of a left PID is a left PID, this is equal to the number of proper non-trivial principal left ideals of $\tilde{H}(\mathbb{Z})/(p) = \tilde{H}(\mathbb{Z}/p\mathbb{Z})$. But $\tilde{H}(\mathbb{Z}/p\mathbb{Z}) \cong M_2(\mathbb{F}_p)$, so we count the proper non-trivial principal left ideals of $M_2(\mathbb{F}_p)$, which are those generated by the non-invertible non-zero elements. There are $|M_2(\mathbb{F}_p)| - |GL_2(\mathbb{F}_p)| - |\{ 0 \}| = p^4 - (p^2 - 1)(p^2 - p) - 1 = (p + 1)(p^2 - 1)$ such elements. The group $GL_2(\mathbb{F}_p)$ acts on them by left multiplication, and two elements lie in the same orbit if and only if they generate the same left ideal. Now we to calculate these size of the orbits. \\

We start by calculating the size of the stabilizer of one of these elements. In general, if $G$ is a group acting on a set $S \subseteq G$ by left multiplication, then for any $s \in S$, the stabilizer of $gsg^{-1}$ is the conjugate by $g$ of the stabilizer of $s$. Therefore, to count the size of the stabilizer of an element, it is enough to calculate the stabilizer of an element conjugated to it. Now for any $A \in M_2(\mathbb{F}_p)$, by the Jordan canonical form, $A$ is conjugate to a matrix of one of the following forms:
$$
\left(
\begin{array}{cc}
\lambda_1 & 0 \\
0 & \lambda_2 \\
\end{array}
\right)
, \,
\left(
\begin{array}{cc}
\lambda & 1 \\
0 & \lambda \\
\end{array}
\right).
$$
The determinant of such a matrix is $\lambda_1 \lambda_2$ (respectively $\lambda^2$) so for the matrix to have determinant 0, we need $\lambda_1$ or $\lambda_2$ to be zero (respectively $\lambda = 0$). Also, by conjugating a diagonal matrix by a permutation matrix, we can switch the diagonal entries. Therefore all non-invertible non-zero elements will be conjugate to a matrix of one of the following forms:
$$
\left(
\begin{array}{cc}
\lambda & 0 \\
0 & 0 \\
\end{array}
\right)
, \,
\left(
\begin{array}{cc}
0 & 1 \\
0 & 0 \\
\end{array}
\right),
$$
where $\lambda \neq 0$.  An easy calculation shows that in both cases the stabilizer is:
$$
\left\{
\left(
\begin{array}{cc}
1 & b \\
0 & d \\
\end{array}
\right)
: d \neq 0
\right\}
$$
which has size $p(p-1)$. \\

Therefore, by the orbit-stabilizer theorem, each orbit has size $|GL_2(\mathbb{F}_p)|/p(p-1) = p^2 - 1$. It follows that there are $(p+1)(p^2 - 1)/(p^2 - 1) = (p+1)$ proper left ideals in $M_2(\mathbb{F}_p)$, so by all the previous results, we conclude that there are $8(p+1)$ integral elements of norm $p$ in $\tilde{H}(\mathbb{Z})$. This concludes the proof.

\end{comment}

\begin{comment}[PP. 9-10]

Here we fill in the details of the paragraph between the proof of theorem 2.1.8 and lemma 2.1.9. The fact that the units of $H(\mathbb{Z})$ are $\{ \pm 1, \pm i, \pm j, \pm k \}$ was already mentioned in comment \ref{218}, and it is proven in appendix \ref{quaternions}. \\

Let $p \equiv 1 \mod 4$ be a prime. $S'$ be the set of integral quaternions of norm p, so $|S'| = 8(p+1)$ by theorem 2.1.8. Since $p \equiv 1 \mod 4$ and a square is either $0$ or $1 \mod 4$, only one of the coordinates of an element of norm $p$ is odd. The units of $H(\mathbb{Z})$ act on $S'$ by left multiplication and each $\alpha \in S'$ has exactly one associate $\epsilon \alpha$, where $\epsilon$ is a unit, with odd positive real part. Indeed, multiplying on the left by a unit amounts to permuting the coordinates up to sign. Multiplying by $i$ sends the second coordinate to the first one, and similarly for $j$ and $k$. Since there is exactly one odd coordinate, there is exactly one units sending it to the first one, and making it positive. Also, since all of the other coordinates will be even, $\epsilon \alpha \equiv 1 \mod 2$. 

Let $S$ be the set of these $(p+1)$ representatives. If $\alpha$ is a representative, then multiplying it by a unit makes the first coordinate either negative or even, so $\alpha$ cannot be associate to $\overline{\alpha}$. Also, if $\alpha \in S$, then $\overline{\alpha} \in S$, so $S = \{ \alpha_1, \overline{\alpha_1}, \ldots, \alpha_s, \overline{\alpha}_s \}$, where $s = \frac{p+1}{2}$.

\end{comment}

\begin{comment}[Lemma 2.1.9, P. 10]

For a discussion on the factorization theory of odd integral quaternions, as well as a proof of the fact that a quaternion is prime if and only if its norm is prime, see \cite{odd_quat}. \\

We fill in the details of the second part of the proof. In the first part it was shown that if $N(\alpha) = p^k$, then there is a unique expression of the form $\alpha = \epsilon p^r R_m(\alpha_1, \ldots, \overline{\alpha}_s)$, where $R_m$ is a reduced word of length $m$ and $\epsilon$ is a unit. Since each $\alpha_i$ has norm $p$, we have $p^k = N(\alpha) = N(p^r)N(R_m(\alpha_1, \ldots, \overline{\alpha}_s)) = p^{2r}p^m$, so $2r + m = k$.

The number of reduced words $R_m(\alpha_1, \ldots, \overline{\alpha}_s)$, when $m \geq 1$, is $(p+1)p^{m-1}$: there are $(p+1)$ choices for the first element, and after that at each step we can choose any element but the conjugate to the previous one. When $m = 0$, then we only have the empty word. Let us denote by $\delta$ the characteristic function of the even integers. Let $l$ be the largest integer strictly smaller than $\frac{k}{2}$, so that $l = \frac{k-2}{2}$ if $k$ is even, and $l = \frac{k-1}{2}$ if $k$ is odd. Then we claim that the total number of reduced words of all possible lengths $m$, with $m + 2r = k$ for some other integer $r$, is:
$$\sum\limits_{0 \leq r \leq l} (p+1)p^{k - 2r - 1} + \delta(k).$$
Indeed, $k - 2r = m$ goes from $k$ to $k - 2l = 2$ if $k$ is even and 1 if $k$ is odd. In the first case, we also have to add 1 since $m$ might be 0, and in the second case $m$ cannot be 0, so in both cases we must add $\delta(k)$. \\

Next, we develop the expression to get:
$$\sum\limits_{0 \leq r \leq l} p^{k - 2r - 1} + \sum\limits_{0 \leq r \leq l} p^{k - 2r} + \delta(k).$$
If $k$ is even, then the first sum goes through the odd powers of $p$ from $p^{k - 1}$ to $p$, and the second one goes through the even powers of $p$ from $p^k$ to $p^2$, while $\delta(k) = p^0$. A similar reasoning holds when $k$ is odd, so this sum is:
$$\sum\limits_{j = 0}^k p^j = \sum\limits_{d | p^k} d.$$
This is the number of possible $R_m$. Once such a word is fixed, $p^r$ is uniquely determined, and then we have 8 choices for the unit $\epsilon$. Therefore the number of expressions of the form $\alpha = \epsilon p^r R_m(\alpha_1, \ldots, \overline{\alpha}_s)$ is 8 times the previous sum, which, by the odd prime power case of theorem 2.1.8, is exactly the number of integral quaternions of norm $p^k$. It follows that each such expression represents a distinct element.

\end{comment}

\begin{comment}[P. 11]

Here we fill in the details of the end of the proof of proposition 2.1.7. \\

As a set, $\Lambda(2)$ is the quotient of $\Lambda'(2)$ by the relation: $\alpha \sim \beta$ if and only if $\pm p^i \alpha = \beta$ for some $i \in \mathbb{Z}$. It is easy to check that the group operation $[\alpha][\beta] = [\alpha \beta]$ is well defined, so $\Lambda(2)$ is a group. For a ring homomorphism $\sigma : H(\mathbb{Z}[\frac{1}{p}]) \to R$ to induce a group homomorphism $\Lambda(2) \to R^*$, it is necessary and sufficient that $x \in \ker \sigma$ for all $x \in \Lambda'(2)$ that are in the equivalence class of the identity. These are exactly the elements of the form $\pm p^i$. In the case of the natural embedding $\sigma : H(\mathbb{Z}[\frac{1}{p}]) \to H(\mathbb{R})$, these elements are sent to scalar matrices. Therefore we get the well-defined group homomorphism $\tilde{\sigma} : \Lambda(2) \to H(\mathbb{R})^* / Z(H(\mathbb{R})^*)$.

Next we need to show that this is an embedding. But this is immediate since if $[x] \in \ker \tilde{\sigma}$, then $\sigma(x)$ is a scalar matrix, which is only the case if $x$ is real, so $[x]$ is the equivalence class of the identity.

\end{comment}

\begin{comment}[Corollary 2.1.14, P. 12]

Using some basic results on amenable groups and amenable actions, and Tarski's theorem (every non-amenable action is paradoxical), we get another nice proof that $S^2$ is $SO(3)$-paradoxical, taken from Monod's class \cite{Monod}. Indeed, the existence of a free subgroup implies that $SO(3)$ is non-amenable. Then we notice that for any point $x \in S^2$, the stabilizer $SO(3)_x$ is the group of rotations whose axis is the vector spanned by $x$, so $SO(3)_x \cong S^1$, the circle group. In particular, all stabilizers for this action are abelian, so amenable. But if an action of a group $G$ on a set $X$ is amenable and all stabilizers are amenable, then $G$ is itself amenable. Therefore, the action of $SO(3)$ on $S^2$ must be non-amenable, hence paradoxical. \\

This argument does not generalize to $S^n$. Indeed, if $n > 2$, then the stabilizer of the action of $SO(n+1)$ on $S^n$ is $SO(n)$, which is non-amenable. Therefore we still need to go through the induction argument to prove corollary 2.1.15, which is what we will do next.

\end{comment}

\begin{comment}[Corollary 2.1.15, P. 15]

Here we fill in the details of the proof of corollary 2.1.15. \\

Let $n > 2$, and let $S^{n-1} = A \sqcup B$ be a paradoxical decomposition. We will use this to provide a paradoxical decomposition of $\hat{S}^n = S^n \, \backslash \, \{(0, \ldots, 0, \pm 1)\}$. Then the same argument of proposition 2.1.13 shows that $S^n$ and $\hat{S}^n$, are equidecomposable, which implies that $S^n$ is paradoxical. \\

Let $\phi : \hat{S}^n \to S^{n-1} : (x_1, \ldots, x_n, x_{n+1}) \mapsto (x_1, \ldots, x_n)_0$, where for all $x \in \mathbb{R}^n \, \backslash \, \{ 0 \}$, we note $x_0 := \frac{x}{||x||}$. For any subset $C \subseteq S^{n-1}$, let $C^* := \phi^{-1}(C) \subseteq \hat{S}^n$. Since $S^{n-1} = A \sqcup B$, we clearly have $\hat{S}^n = A^* \sqcup B^*$. We claim that this is a paradoxical decomposition, for which we need to prove that $\hat{S}^n$ and $A^*$ are equidecomposable, and the same argument works for $B^*$. \\

For any $g \in SO(n)$, define
$g^* =
\left(
\begin{array}{c|c}
g & 0 \\
\hline
0 & 1
\end{array}
\right)
\in SO(n+1).$
Then a simple calculation shows that for all $g \in SO(n)$, we have $g \phi(x) = \phi(g^*x)$. It follows that for all $C \subseteq S^{n-1}$, we have $(gC)^* = g^*C^*$. Now since $S^{n-1}$ and $A$ are equidecomposable, there exist $g_1, \ldots, g_k \in SO(n)$ and a partition $A = \bigsqcup\limits_{i = 0}^k A_i$ such that $S^{n-1} = \bigsqcup\limits_{i = 0}^k g_i A_i$. Then we conclude:
$$\hat{S}^n = (S^{n-1})^* = \bigsqcup\limits_{i = 0}^k (g_i A_i)^* = \bigsqcup\limits_{i = 0}^k g^*_i A_i^* \sim \bigsqcup\limits_{i = 0}^k A_i^* = A^*.$$

\end{comment}

\begin{comment}[Theorem 2.1.7, P. 13]

Here the author is probably referring to the infinite version of Hall's marriage lemma (see theorem \ref{Hall} in the appendix). However, the marriage lemma only ensures that there is a matching covering $A$, but since $A$ and $B$ are both infinite this does not imply that there is a matching covering $B$. A simple counterexample is given by the bipartite graph with vertex set $\mathbb{N} \sqcup \mathbb{N}$ where $n$ is connected to $2n$ and $2n + 1$.

Luckily, the argument still shows that $A \lesssim B$, and since it works symmetrically as well, we can conclude by proposition 2.1.2. Indeed, the counterexample given above is not a regular graph. Still, this is not quite a direct application of Hall's marriage lemma.

\end{comment}

\subsection{Invariant Measures}

\begin{comment}[Theorem 2.2.2, P. 14]

For a proof of Tarski's theorem, see section \ref{Tarski} of the appendix.

\end{comment}

\begin{comment}[Definition 2.2.3, P.14]
\label{(F)}

Simply adapting the definition of amenability to the case of discrete groups, we get: \\
(F): for all finite subsets $K \subseteq_f G$ and for all $\epsilon > 0$, there exists some $U \subseteq_f G$ such that for all $x \in K$: $|xU \Delta U| < \epsilon |U|$. 

This is equivalent to the characterization of amenability of discrete groups given in the book, namely: \\
(F'): for all $K \subseteq_f G$ and for all $\epsilon > 0$, there exists some $U \subseteq_f G$ such that $|KU \Delta U| < \epsilon |U|$. \\

For the proof of this fact, see lemma \ref{(F)app} in the appendix, and replace $X$ by $G$ throughout.

\end{comment}

\begin{comment}[P. 14]

Here we prove that if a finitely generated discrete group $G$ satisfies the Følner condition for a finite generating subset $K_0$, then it satisfies it for every finite subset $K$.

\begin{proof}
First suppose that $K_0$ is symmetric: that is, $K_0 = K_0^{-1}$. Let $\epsilon > 0$. We want to find a finite set $U$ such that $|KU \Delta U| < \epsilon |U|$. Since $K$ is finite, and $K_0$ is symmetric, there exists some $n \in \mathbb{Z}_{>0}$ such that $K \subseteq K_0^n$. Let $U$ be a finite set such that $|K_0 U \Delta U| < \delta |U|$, for some $\delta > 0$. Now it is easy to see that $A \Delta B \subseteq (A \Delta C) \cup (C \Delta B)$ in any set. Inductively, this shows that $$K_0^n U \Delta U \subseteq \bigcup\limits_{i=0}^{n-1} K_0^{i+1} U \Delta K_0^i U.$$
Also, whenever $S \subseteq G$, clearly $|SU| \leq |S||U|$ and $SA \Delta SB \subseteq S(A \Delta B)$.
Therefore: $$|KU \Delta U| \leq |K_0^n \Delta U| \leq \sum\limits_{i=0}^{n-1} |K_0^{i+1} U \Delta K_0^i U| \leq \sum\limits_{i=0}^{n-1} |K_0^i||K_0 U \Delta U| \leq n|K_0^n||K_0U \Delta U| < n|K_0^n| \delta |U|.$$
Choosing $\delta = \frac{\epsilon}{n|K_0^n|}$ yields the desired result. \\

Now for the general case, let $\epsilon > 0$. We want to find a finite set $U$ such that $|xU \Delta U| < \epsilon |U|$, where $x \in K_0 \cup K_0^{-1}$ (here we are using the alternative definition given in the comment \ref{(F)}). If we do this, then we have proven the condition for the symmetric generating set $K_0 \cup K_0^{-1}$, which allows us to conclude by the previous part.

So let $\epsilon > 0$, and $U$ such that $|xU \Delta U| < \epsilon |U|$ for all $x \in K_0$. Then if $x \in K_0$:
$$|x^{-1} U \Delta U| = |x^{-1} U \Delta x^{-1} x U| = |U \Delta xU| < \epsilon |U|.$$
\end{proof}

\end{comment}

\begin{comment}[P.14]
\label{amenable_discrete}

Here the author states three facts about discrete amenable groups, namely:
\begin{enumerate}
\item if every finitely generated subgroup of a group $G$ is amenable, then G is amenable;
\item finitely generated abelian groups are amenable;
\item extensions of amenable by amenable groups are amenable.
\end{enumerate}
These are all easy consequences of the fixed-point theorem: a discrete group $G$ is amenable if and only if any action of $G$ on a convex compact set $K$ by affine homeomorphisms has a fixed point. However, here we are starting with the Følner property as a definition of amenability, so we will prove these facts starting from there. For the first two, this is even easier than with the fixed-point theorem. The third is trickier than it seems, and its proof was given in an answer on StackExchange \cite{stack_ext}.

\begin{proof}
1. Let $K \subseteq_f G$, and let $\epsilon > 0$. Then $\langle K \rangle$ is finitely generated, so amenable. Therefore there exists some $U \subseteq_f \langle K \rangle \subseteq G$ such that $|KU \Delta U| < \epsilon |U|$. This is obviously true also in $G$, so we conclude. \\

2. We start by proving that $\mathbb{Z}$ is amenable. To do this, we can use the previous comment, for the generating set $\{1 \}$. Let $\epsilon > 0$. Then $|(\{ 1 \} + [1, n]) \Delta [1, n]| = |[2, n+1] \Delta [1, n]| = |\{ n+1 \}| = 1 = \frac{1}{n} \cdot |[1, n]|$. Choosing $n > \frac{1}{\epsilon}$ we conclude.

So $\mathbb{Z}$ is amenable. Trivially, all finite groups are amenable. Since every finitely generated abelian group is a direct product of copies of $\mathbb{Z}$ and finite groups, we conclude by 3. \\

3. We use definition (F) to prove this (see comment \ref{(F)}). Let $G$ be a group, with an amenable normal subgroup $N$, such that $Q := G/N$ is amenable. We denote the canonical projection $G \to G/N : x \mapsto \overline{x}$. Let $K \subseteq_f G$, and $\epsilon > 0$.

Let $R$ be a finite set of representatives of $Q$ (so $R \to \overline{R}$ is injective) such that $|\overline{xR} \Delta \overline{R}| < \delta |\overline{R}| = \delta |R|$ for some $\delta > 0$. The existence of $R$ is guaranteed by the amenability of $Q$. Now if $x \in K$ and $r \in R$ are such that $\overline{xr} \in \overline{R}$, then write $xr = \rho(x, r)\nu(x, r)$, where $\rho(x, r) \in R$ and $\nu(x, r) \in N$. Note that $\overline{xr} = \overline{\rho(x, r)}$, so $\rho(x, r)$ is well-defined and thus $\nu(x, r)$ is too. Define $L := \{ \nu(x, r) \, : \, x \in K, \, r \in R, \, \overline{xr} \in \overline{R} \} \subseteq_f N$, and let $M \subseteq_f N$ be such that $|lM \Delta M| < \delta |M|$ for all $l \in L$ (and the same $\delta$ as before). The existence of $M$ is guaranteed by the amenability of $N$. We will show that $RM \subseteq_f G$ is the appropriate set. \\

Fix $x \in K$. Define $R_x := \{ r \in R \, : \, \overline{xr} \in \overline{R} \} \subseteq R$ and $R'_x := \{ \rho(x, r) \, : \, r \in R_x \} \subseteq R$. Now if $r \in R_x$, then
$$|xr M \Delta \rho(x, r) M| = |\rho(x, r) \nu(x, r) M \Delta \rho(x, r) M | = |\nu(x, r) M \Delta M| < \delta |M|.$$
Thus we get inequality (1):
$$|x R_x M \Delta R'_x M| \leq | \bigcup\limits_{r \in R_x} (x r M \Delta \rho(x, r) M)| < |R_x| \cdot \delta |M| = \delta |R_x M| \leq \delta |RM|,$$
where the next-to-last equality follows from the fact that $R$ is a set of representatives of $Q = G/N$ and $M \subseteq N$. Moreover, notice that
$$|R_x' \Delta R| = |\overline{R_x'} \Delta \overline{R}| = |\overline{x R_x} \Delta \overline{R}| \leq |\overline{xR} \Delta \overline{R}| < \delta |R|.$$
Since $|R_x| = |R_x'|$ it follows that
$$|R \Delta R_x| = |R \, \backslash \, R_x| = |R \, \backslash \, R'_x| = |R \Delta R'_x| < \delta|R|$$
as well. Then we get inequalities (2):
$$|R M \Delta R_x' M| \leq |R \Delta R'_x||M| < \delta |R||M| = \delta |RM|$$
and (3):
$$|x R M \Delta x R_x M| \leq |R \Delta R_x||M| < \delta |R||M| = \delta |RM|.$$
Finally, combing these three inequalities:
$$|x R M \, \backslash \, R M| \leq |xRM \, \backslash \, R'_x M| \leq |x R M \, \backslash \, x R_x M| + |x R_x M \, \backslash \, R'_x M | < 2 \delta |RM|$$
and
$$|R M \, \backslash \, x R M| \leq |R M \, \backslash \, x R_x M| \leq |R M \, \backslash R'_x M| + |R'_x M \, \backslash \, x R_x M| < 2 \delta |RM|.$$
We conclude that $|x R M \Delta R M| < 4 \delta |R M|$. Choosing $\delta = \frac{\epsilon}{4}$, we conclude.
\end{proof}

\end{comment}

\begin{comment}[Definition 2.2.4, P. 14]
\label{M'_equiv}

Here we discuss definition 2.2.4. The first important fact to notice is that the elements of $L^\infty(G)$ are equivalence classes of functions. This means that any $m : L^\infty(G) \to \mathbb{R}$ can only be well-defined if for all $f, g : G \to \mathbb{R}$ such that $|f - g| \equiv 0$ almost everywhere, we have $m(f) = m(g)$. This implies in particular that if $A \subseteq G$ is a null set, then $m(\chi_A) = 0$. This is why there is a correspondence between means and absolutely continuous measures, as described at P. 17. \\

Now we prove that we can add another condition to the definition of an invariant mean without affecting it. This is used later in at the end of the proof of proposition 2.2.5 (see comment \ref{225}). The discrete version of this result was an exercise in Monod's class \cite{Monod}. \\

Let $X$ be a topological space with a measure $\lambda$, and let $L^\infty(X)$ be the set of equivalence classes of essentially bounded real-valued measurable functions. Let $m : L^\infty(X) \to \mathbb{R}$ be a linear functional. Let $||m||_\infty := \sup\limits_{||f||_\infty = 1} m(f)$. Then any two of the following three conditions implies the third:
\begin{enumerate}[(a)]
\item $m(f) \geq 0$ whenever $f \geq 0$.
\item $m(\chi_X) = 1$.
\item $||m||_\infty = 1$.
\end{enumerate}

Before starting with the proof of the equivalences, we  make some remarks. \\

1. Given $f \in L^\infty(X)$ with $f \leq 1$ everywhere, we can always write $f + g = \chi_X$, with $g \geq 0$. If furthermore $f \geq 0$, then $g \leq 1$. This is done by setting $g = \chi_X - f$. \\

2. (c) implies that for all $f \in L^\infty(X)$, we have $m(f) \leq ||f||_\infty$. Indeed, if $||f||_\infty \neq 0$, then by letting $f' = \frac{f}{||f||_\infty}$ we have $m(f) = ||f||_\infty m(f') \leq ||f||_\infty$. If instead $||f||_\infty = 0$, then:
$$m(f) = m(f + \chi_X) + m(- \chi_X) \leq ||f + \chi_X||_\infty + ||- \chi_X||_\infty = 1-1 = 0.$$

3. We will prove that (b) + (c) $\Rightarrow$ (a). However, for this implication we can replace (c) by "$||m||_\infty \leq 1$", or "$m(f) \leq ||f||_\infty$ for all $f \in L^\infty(X)$". Indeed, (b) implies that $||m||_\infty \geq 1$, so together with that, we get (c). This is the way it is used in proposition 2.2.5: see comment \ref{225}.

\begin{proof}
(a) + (b) $\Rightarrow$ (c). (b) readily implies $||m||_\infty \geq 1$. Let $f \in L^\infty(X)$ such that $||f||_\infty = 1$, we must show that $m(f) \leq 1$. Up to equivalence, we may assume that $f \leq 1$ everywhere. Write $f + g = \chi_X$ with $g \geq 0$ as in the first remark. Then $1 = m(\chi_X) = m(f + g) = m(f) + m(g) \geq m(f)$, since $m(g) \geq 0$ by (a). \\

(b) + (c) $\Rightarrow$ (a). Let $f \geq 0$ (so in particular $||f||_\infty \geq 0$). Up to normalizing we may assume that $|| f ||_\infty \leq 1$, and up to equivalence we may assume that $f \leq 1$ everywhere. Suppose by contradiction that $m(f) < 0$. Write $f + g = \chi_X$ with $0 \leq g \leq 1$ as in remark 1. Then:
$$1 = m(\chi_X) = m(f + g) = m(f) + m(g) <  m(g) \leq || g ||_\infty \leq 1,$$
which is absurd. \\

(a) + (c) $\Rightarrow$ (b). Let $||f||_\infty \leq 1$. Up to equivalence, we may assume that $f \leq 1$ everywhere. Write $f + g = \chi_X$ with $g \geq 0$ as in remark 1. Then
$$m(\chi_X) = m(f + g) = m(f) + m(g) \geq m(f).$$
This being true for all such $f$, we conclude that $||m||_\infty \leq m(\chi_X)$. But by definition $m(\chi_X) \leq ||m||_\infty$. So $m(\chi_X) = ||m||_\infty = 1$.
\end{proof}

\end{comment}

\begin{comment}[Remark 2.2.7, P. 16]
\label{227}

Here we treat remark 2.2.7.

We rewrite the statement: "Let $G = S^1$ as a discrete group, and $L^\infty(G)$ the space of bounded functions $G \to \mathbb{R}$. Let $H$ be defined as before. Then if $h \in H$ and $A$ is a $G_\delta$-dense subset, or a subset of Lebesgue measure 1, of $G$, we have $\sup\limits_{x \in A} h(x) \geq 0$." This is the way it is used in 2.2.10. \\

Now we fill in the details of the proof. We start by showing that in $S^1$, a finite intersection of $G_\delta$ sets, or of sets of measure 1, is non-empty, in fact dense. For the first case, if $A$ and $B$ are $G_\delta$-dense sets, then they can be written as a countable intersection of open sets. By density, each of these open sets is dense, so $A \cap B$ is again a countable intersection of open dense sets, which is dense by Baire's theorem ($S^1$ is a compact, so complete, metric space). Therefore the intersection of two $G_\delta$-dense sets is again a $G_\delta$-dense set, and by induction we get the result. 

For the second case, if $\lambda(A) = \lambda(B) = 1$, then $\lambda((A \cap B)^c) = \lambda(A^c \cup B^c) \leq \lambda(A^c) + \lambda(B^c) = 0$, so $\lambda(A \cap B) = 1$ and we conclude by induction once again. Also, a set of measure 1 is dense, since its complement has measure 0 so it must have empty interior. \\

Let $F \subset_f S^1$, and let $A$ be a $G_\delta$-dense set or a set of measure 1. Then we claim that there exists some $x \in S^1$ such that $Fx \subset A$. Indeed, for all $f \in F$, define $A_f := \{ x \in S^1 : fx \in A \}$. We want to show that $\bigcap\limits_{f \in F} A_f \neq \emptyset$. But $A_f = f^{-1}A$, and since $f^{-1}$ acts as a rotation, $A_f$ is still a $G_\delta$-dense set or a set of measure 1. So we conclude by the previous paragraph.

In particular, if $U, K \subset_f S^1$, then there exists $x \in S^1$ such that $Ux$ and $K^{-1}Ux$ are in $A$, because $Ux \cup K^{-1}Ux = (U \cup K^{-1}U)x$ and $U \cup K^{-1}U \subset_f X$. \\

Now we repeat the proof of 2.2.6 with the appropriate changes, in order to prove (this version of the statement of) 2.2.7. Here we treat $S^1$ as a discrete group, so we consider bounded functions instead of essentially bounded ones. Note that the following argument works also when $K^{-1}Ux$ is not a subset of $A$, which is what the author suggests. This argument shows that we have the same conclusion when considering functions which are only bounded on $A$. This is a nice fact, but it is not needed for 2.2.10. \\

Let $h, K$ be as in the proof of 2.2.6. Let $U \subset_f G$ be such that $|k_i^{-1} U \Delta U| < \epsilon |U|$ for all $k_i \in K$. Let $T(x) = \sum\limits_{u \in U} h(ux)$. Suppose that $\sup\limits_{x \in A} h(x) = -\delta < 0$. Choose $x$ such that $(U \cup K^{-1}U)x \subset A$, so $|T(x)| \geq \delta |U|$. Now:
$$T(x) = \sum\limits_{i = 1}^n \sum\limits_{u \in U} (f_i(k_i^{-1}ux - f_i(ux)) = \sum\limits_{i = 1}^n \left( \sum\limits_{u \in k_i^{-1}U} f_i(ux) - \sum\limits_{u \in U} f_i(ux) \right).$$
Then
$$\delta |U| \leq |T(x)| \leq n \cdot \epsilon \cdot |U| \cdot \max \{ \sup\limits_{x \in A} |f_i(x)| \}.$$
and we conclude the same way.

\end{comment}

\begin{comment}[Proposition 2.2.5, P. 16]
\label{225}

Here we fill in two details of the end of the proof of Proposition 2.2.5. \\

To write $Y = H \oplus \mathbb{R}\chi_G$ and thus be able to define the linear functional $\nu$, we need to check that the if $h \in H$ is constant, then it is 0. This is because by 2.2.6 we have $||\pm h||_\infty \geq 0$, so if $h \equiv \alpha$, then $\pm \alpha \geq 0$, so $\alpha = 0$. \\ 

Having defined $m$ as a linear extension of $\nu$, we immediately have that $m$ is left-invariant and that $m(\chi_1) = 1$. However, we are missing positivity. This follows from comment \ref{M'_equiv}.

\end{comment}

\begin{comment}[Proposition 2.2.9, P. 16]

For a detailed proof of the uniqueness of the Lebesgue measure, see appendix \ref{Lebesgue}.

\end{comment}

\begin{comment}[Proposition 2.2.10, P. 17]
\label{dense_0}

In the proof of 2.2.10, the author is using 2.2.7 on the $G_\delta$-dense set $A$ and on the set $A^c$ of measure 1. We show that a $G_\delta$-dense subset of $S^1$ of measure 0 exists. \\

Let $D = \{x_1, x_2, \ldots \}$ be a dense countable set of $S^1$. Such a $D$ exists, since $S^1$ is a compact metric space, so it is separable. (For an explicit $D$, we may take the angles which are rational multiples of $\pi$). Let $\epsilon > 0$. Let $B_n^\epsilon$ be a ball centered at $x_n$ of area $2^{-n} \epsilon$, and let $U^\epsilon = \bigcup\limits_{n \geq 1} B_n^\epsilon$. Then $U^\epsilon$ is an open dense subset of $S^1$ of measure at most $\epsilon$. This actually proves the existence of the set $A$ chosen in 2.2.11 (see the next comment).

Next, note that if $\epsilon \leq \delta$, then $U^\epsilon \subseteq U^\delta$. Therefore we can obtain the $G_\delta$-dense set of measure 0 as $A = \bigcap\limits_{n \geq 1} U^{\frac{1}{n}}$. This is still dense by Baire's category theorem, since $S^1$ is a compact (so complete) metric space.

\end{comment}

\begin{comment}[Proposition 2.2.11, P. 17]

Here we discuss the proof of proposition 2.2.11. For a proof that there exists an open dense subset of measure $\lambda(A) < 1$, see comment \ref{dense_0}. \\

Note that here we are not talking about $S^1$ as a discrete group anymore, but rather a group with its usual topology. However, in $L^\infty(G)$, every function is equivalent to a bounded function. Therefore we can apply the statement of remark 2.2.7 to get: "If $h \in H$ and $A$ is a $G_\delta$-dense subset, or a subset of Lebesgue measure 1, of $G$, then $\text{ess} \sup\limits_{x \in A} h(x) \geq 0$". This is good enough for our purposes. \\

As in the proof of 2.2.10, in order to get a direct sum $H \oplus \mathbb{R}\chi_G \oplus \mathbb{R}\chi_B$, we need to use the result of 2.2.7 on both $B$ and its complement $A$. Now the result of 2.2.7 still applies to $A$, since $A$ is an open dense set, so in particular a $G_\delta$-dense set. However, it does not necessarily apply to $B$ (at least, I do not see why it should).

We can still prove the result of the theorem by choosing $A$ a $G_\delta$-dense set of measure 0, just as in 2.2.10. Then it follows the same way that $\chi_B \notin H \oplus \mathbb{R}\chi_G$. Therefore we can define the same $\nu$ and get a linear functional $I$ on $L^\infty(G)$ satisfying $I(\chi_B) = 0 \neq m(B) = 1$.

\end{comment}

\pagebreak

\section{Kazhdan Property $(T)$ and its Applications}

Let $G$ be a locally compact group, $K \subseteq_c G$ a compact subset, $\epsilon > 0$ and $(H, \rho)$ a unitary representation on a Hilbert space $H$. A vector $v \in H$ of norm 1 is \textbf{$(\epsilon, K)$-invariant} if $||\rho(k) v - v|| < \epsilon$ for all $k \in K$. A unitary representation has almost invariant vectors if it has $(\epsilon, K)$-invariant vectors for all $K \subseteq_c G$ and for all $\epsilon > 0$. (The "almost" here refers to the fact that the representation is "close" to having an invariant vector, not that it has some almost-invariant vector: indeed it is crucial in the definition that the representation has $(\epsilon, K)$-invariant vectors for \textit{any} pair $(\epsilon, K)$). \\

Throughout the comments, we will frequently use the Fell topology on $\tilde{G}$, instead of just $\hat{G}$, which is defined exactly the same way, but with all unitary representations on Hilbert spaces, rather than just the irreducible ones. (Actually, there is a subtlety to be addressed in order to define the topology on $\tilde{G}$, which is discussed in comment \ref{Fell_set}). \\

We denote $W(\rho, K, \epsilon; v_1, \ldots, v_n) = \bigcap\limits_{i = 1}^n W(\rho, K, \epsilon; v_i)$, so that the sets of this form are a basis of neighbourhoods of elements of $(H(\rho), \rho) \in \tilde{G}$. More explicitely, $$W(\rho, K, \epsilon; v_1, \ldots, v_n) = \{(H', \sigma) : \exists \, w_1, \ldots, w_n \in H' \text{ such that } ||w_i|| = 1 \text{ and}$$ $$| \langle v_i, \rho(g) v_i \rangle - \langle w_i, \sigma(g) w_i \rangle | < \epsilon \text{ for all } i = 1, \ldots, n \text{ and all } g \in K\}.$$
Note that in the case of $\rho_0$, we can just write $W(\rho_0, K, \epsilon) = \{ (H', \sigma) : \exists v \in H' \text{ such that } ||v|| = 1 \text{ and } |\langle v, \rho(g) v \rangle - 1| < \epsilon \}$, since this equals $W(\rho_0, K, \epsilon; v_1, \ldots, v_n)$ for any $v_i \in H(\rho_0) = \mathbb{C}$. \\

Recall that if $(H_i)_{i \in I}$ is a family of Hilbert spaces, then the Hilbert direct sum is the Hilbert space $H = \oplus_{i \in I} H_i$ defined as follows. As a set, $H$ is the set of $(v_i)_{i \in I}$ such that $\sum_{i \in I} ||v_i||_{H_i}^2 < \infty$. It is made into a vector space via coordinate-wise addition and scalar multiplication. The finiteness condition above allows to define a scalar product $\langle v, w \rangle = \sum_{i \in I} \langle v_i, w_i \rangle_{H_i}$. The fact that the $H_i$ are complete then implies that $H$ is complete.

We can use this construction to define a direct sum of unitary representations. If $G$ is a group, and $(\rho_i)_{i \in I}$ is a family of unitary representations on Hilbert spaces $(H_i)_{i \in I}$, then we define $\oplus_{i \in I} \rho_i$ to be the unitary representation of $G$ on $\oplus_{i \in I} H_i$ defined by coordinate-wise action.

\subsection{Kazhdan property $(T)$ for semi-simple groups}

\begin{comment}[Definition 3.1.1, P. 19]
\label{strong_cont}

Here we discuss more in depth the requirement that the unitary representations considered be continuous. As it is stated in the book, one might think that we are asking for the map $G \to U(H)$ to be continuous. But then it is not clear what is the topology on $U(H)$. The right requirement is that the representation $\pi$ has to be \textbf{strongly continuous}, meaning that for all $v \in H$, the map $G \to H : g \mapsto \pi(g)v$ is continuous with respect to the norm topology. Equivalently, if a net $(g_\alpha) \subseteq G$ converges to $g \in G$, then for all $v \in H$, $||\pi(g_\alpha)v - \pi(g)|| \to 0$ in $\mathbb{R}$. One immediate consequence of strong continuity, is that if $g_\alpha \to g$ and $v \in H$ is fixed by the $g_\alpha$, then $v$ is also fixed by $g$. In other words, all stabilizers are closed, so if $\Gamma \leq G$ fixes $v \in H$, then $\overline{\Gamma}$ also fixes $v$. In particular, if $\Gamma$ is a dense subgroup of $G$ admitting a non-zero invariant vector, then $G$ also admits the same non-zero invariant vector. This is used in many instances throughout the book: see for example lemma 3.1.13 (comment \ref{3113}), theorem 3.4.2 (comment \ref{342}), example 4.3.3 E (comment \ref{433}). \\

One equivalent condition for strong continuity is the continuity of the map $ev : G \times H \to H : (g, v) \mapsto \pi(g)v$. If $ev$ is continuous, then $ev|_{G \times \{ v \} }$ is continuous for all $v \in H$, so $\pi$ is strongly continuous. Conversely, suppose that $\pi$ is strongly continuous. Fix $(g_0, v_0) \in G \times H$ and $\epsilon > 0$. We want to show that there exists a neighbourhood $U$ of the identity in $G$ and a $\delta > 0$ such that for all $x \in U$ and for all $u \in H$ with $||u|| < \delta$, we have $||\pi(xg_0)(v_0 + u) - \pi(g_0)v_0|| < \epsilon$. Let $w_0 := \pi(g_0)v_0$. Then:

$$||\pi(xg_0)(v_0 + u) - \pi(g_0)v_0|| = ||\pi(x)\pi(g_0)v_0 - \pi(g_0)(v_0) + \pi(x)\pi(g_0)u|| \leq $$
$$ \leq ||\pi(x)w_0 - w_0|| + ||u|| < ||\pi(x)w_0 - w_0|| + \delta.$$
Since $\pi$ is strongly continuous, $x \mapsto \pi(x)w_0$ is continuous, so we can choose $U$ in such a way that $||\pi(x)w_0 - w_0|| < \frac{\epsilon}{2}$ for all $x \in U$. Then choosing $\delta = \frac{\epsilon}{2}$, we conclude.

\end{comment}

\begin{comment}[Definition 3.1.1, P. 19]
\label{Fell_set}

The arguments in this comment come from a discussion on a StackExchange post \cite{stack_set}. \\

There is a problem in the definition of the Fell topology on $\tilde{G}$: if $G$ is any group, then the collection $\tilde{G}$ of equivalence classes of unitary representations on Hilbert spaces is not a set. This is why when talking about the Fell topology, we always restrict to some subset of unitary representations. In F.2 of \cite{BHV}, the authors mention that usually we choose to restrict to unitary representations on Hilbert spaces of dimension bounded by some cardinal, or to irreducible unitary representations. \\

We prove that $\tilde{G}$ is not a set. Indeed, let $\mathcal{C}$ be the collection of all cardinals and let $\rho$ be any unitary representation of $G$ on a Hilbert space. Then for any $\kappa \in \mathcal{C}$, we can define $\rho_\kappa$ to be the Hilbert direct sum of $\kappa$ copies of $\rho$. Then $\rho_\kappa$ is a unitary representation of $G$ on a Hilbert space which has dimension at least $\kappa \cdot \operatorname{dim}(\rho)$. This shows that $\tilde{G}$ contains representations whose underlying vector spaces have bases of arbitrarily large cardinality. In other words, we can inject a collection of arbitrarily large cardinals into $\tilde{G}$. A collection of arbitrarily large cardinals is not a set, since otherwise it would have cardinality larger than any cardinal ($\mathcal{C}$ is totally ordered by inclusion). So $\tilde{G}$ is not a set either. \\

Bounding the dimension works. Indeed, since up to isomorphism there is only one vector space of each dimension, the collection of isomorphism classes of vector spaces of dimension bounded by $\kappa$ is a set. Then fixing a vector space $V$, the collection of all inner products on $V$ is contained in the collection of all maps $V \times V \to \mathbb{C}$, which again is a set. So the collection of all inner product spaces of dimension bounded by $\kappa$, up to unitary isomorphism, is a set, and it follows that the collection of all Hilbert spaces of dimension bounded by $\kappa$ is a set. 

Denote by $\tilde{G}_\kappa$ the set of unitary representations of dimension bounded by $\kappa$. Then we can rephrase all general topological statements about $\tilde{G}$ in terms of $\tilde{G}_\kappa$ with the Fell topology, starting with "for any cardinal $\kappa$". This can be done for all statements which preserve a bound on the cardinality. For instance, let $G$ be a group, $H \leq G$ a subgroup. Then proposition 3.1.8 becomes: for any cardinal $\kappa$, the restriction map $\tilde{G}_\kappa \to \tilde{H}_\kappa$ is continuous, and the induction map $\tilde{H}_\kappa \to \tilde{G}_{\kappa [G:H]}$ is continuous. We will not bother rephrasing all statements in this chapter this way, but the reader should keep in mind how to interpret them in this context, so that the topology is well-defined. \\

Restricting to irreducible representations is a special case of this, since if $\rho$ is an irreducible representation of $G$ on some vector space $V$, then the dimension of $V$ is bounded by the cardinality of $\mathbb{C}[G]$. Indeed, let $\rho$ be an irreducible representation of $G$ on the vector space $V$. Then $V$ becomes an irreducible $\mathbb{C}[G]$-module. So if $v \in V$ is any non-zero vector, the unique map $\mathbb{C}[G] \to V$ sending $1 \to v$ is a $\mathbb{C}[G]$-module epimorphism. So $V$ is isomorphic to a quotient of $\mathbb{C}[G]$, which implies that its cardinality, and thus its dimension, is bounded by the cardinality of $\mathbb{C}[G]$. More simply, we can say that $\mathbb{C}[G]$ is a set, so the collection of quotients of $\mathbb{C}[G]$ is also a set. Note that this shows that the collection of all irreducible representations is a set, including the non-unitary ones. \\

Another special case of bounding the dimension by some cardinal is considering only separable Hilbert spaces. Indeed, a Hilbert space is separable if and only if it admits a countable orthonormal basis. For a proof, see \cite[Proposition 4]{Hilb_sep}. Notice that in definition 3.1.1, Lubotzky puts the word separable in parentheses when referring to the group $G$ and the Hilbert space $H$. Indeed, if $G$ is separable, then the equivalent formulations of property $(T)$ may be restated in terms of separable Hilbert spaces. This implies that if we are dealing with property $(T)$ for separable groups, we can consider the Fell topology restricted to unitary representations on separable Hilbert spaces without losing any generality. See comment \ref{313} for more details.
\end{comment}

\begin{comment}[Definition 3.1.1, P. 19]
\label{311}

More explicitely, the definition of weak containment is the following. Let $\rho, \sigma \in \tilde{G}$. To every $v \in H(\rho)$ of norm one, we associate a coefficient, which is the function $G \to \mathbb{R} : g \mapsto \langle \rho(g)v, v \rangle$. Then $\rho \propto \sigma$ if and only if every coefficient of $\rho$ is a limit, uniformly on compact sets, of coefficients of $\sigma$. That is, for every $v \in H(\rho)$ there exists a sequence $(v_n)_{n \geq 1} \subset H(\sigma)$ of vectors of norm one such that for all $g \in G$ the sequence $\langle \sigma(g)v_n, v_n \rangle$ converges to $\langle \rho(g)v, v \rangle$ in $\mathbb{R}$. Moreover, the convergence is uniform when we restrict the coefficients to functions $K \to \mathbb{R}$, for any $K \subseteq_c G$. \\

At a first glance, it is clear that there is a relationship between weak containment and the neighbourhoods of the Fell topology on $\tilde{G}$. We make this precise: let $\rho, \sigma \in \tilde{G}$. Then $\rho \propto \sigma$ if and only if $\sigma \in U$ for any neighbourhood $U$ of $\rho$ in $\tilde{G}$. Note that the second statement is equivalent to: $\sigma \in W(\rho, K, \epsilon; v)$ for all $\epsilon > 0, \, K \subseteq_c G, \, v \in H(\rho)$. Indeed, finite intersections of sets of this form are a neighbourhood basis of $\rho$. \\

Now we prove this equivalence. Fixing $K \subseteq_c G, \, v \in H(\rho)$, it is enough to prove that $(a):$ the coefficient $(K \to \mathbb{R} : g \mapsto \langle \rho(g)v, v \rangle)$ is a uniform limit of coefficients $(K \to \mathbb{R} : g \mapsto \langle \sigma(g)v_n, v_n \rangle)_{n \geq 1}$; if and only if $(b): \, \sigma \in W(\rho, K, \epsilon; v)$ for any $\epsilon > 0$. Supposing $(a)$, fix $\epsilon > 0$ and pick $N \in \mathbb{Z}_{\geq 1}$ such that for all $g \in K$ we have $|\langle \sigma(g)v_N, v_N \rangle| < \epsilon$ for all $g \in K$. This is possible since the convergence is uniform on $K$. Then the existence of such a $v_N$ implies by definition that $\sigma \in W(\rho, K, \epsilon; v)$. This being true for all $\epsilon > 0$, we deduce $(b)$. Conversely, supposing $(b)$, for every $n \geq 1$ pick $v_n$ such that $|\langle \sigma(g)v_n, v_n \rangle - \langle \rho(g) v, v \rangle| < \frac{1}{n}$ for all $g \in K$. Such a $v_n$ exists by definition of $W(\rho, K, \frac{1}{n}; v)$. Then $(K \to \mathbb{R} : g \mapsto \langle \sigma(g)v_n, v_n \rangle)_{n \geq 1}$ converges to $(K \to \mathbb{R} : g \mapsto \langle \rho(g)v, v \rangle)$ uniformly on $g$.

\end{comment}

\begin{comment}[Definition 3.1.3, P. 20]
\label{313}

Here we prove the equivalence of the definitions of a Kazhdan group. There is a third definition ($(b)$ below) that is not mentioned explicitly in Lubotzky's book, but that is used, for instance, in all of section 3.3. The proof of $(c) \Rightarrow (b)$ is taken from \cite[Proposition 1.2.1]{BHV}. We will skip $(a) \Rightarrow (b)$, since it uses tools that fall out of the scope of this document (and of my personal understanding). For a proof, and a more detailed discussion, see \cite[Subsection 1.2]{BHV}.

\begin{theorem*}
Let $G$ be a locally compact group , whose trivial representation we denote by $\rho_0$. The following are equivalent:
\begin{enumerate}[(a)]
\item The trivial representation is an isolated point in $\hat{G}$. That is, there exist $\epsilon > 0$ and a compact subset $K \subseteq_c G$ such that for every non-trivial irreducible unitary representation $(H, \rho)$ of $G$ and every vector $v \in H$ of norm one, $ || \rho(k) v - v|| > \epsilon$ for some $k \in K$. (This is the definition of property $(T)$ in Lubotzky's book).
\item There exist $\epsilon > 0$ and $K \subseteq_c G$ such that every unitary representation which has $(\epsilon, K)$-invariant vectors contains a non-zero invariant vector. In this case, $(\epsilon, K)$ is called a \textbf{Kazhdan pair} for $G$. (This is the definition of property $(T)$ in \cite{BHV}).
\item Every unitary representation which has almost invariant vectors (i.e., weakly contains $\rho_0$) contains a non-zero invariant vector (i.e., contains $\rho_0$). (This is the definition of property $(T)$ in \cite{HV}).
\end{enumerate}
\end{theorem*}

\begin{proof}
$(b) \Rightarrow (a), (c)$ is clear. \\

$(c) \Rightarrow (b)$. Suppose that $(b)$ does not hold. Then for all $\epsilon > 0$ and $K \subseteq_c G$, there exists some unitary representation $\rho_{\epsilon, K}$ on a Hilbert space $H_{\epsilon, K}$ that has $(\epsilon, K)$-invariant vectors but no non-zero invariant vector. Let $H = \oplus_{\epsilon, K} H_{\epsilon, K}$ be the Hilbert direct sum and $\rho = \oplus_{\epsilon, K} \rho_{\epsilon, K}$ the direct sum of representations, so that $\rho$ is a unitary representation of $G$ on $H$. Then each $\rho_{K, \epsilon}$ is a subrepresentation of $\rho$, and since $\rho_{K, \epsilon}$ has $(\epsilon, K)$-invariant vectors, $\rho$ does too. This being true for all $\epsilon$ and for all $K$, we conclude that $\rho$ has all almost invariant vectors. However, $\rho$ has no non-zero invariant vector, so $(a)$ does not hold.

To see this, suppose by contradiction that $v = (v_{\epsilon, K})_{\epsilon, K} \in H$ is a non-zero invariant vector. Then for all $g \in G$ we have $(v_{\epsilon, K})_{\epsilon, K} = v = \rho(g)(v) = (\rho_{\epsilon, K}(g)v_{\epsilon, K})_{\epsilon, K}$. Since $v \neq 0$, there exists some $(\epsilon, K)$ such that $v_{\epsilon, K} \neq 0$. But the equality above says that $v_{\epsilon, K} = \rho_{\epsilon, K}(g)v_{\epsilon, K}$, so $v_{\epsilon, K}$ is a non-zero invariant vector of $\rho_{\epsilon, K}$, which by hypothesis has no non-zero invariant vector.
\end{proof}

Next, we treat the case of separable groups. In definition 3.1.1, the author talks about unitary representations of (separable) groups on (separable) Hilbert spaces, which hints to the fact that when dealing with property $(T)$ for separable groups, we can restrict ourselves to separable Hilbert spaces. This is indeed the case. See the end of comment \ref{Fell_set} for why this is significant.

\begin{proposition*}
Let $G$ be a locally compact separable group. Then the following are equivalent to property $(T)$:
\begin{enumerate}[(a')]
\item The trivial representation is an isolated point in $\hat{G}'$, the space of irreducible unitary representations on separable Hilbert spaces. That is, there exist $\epsilon > 0$ and a compact subset $K \subseteq_c G$ such that for every non-trivial irreducible unitary representation $\rho$ on a separable Hilbert space $H$ of $G$ and every vector $v \in H$ of norm one, $ || \rho(k) v - v|| > \epsilon$ for some $k \in K$.
\item There exist $\epsilon > 0$ and $K \subseteq_c G$ such that every unitary representation on a separable Hilbert space which has $(\epsilon, K)$-invariant vectors contains a non-zero invariant vector.
\item Every unitary representation on a separable Hilbert space which has almost invariant vectors contains a non-zero invariant vector.
\end{enumerate}
\end{proposition*}

\begin{proof}
We will prove equivalences to the definitions $(a)$, $(b)$ and $(c)$ of property $(T)$, as formulated in the previous theorem. \\

$(a') \Rightarrow (a)$. We can actually say more: $\hat{G} = \hat{G}'$. That is, any irreducible unitary representation of a separable group must be on a separable Hilbert space. Indeed, let $(H, \rho)$ be an irreducible unitary representation of the separable group $G$. Let $D$ be a countable dense subgroup of $G$ (we can choose, for instance, the group generated by a given countable dense subset), and pick any non-zero vector $v \in H$. Let $V$ be the linear span of the countable set $D \cdot v$. Notice that $V$ is $D$-invariant. Since $V$ is spanned by a countable set, it is separable, so its closure $\overline{V} \subseteq H$ is also separable. Since addition and scalar multiplication are continuous in $H$, and $V$ is a vector space, $\overline{V}$ is also a vector space. Since a closed subset of a complete space is complete, $\overline{V}$ is a Hilbert space. Furthermore, it is a $G$-invariant subspace. Indeed, let $g \in G$ and $v \in \overline{V}$. Let $(v_\alpha)_{\alpha \in A} \subseteq V$ be a net converging to $v$ and $(d_\beta)_{\beta \in B} \subseteq D$ be a net converging to $g$. Then, since the representation is strongly continuous, $(\rho(d_\beta) v_\alpha)_{(\alpha, \beta) \in A \times B}$ is a net in $V$ converging to $\rho(g)v$ (here we are considering the natural structure of directed set on the product of two directed sets). We conclude that $\rho(g)v \in \overline{V}$, so $\overline{V}$ is a non-zero $G$-invariant Hilbert subspace of $H$. Since $\rho$ is irreducible, this implies that $\overline{V} = H$. Since $\overline{V}$ is separable, $H$ is separable, which is what we wanted to prove. \\

$(c') \Rightarrow (c)$. Let $\rho$ be a unitary representation on a Hilbert space $H$, which is not necessarily separable, and suppose that $H$ has almost invariant vectors. We need to show that $H$ has a non-zero invariant vector. Let $D$ be a countable dense subgroup of $G$, and enumerate $D = \{e = d_1, d_2, \ldots \}$. Fix a compact neighbourhood $K_1$ of the identity, and let $U_1$ be an open set such that $e \in U_1 \subseteq K_1$. Then let $U_i := d_iU_1$ and $K_i = d_i K_1$, so that $d_i \in U_i \subseteq K_i$. Notice that since $D$ is dense, and the $U_i$ are all shifts of the same open set, $(U_i)_{i \geq 1}$ is an open cover of $G$. Finally let $L_n = \bigcup\limits_{i = 1}^n K_i$, which is compact. Then the set $\mathcal{S} = \{ (\epsilon, L_n) : \epsilon \in \mathbb{Q}, n \in \mathbb{N} \}$ is countable. By hypothesis, for all $(\epsilon, L_n) \in \mathcal{S}$, we can choose a non-zero $(\epsilon, L_n)$-invariant vector $v_{(\epsilon, L_n)}$. Let $V$ be the linear span of the countable set $D \cdot \{ v_s : s \in \mathcal{S} \}$. By the same argument as in the previous paragraph, $\overline{V}$ is a non-zero $G$-invariant Hilbert subspace of $H$.

This yields a subrepresentation of $G$ on a separable Hilbert space. We now show that $\overline{V}$ has almost invariant vectors. Indeed, let $\epsilon > 0$ and $K \subseteq_c G$. Let $\epsilon' \leq \epsilon$ be rational. By compactness of $K$, and since $G$ is covered by the $U_i$, there exists $n \geq 1$ such that $K \subseteq \bigcup\limits_{i = 1}^n U_i$. It follows that $K \subseteq L_n$. Then if we pick any $(\epsilon', L_n)$-invariant vector, it will also be an $(\epsilon, K)$-invariant vector, and such a vector exists by construction of $\overline{V}$. By $(c')$, we conclude that $\overline{V}$ has a non-zero invariant vector and so $H$ also has a non-zero invariant vector. \\

Clearly each of these statements is weaker than their general version in the previous theorem. Also, the implication $(b') \Rightarrow (c')$ is again trivial. Therefore we have $(a) \Leftrightarrow (a')$, so $(a')$ is equivalent to property $(T)$; and $(c') \Rightarrow (c) \Rightarrow (b) \Rightarrow (b') \Rightarrow (c')$, so $(b')$ and $(c')$ are also equivalent to property $(T)$.
\end{proof}
\end{comment}

\begin{comment}[Theorem 3.1.5, PP. 21-22]

Here we prove the integral equalities illustrated in figures 1 and 2 on page 22. Let $0 \leq F \in L^1(G)$, and let $U_a := \{ x \in G : F(x) \geq a \}$. Let $\chi_a$ be the characteristic function of $U_a$. Then the first equality is:
$$\int\limits_0^\infty \lambda(U_a) da = \int\limits_0^\infty \left( \int_G \chi_a(x) d \lambda(x) \right) da = \int_G \left( \int\limits_0^\infty \chi_a(x) da \right) d \lambda(x) = $$ $$ = \int_G \left( \int\limits_0^{F(x)} 1 da \right) d \lambda(x) = \int_G F(x) d \lambda(x).$$
Given $k \in G$, let $\psi_a$ be the characteristic function of $kU_a \Delta U_a$. Then the second equality is:
$$\int\limits_0^\infty \lambda(kU_a \Delta U_a) da = \int\limits_0^\infty \left( \int_G \psi_a(x) d \lambda(x) \right) da = \int_G \left( \int\limits_0^\infty \psi_a(x) da \right) d \lambda(x). \overset{*}{=} $$ $$ \overset{*}{=} \int_G \left( \int\limits_{\min \{ kF(x), F(x) \} }^{\max \{ kF(x), F(x) \} } 1 da \right) d \lambda(x) = \int_G |kF(x) - F(x)| d \lambda(x).$$
Here $(*)$ follows from the fact that $\psi_a(x) = 1$ if and only if either $x \in kU_a \, \backslash \, U_a$, in which case $F(x) < a \leq F(k^{-1}x)$; or $x \in U_a \, \backslash \, kU_a$, in which case $F(k^{-1}x) < a \leq F(x)$. By definition $F(k^{-1}x) = kF(x)$.
\end{comment}

\begin{comment}[Corollary 3.1.6 - Proposition 3.1.7, P. 23]
\label{cont_hom_amenable}

Here the author uses that in any group $G$, the abelianization $G/[G, G]$ is abelian, so amenable. We have proved that an abelian discrete group is amenable (see comment \ref{amenable_discrete}), but to prove that any abelian topological group is abelian we need something else.

\begin{proposition*}
Let $G, H$ be locally compact groups, $\phi : G \to H$ a continuous homomorphism with dense image. Then if $G$ is amenable, $H$ is amenable.
\end{proposition*}

A proof of this can be found in \cite[Proposition 1.2.1]{Runde}. Note that by taking $G$ to be a discrete group, $H$ the same underlying group with another topology, and $\phi$ the identity, this proves that if a group is amenable as a discrete group, then it is amenable for any other structure of locally compact group. This also proves that quotients of amenable groups are amenable, which is used in proposition 3.1.11 (see comment \ref{3111}). \\

However, we do not need this for the proof of 3.1.7. Indeed, the free group of rank $F_2$ has as abelianization the infinite discrete abelian group $\mathbb{Z}^2$, so it cannot have property $(T)$. If $F$ is any free group (even of infinite rank, so with an abelianization that is not discrete with its usual topology), then it has $F_2$ as a quotient, and since $F_2$ does not have property $(T)$, $F$ does not have property $(T)$ either.

\end{comment}

\begin{comment}[Proposition 3.1.8, P. 23]
\label{continuity_Fell}

The proof of proposition 3.1.8 would take up too much space. Proofs can be found in \cite{BHV} (theorem F.3.5 for $(a)$ and theorem 1.7.1 for $(b)$). However, there is a significant difference, which is an important point to make since \cite{BHV} is the main reference when it comes to property $(T)$. Indeed, Lubotzky defines weak containment and the Fell topology in terms of approximation of coefficients by coefficients. On the other hand, Bekka, de la Harpe and Valette define weak containment and the Fell topology in terms of approximations of coefficients by finite sums of coefficients. \\

More explicitly, let $\rho, \sigma$ be two unitary representations of the locally compact group $G$ on Hilbert spaces $H(\rho), H(\sigma)$. We say that $\rho$ is "weakly contained" in $\sigma$, denoted $\rho \prec \sigma$, if every coefficient of $\rho$ is a limit, uniformly on compact sets, of finite sums of coefficients of $\sigma$.

Accordingly, we define the "Fell topology" by specifying a subbase of neighbourhoods of $\rho$ as follows. If $\epsilon > 0, K \subseteq_c G$ and $v \in H(\rho)$ is of norm one, then define
$$V(\rho, K, \epsilon; v) = \{(H', \sigma) : \text{there exists a finite sum } \psi \text{ of coefficients of } \sigma \text{ such that }$$ $$| \langle v, \rho(g) v \rangle - \psi(g) | < \epsilon \text{ for all } g \in K \}.$$
Just as in the beginning of this section, we can similarly define $V(\rho, K, \epsilon; v_1, \ldots, v_n)$ if $v_i \in H(\rho)$ are vectors of norm one. (Actually in \cite{BHV} all this is defined in terms of functions of positive type, but the definition above is equivalent by C.4.10). \\

These definitions of "weak containment" and "Fell topology" are compatible as the definitions of weak containment and Fell topology in Lubotzky's book, meaning that $\rho \prec \sigma$ if and only if $\sigma \in U$ for any neighbourhood $U$ of $\rho$ in the "Fell topology" (see comment \ref{311}). The key fact (F.1.4) is that if $\rho$ is irreducible, then $\rho \propto \sigma$ if and only if $\rho \prec \sigma$. This shows that the two definitions coincide on $\hat{G}$, and so the definition of property $(T)$ is the same. Also, the characterization of amenability in 3.1.5 ($G$ is amenable if and only if $\rho_0 \propto L_G$) is the same in these terms ($G$ is amenable if and only if $\rho_0 \prec L_G$). \\

The nice thing about working with the "Fell topology" is that a function $\phi : \tilde{G} \to \tilde{H}$ which is compatible with the direct sum of representations (that is: $\oplus_i \phi(\rho_i) \prec \phi(\oplus_i \rho_i) \prec \oplus_i \phi(\rho_i)$) is continuous if and only if it preserves $\prec$ (that is, if $\rho \prec \sigma$ then $\phi(\rho) \prec \phi(\sigma)$). This follows immediately from F.2.2. This is the tool used to show the continuity of the most natural operations with respect to the Fell topology: direct sum, restriction, induction, tensor product. Actually, the fact that the tensor product preserves the weak containment of the trivial representation is also used in the proof of proposition 3.1.11 (see comment \ref{3111}).

\end{comment}

\begin{comment}[Definition 3.1.10, P. 24]
\label{TR}

As in the case of property $(T)$ (see comment \ref{313}), property $(T : R)$ also has equivalent definitions in terms of $\hat{G}$ and $\tilde{G}$. More precisely, let $\mathcal{R} \subseteq \hat{G} \, \backslash \, \{ \rho_0 \}$. Then the following are equivalent.
\begin{enumerate}[(a)]
\item $\rho_0$ is an isolated point in $\mathcal{R} \cup \{ \rho_0 \}$.
\item There exists $\epsilon > 0$ and $K \subseteq_c G$ such that every unitary representation which has $(\epsilon, K)$-invariant vectors contains some $\rho \notin \mathcal{R}$.
\item Every unitary representation which has almost invariant vectors contains some $\rho \notin \mathcal{R}$.
\end{enumerate}
For a special case of this, and in fact the only case we are interested in for theorem 3.1.12, see chapter 1.4 in \cite{BHV}. This is important for our purposes, since the way this property is proved in proposition 3.1.11 uses $(a)$, while its application in the proof of theorem 3.1.12 uses $(c)$.

\end{comment}

\begin{comment}[Proposition 3.1.11, PP. 24-25]
\label{3111}

Here we fill in the details of the proof of proposition 3.1.11. Let $H := \mathbb{R}^2 \rtimes SL_2(\mathbb{R})$ with the natural action. Let $\mathcal{R} = \{ \rho \in \hat{H} : \rho|_{\mathbb{R}^2} \neq \rho_0 \}$. We want to prove that $\rho_0$ is isolated in $\mathcal{R} \cup \{ \rho_0 \} \subseteq \hat{H}$. \\

The first step is to study the action of $SL_2(\mathbb{R})$ on $\hat{\mathbb{R}}^2$. Recall that if $G$ is a locally compact abelian group, then any strongly continuous unitary representations of $G$ can be decomposed into a direct sum, or integral, of irreducible unitary representations, and all irreducible representations are of dimension 1 (see chapter 7.3 of \cite{Kow}). Therefore $\hat{\mathbb{R}}^2$ is the space of characters $\chi : \mathbb{R}^2 \to S^1$. The action of $SL_2(\mathbb{R})$ on $\mathbb{R}^2$ induces an action on $\hat{\mathbb{R}}^2$ by $A\chi : \mathbb{R}^2 \to S^1 : x \to \chi(A^{-1}x)$, for $A \in SL_2(\mathbb{R})$, $\chi \in \hat{\mathbb{R}}^2$. \\

We prove that this action is transitive on non-trivial characters. Let $\chi$ be a non-trivial character. Then $\chi(\mathbb{R}^2)$ is a non-trivial subgroup of $S^1$, which must be connected by continuity, so $\chi$ is surjective. We will find $A \in SL_2(\mathbb{R})$ such that $A \chi = \chi_0$, where $\chi_0
\left(
\begin{array}{c}
\alpha_1 \\
\alpha_2
\end{array}
\right)
 = e(\alpha_2)$ (for $\alpha \in \mathbb{R}$, we define $e(\alpha) = e^{2 \pi i \alpha}$).
 
The map $e : \mathbb{R} \to S^1$ is a cover, so by the unique lifting property there exists a continuous map $\phi : \mathbb{R}^2 \to \mathbb{R}$ such that $e\phi = \chi$. We claim that $\phi$ can be chosen to be a group homomorphism. Indeed, the identity $e\phi = \chi$ means that $\phi(x + y) - \phi(x) - \phi(y) = n(x, y) \in \mathbb{Z}$. Then $n : \mathbb{R}^4 \to \mathbb{Z}$ is continuous, so its image is connected in $\mathbb{Z}$, so $n(x, y) = n$ is constant. Replacing $\phi$ by $\phi - n$, we get a continuous homomorphism $\phi : \mathbb{R}^2 \to \mathbb{R}$. But then $\phi$ must be linear. Since it cannot be trivial, as $\chi$ is surjective, we conclude that there exists a basis $(f_1, f_2)$ of $\mathbb{R}^2$ such that $\phi(\alpha_1 f_1 + \alpha_2 f_2) = \alpha_2$. So $\chi(\alpha_1 f_1 + \alpha_2 f_2) = e(\alpha_2)$. Up to multiplying $f_1$ by a scalar, we may assume that if $A = (f_1, f_2)$, then $\det (A) = 1$. Then $A^{-1} \chi (\alpha e_1 + \alpha_2 e_2) = \chi(\alpha_1 f_1 + \alpha_2 f_2) = e(\alpha_2)$. So $A^{-1} \chi = \chi_0$. \\

We calculate the stabilizer of one non-trivial character. Let $\chi_0$ be as above. Let $A^{-1}$ be in the stabilizer of $\chi_0$, and write $A = 
\left(
\begin{array}{cc}
a & b \\
c & d \\
\end{array}
\right)$.
Then for every $\alpha \in \mathbb{R}^2$ we have $e(\alpha_2) = \chi_0(\alpha) = (A^{-1} \chi_0)(\alpha) = \chi_0(A \alpha) = \chi_0
\left(
\begin{array}{c}
a \alpha_1 + b \alpha_2 \\
c \alpha_1 + d \alpha_2
\end{array}
\right)
= e(c \alpha_1 + d \alpha_2)$. Thus $e(c \alpha_1 + (d - 1) \alpha_2) = 0$, so $c \alpha_1 + (d - 1) \alpha_2 \in \mathbb{Z}$, and this is true for any $\alpha_1, \alpha_2 \in \mathbb{R}$. Therefore the linear subspace of $\mathbb{R}$ generated by $c$ and $(d-1)$ is contained in $\mathbb{Z}$, so it must be trivial. Then $c = 0, d = 1$ and $1 = det(A) = ad - bc = a$. So $A = 
\left(
\begin{array}{cc}
1 & b \\
0 & 1 \\
\end{array}
\right)$. Since all matrices of this form stabilize $\chi_0$, it follows that the group $N$ of these matrices is exactly the stabilizer of $\chi_0$ for the action of $SL_2(\mathbb{R})$. In particular $N \cong \mathbb{R}$, therefore $\mathbb{R}^2 \rtimes N$ is an extension of abelian, so amenable groups, so it is itself amenable.
Since the action is transitive on non-trivial characters, this is also true for $\mathbb{R}^2 \rtimes M$, where $M < SL_2(\mathbb{R})$ is the stabilizer of any other non-trivial character. \\

Let $\rho \in \mathcal{R}$. Then $\rho|_{\mathbb{R}^2}$ is a non-trivial unitary representation of the abelian group $\mathbb{R}^2$, so it is a direct sum, or integral, of one-dimensional characters of $\mathbb{R}^2$. Let $\chi$ be one of those characters. Since $\rho|_{\mathbb{R}^2}$ is non-trivial, we can assume that $\chi$ is also non-trivial. Let $M$ be its stabilizer in $SL_2(\mathbb{R})$, so $M \cong N$, the stabilizer described in the previous paragraph. Then there exists an irreducible unitary representation $\sigma$ of $M$ such that $\rho = \operatorname{Ind}_{M_1}^H \chi \sigma$, where $\chi \sigma$ is the coordinate-wise representation of $M_1 := \mathbb{R}^2 \rtimes M$ defined by $\chi$ and $\sigma$. This is proven in \cite[Example 7.3.3]{Zimmer} and is a corollary of Mackey's theorem \cite[Theorem 7.3.1]{Zimmer}. 

As we already mentioned, $M_1$ is amenable, which by 3.1.5 implies that $\rho_0 \propto L_{M_1}$, the left regular representation of $M_1$. Now since the tensor product with the trivial representation yields an equivalent one, and the tensor product preserves the weak containment of the trivial representation (see the end of comment \ref{continuity_Fell}), we have: $\chi \sigma \cong (\chi \sigma \otimes \rho_0) \propto \chi \sigma \otimes L_{M_1}$. \\

We claim that $\chi \sigma \otimes L_{M_1} \cong L_{M_1}$, so $\chi \sigma \propto L_{M_1}$. We start by noticing that $\chi, \sigma$ are irreducible representations of the locally compact abelian groups $\mathbb{R}^2$ and $M \cong \mathbb{R}$, so they are both of dimension one. By the definition of $\chi \sigma$, this representation is also of dimension one. Then we have the following general fact: if $G$ is a group and $\psi$ is a unitary representation of $G$ of dimension one, then $L \otimes \psi$ is equivalent to $L$. Here $L$ is the left regular representation of $G$ on $L^2(G)$ defined by $L(g)(f)(x) = f(g^{-1}x)$ for all $g, x \in G$.

Note that in the book a more general fact is stated, but once we clarified that dimension of $\chi \sigma$ is one, there are fewer calculations to do. Also note that the formula provided by Lubotzky is not an equivalence of representations: this is easily seen to be the case looking at the calculations in the next paragraph. The $\theta$ we define differs slightly and gives a homomorphism of representations, and the definition can be adapted to prove the equivalence of representations in higher dimensions as well, as stated in the book. \\ 

We prove this equivalence of representations. Define $\theta : L^2(G) \otimes \mathbb{C} \to L^2(G)$ by $\theta(f \otimes z)(x) = f(x)\psi(x^{-1})z$. The map $\theta$ is well-defined: since $\psi$ is a one-dimensional unitary representation, $|\psi(x)| = 1$, so the integrability of $f$ implies that of $\theta(f \otimes z)$. It is a linear map, since the associated map $\tilde{\theta} : L^2(G) \times \mathbb{C} \to L^2(G)$ defined by $\tilde{\theta}(f, z)(x) = f(x)\psi(x^{-1})z$ is easily seen to be bilinear.

Next we want to prove that for all $g \in G$ we have $L(g) \circ \theta = \theta \circ (L \otimes \psi)(g) : L^2(G) \otimes \mathbb{C} \to L^2(G)$. Given $f \in L^2(G), z \in \mathbb{C}, x \in G$, we calculate:
$$(L(g) \circ \theta)(f \otimes z)(x) = \theta(f \otimes z)(g^{-1}x) = f(g^{-1}x)\psi((g^{-1}x)^{-1})z = L(g)(f)(x) \cdot \psi(x^{-1}) \psi(g)z = $$ $$ = \theta(L(g)(f) \otimes \psi(g) z)(x) = (\theta \circ (L \otimes \psi)(g))(f \otimes z)(x).$$
Therefore $\theta$ is a homomorphism of representations $L \otimes \psi \to L$.

Finally we need to prove that $\theta$ is an isomorphism. It is injective: suppose that $\theta(f \otimes z) = 0$. Then for all $x \in G$ we have $0 = \theta(f \otimes z)(x) = f(x)\psi(x^{-1})z$. Since $\psi(x^{-1})$ is always non-zero, this is equivalent to $f(x)z = 0$ for all $x \in G$. So either $z = 0$ or $f(x) = 0$ for all $x \in G$. In both cases $f \otimes z = 0$. It is surjective: for all $x \in G$ we have $f(x) = (f\psi)(x)\psi(x^{-1})1 = \theta(f \psi \otimes 1)(x)$. Once again, since $|\psi(x)| = 1$, the integrability of $f$ implies that of $f \psi$. We conclude that $\theta$ is an equivalence of representations, and thus $L \otimes \psi \cong L$. \\

Therefore $\chi \sigma \propto L_{M_1}$. Since induction preserves weak containment (by continuity: see comment \ref{continuity_Fell}), it follows that $\rho = \operatorname{Ind}_{M_1}^H (\chi \sigma) \propto \operatorname{Ind}_{M_1}^H ( L_{M_1} ) = L_H$, since the induced representation of the regular representation of a subgroup is the regular representation of the ambient group.

Finally, suppose by contradiction that $\rho_0$ is not isolated from $\mathcal{R} \cup \{ \rho_0 \}$. Then for any neighbourhood $U$ of $\rho_0$ in $\tilde{G}$ there exists $\rho \in U \cap \mathcal{R}$, by comment \ref{311}. But by the above $\rho \propto L_H$, so $L_H$ is in all neighbourhoods of $\rho$, in particular $L_H \in U$. This implies that $L_H$ is in all neighbourhoods of $\rho_0$, so again by comment \ref{311} we have $\rho_0 \propto L_H$, and $H$ is amenable. This is false by 3.1.9, and by the fact that the quotient of an amenable group is also amenable (see comment \ref{cont_hom_amenable}). We conclude that $\rho_0$ is isolated in $\mathcal{R} \cup \{ \rho_0 \}$, which is what we wanted to prove. 

Note that this conclusion is different from the one in the book. Lubotzky concludes by saying that $\rho_0$ is not weakly contained in $\rho$ for any $\rho \in \mathcal{R}$, which is the case, but does not imply that $\rho_0$ is isolated from $\mathcal{R}$.

\end{comment}

\begin{comment}[Lemma 3.1.13, PP. 25-26]
\label{3113}

Here we fill in a few details of the proof of lemma 3.1.13.

In Step 2, the conclusion is reached by the identification of cosets with points in $\mathbb{R} \, \backslash \, \{ 0 \}$. More specifically, if $p \in P$, then $pN$ is identified with $p \cdot {1\choose0}$, and $\{ p \cdot {1\choose0} : p \in P \}$ is in fact the $x$-axis.

Notice that in both Step 2 and Step 4, in order to go back and forth, we need the quotient by the stabilizer and the orbit to be homeomorphic. This is false in general, but true in our case because $SL_2(\mathbb{R}), \mathbb{R} \, \backslash \, \{ 0 \}$ and $P^1(\mathbb{R})$ are locally compact Hausdorff and $SL_2(\mathbb{R})$ is separable. Indeed, it is well-known that if $G$ is a locally compact Hausdorff group which is separable, and $G \times X \to X$ is a transitive continuous action on the locally compact Hausdorff space $X$, then the induced bijection $G/G_x \to X$ is a homeomorphism.

\end{comment}

\begin{comment}[Lemma 3.1.14, PP. 26]

Here we fill in a few details of the proof of lemma 3.1.14. In particular, we show that $E_1$ and $E_2$ generate a dense subgroup, and that lemma 3.1.13 implies that a vector fixed by $N_i$ is also fixed by $E_i$, for $i = 1, 2$. \\

We start with the first claim. Let $S$ be the subgroup of $G = SL_3(\mathbb{R})$ generated by $E_1$ and $E_2$. We will actually show that $S = G$, which is even stronger than $\overline{S} = G$. Let $T_{ij}(x)$, for $i \neq j \in \{ 1, 2, 3 \}$ and $x \in \mathbb{R}$, be the element of $G = SL_3(\mathbb{R})$ with $x$ in the $(i, j)$ coordinate, ones on the diagonal and zeros everywhere else. These elements are called transvections, and we start by showing that they are all in $S$. Since $E_1, E_2 \subset S$, we already have all the transvections except for $T_{12}(x)$ and $T_{21}(x)$. Also, note that $E_1$ and $E_2$ are stable under transposition, so $S$ is also stable under transposition, which implies that it is enough to show that $T_{12}(x) \in S$ for any $x \in \mathbb{R}$. One can check that $T_{13}(-1)T_{32}(-x)T_{13}(1)T_{32}(x) = T_{12}(x)$. So $S$ contains all the transvections. \\

Let $A \in G$. By a the process of Gauss elimination, we can multiply $A$ on the left by transvections to get a matrix $B$ whose first column has a unique non-zero entry $x$. By multiplying on the left by two more transvections, if necessary, we can move this to the first entry (if $x$ is the second entry, we multiply by $T_{21}(-1)T_{12}(1)$, and similarly if $x$ is the third entry). Then
$$
\left(
\begin{array}{ccc}
x^{-1} & 0 & 0 \\
0 & 1 & 0 \\
0 & 0 & x \\
\end{array}
\right)
\left(
\begin{array}{c}
x \\
0 \\
0 \\
\end{array}
\right)
=
\left(
\begin{array}{c}
1 \\
0 \\
0 \\
\end{array}
\right).
$$
Since this matrix is in $E_1$, we have shown that any $A \in G$ can be multiplied on the left by an element of $S$ to get to a matrix whose first column is $e_1$. Then by multiplying on the right by transvections, we eliminate the rest of the first row, which leaves us with a matrix in $E_2$. Thus $S_1AS_2 = S_3$, for some $S_1, S_2, S_3 \in S$, and so $A \in S$. We conclude that $S = G$. \\

For both $E_1$ and $E_2$, there is a natural bijection with $E = SL_2(\mathbb{R})$. A small calculation shows that this is also an isomorphism (and it shows in fact that $E_1$ and $E_2$ are groups). Under these isomorphisms, $N$ is sent to $N_i$. So given a unitary representation of $G$, we let the isomorphism $E_i \cong E$ induce a unitary representation of $E$. If a vector is fixed by $N_i$, then it is fixed by $N$ in this induced representation, which by lemma 3.1.13 implies that it is fixed by $E$ under this induced representation, which finally implies that it is fixed by $E_i$ under the original representation. This proves the second claim.

\end{comment}

\begin{comment}[P. 26]

Here we fill the details of the proof of theorem 3.1.12.

Let $G = SL_3(\mathbb{R})$. Consider the subgroups:
$$
E := \{ \left(
\begin{array}{cc}
A & 0 \\
0 & 1 \\
\end{array}
\right)
: A \in SL_2(\mathbb{R}) \};
\hspace{1cm}
J := \{ \left(
\begin{array}{cc}
I_2 & v \\
0 & 1 \\
\end{array}
\right)
: v \in \mathbb{R}^2 \}; \,$$
$$H :=\{ \left(
\begin{array}{cc}
A & v \\
0 & 1 \\
\end{array}
\right)
: A \in SL_2(\mathbb{R}), v \in \mathbb{R}^2 \}.
$$
Then $E \cong SL_2(\mathbb{R})$, $J \cong \mathbb{R}^2$, and $H$ is the internal semidirect product $J \rtimes E$, which is naturally isomorphic to the external semidirect product $\mathbb{R}^2 \rtimes SL_2(\mathbb{R})$, since the action of $E$ on $J$ by conjugation corresponds, under the respective isomorphisms, to the standard action of $SL_2(\mathbb{R})$ on $\mathbb{R}^2$. All of this is easy to check by block multiplication. \\

Let $\rho$ be a unitary representation of $G$ which weakly contains the trivial representation $\rho_0$. We want to show that $\rho$ contains $\rho_0$. Now $\rho|_H$ also weakly contains the trivial representation. By $(c)$ in comment \ref{TR} and proposition 3.1.11, this means that $\rho$ must contain some unitary representation not in $\mathcal{R}$, so some unitary representation whose restriction to $\mathbb{R}^2$ is trivial. Therefore $\rho_0 \subseteq \rho|_J$, which means that there is some non-zero $J$-invariant vector for $\rho$, which by lemma 3.1.14 is also $G$-invariant.

\end{comment}

\begin{comment}[P. 26]
\label{compact_generation}

Here the author probably means that every discrete group (not every countable group) with property $(T)$ is finitely generated. Indeed, $SL_3(\mathbb{Q})$ is being considered as a discrete group: if it were considered as a group with its usual topology then it would not be locally compact (its center is closed and isomorphic to $\mathbb{Q}$, which is not locally compact), and so by definition we cannot discuss property $(T)$. This is why we require non-discreteness, and this is indeed proven in the references given in the book. More generally, every group with property $(T)$ is compactly generated. We prove this following \cite{BHV}. \\

Let $G$ be a locally compact group and $H \leq G$ a closed subgroup. If $G/H$ carries a $G$-invariant measure, then we have a unitary representation on $L^2(G/H)$, called the regular representation, defined just as the regular representation on $L^2(G)$. In the general case this is not true, but we still have the unitary representation $\lambda_{G/H}$ on $L^2(G/H)$, called the \textbf{quasi-regular representation}, which keeps some of the nice properties of the regular representation. Specifically, we are interested in the two following facts: the Dirac mass $\delta_H \in L^2(G/H)$ is $H$-invariant, and if $L^2(G/H)$ contains a non-zero invariant vector, then $G/H$ is compact. All of this is proven in detail in appendix B.1 of \cite{BHV}.

\begin{proof}

Let $\mathcal{C} := \{ H \leq G : H$ is an open subgroup generated by a subset of a compact set in $G \}$. Since $G$ is locally compact, each element has a compact neighbourhood, and the subgroup generated by a non-empty open set is open. So $G = \cup_{H \in \mathcal{C}} H$. Every open subgroup of a topological group is also closed, so we have a quasi-regular representation $\lambda_{G/H}$ on $L^2(G/H)$. Also, since $H$ is open, $\{ H \} \subseteq G/H$ is open, so $G/H$ is discrete. In particular, suppose that there exists some $H \in \mathcal{C}$ such that $\lambda_{G/H}$ has a non-zero invariant vector. Then $G/H$ is discrete and compact, so finite, and since $H$ is generated by a subset of a compact set, $G$ is compactly generated. So our goal is to prove this fact. \\

Let $\lambda = \bigoplus\limits_{H \in \mathcal{C}} \lambda_{G/H}$ be the direct sum of representations. Then $\lambda$ has almost invariant vectors. Indeed, let $K \subseteq_c G$. Since $K \subseteq \bigcup\limits_{H \in \mathcal{C}} H$, there exist $H_1, \ldots, H_n \in \mathcal{C}$ such that $K \subseteq \bigcup\limits_{i = 1}^n H_i \subseteq H_0$, where $H_0 \in \mathcal{C}$ is the group generated by the $H_i, 1 \leq i \leq n$. (This seems uselessly complicated since we could have chosen $H_0$ to be the group generated by $K$, but then $H_0$ would not necessarily be open). Since the Dirac mass $\delta_{H_0} \in \lambda_{G/H_0}$ is $H_0$-invariant, it is also $K$-invariant. But $\lambda_{G/H_0} \subseteq \lambda$, so $\lambda$ has $(\epsilon, K)$-invariant vectors for all $\epsilon > 0$. Therefore $\lambda$ has almost invariant vectors.

Since $G$ has property $(T)$, there exists a non-zero invariant vector $(v_H)_{H \in \mathcal{C}} \in \bigoplus\limits_{H \in \mathcal{C}} L^2(G/H)$. Since it is non-zero, let $H$ be such that $v_H \neq 0$. Then $v_H \in L^2(G/H)$ is a non-zero invariant vector, so we conclude.
\end{proof}

\end{comment}

\subsection{Lattices and arithmetic subgroups}

\begin{comment}[P. 29]
\label{T_ext}

Here the author mentions that a direct product of finitely many Kazhdan groups is also Kazhdan. By induction it is enough to prove it for two groups. We prove this fact next, which is trickier than it seems. The easiest way to do it is to use the equivalent definition with Kazhdan pairs (see comment \ref{313}). We will actually prove that a large family of extensions, containing the direct products, satisfies this. We do this following \cite{BHV}; more specific references are given at the statement of each result.

\begin{lemma*}
Let $G$ be a group, $\rho$ a unitary representation on the Hilbert space $H$. Let $V \subseteq H$ be a $G$-invariant subspace. Let $W$ be the orthogonal complement of $V$, and $P_V, P_W : H \to V, W$ the projections. Then
\begin{enumerate}[(i)]
\item $W$ is also $G$-invariant.
\item For all $v \in H$, $||P_{V, W}v|| \leq ||v||$.
\item $\rho(g) P_{V, W} = P_{V, W} \rho(g)$ for all $g \in G$, where on the left-hand side we consider $\rho(g)$ as a representation on $V, W$.
\end{enumerate}
\end{lemma*}

\begin{proof}
$(i)$: Since the representation is unitary, it preserves orthogonality, so the $G$-invariance of $V$ implies that of its orthogonal complement $W$. So from now on the roles of $V$ and $W$ are symmetric, and we only need to prove the next two points for $V$. \\

$(ii)$: By orthogonality $||v||^2 = ||P_Vv + P_Wv||^2 = ||P_Vv||^2 + ||P_Wv||^2$ for all $v \in H$. \\

$(iii)$: Let $v \in H$. Then for all $g \in G$ we have:
$$\rho(g)P_V(v) = \rho(g)P_V(v) + 0 = P_V(\rho(g)P_V(v)) + P_V(\rho(g)P_W(v)) = $$
$$ = P_V(\rho(g)P_V(v) + \rho(g)P_W(v)) = P_V(\rho(g)(v)).$$
\end{proof}

\begin{proposition*}[Proposition 1.1.9]
Let $G$ be a group with property $(T)$, and let $(\epsilon, K)$ be a Kazhdan pair for $G$. Let $\delta > 0$, $\rho$ a unitary representation on the Hilbert space $H$, and $P : H \to H^G$ be the orthogonal projection onto the space of $G$-invariant vectors. Suppose that there exists an $(\epsilon \delta, K)$-invariant vector $v \in H$. Then $||v - Pv|| \leq \delta||v||$.
\end{proposition*}

\begin{remark}
Intuitively, this proposition gives a measure of how close almost-invariant vectors can get to invariant vectors.
\end{remark}

\begin{proof}
Write $v = a + b$, where $a = Pv \in H^G$ and $b = v - a \in (H^G)^\perp$. If $b = 0$, then $v$ is invariant, so we may assume that $b \neq 0$. Let $0 < t < 1$. Now $(H^G)^\perp$ is $G$-invariant subspace, so we can consider the representation of $G$ on $(H^G)^\perp$. Since $(\epsilon, K)$ is a Kazhdan pair for $G$, and $(H^G)^\perp$ has no non-zero invariant vectors, $(H^G)^\perp$ has no $(\epsilon, K)$-invariant vectors. In particular, $b$ is not $(\epsilon, K)$-invariant, thus there exists some $g \in K$ such that $||\rho(g)b - b|| \geq t \epsilon ||b||$. Let $P' = P - Id$ be the orthogonal projection onto $(H^N)^\perp$. Then by the previous lemma:
$$||\rho(g)b - b|| = ||\rho(g)P'v - P'v|| = ||P'(\rho(g)v - v)|| \leq ||\rho(g)v - v|| < \delta \epsilon ||v||.$$
Then $||v - Pv|| = ||b|| \leq \frac{1}{t \epsilon}||\rho(g)b - b|| < \frac{\delta}{t}||v||$. This being true for all $0 < t < 1$, we conclude.
\end{proof}

\begin{lemma*}[Lemma 1.7.5]
Let $G$ be a locally compact group, $N$ a locally compact normal subgroup, $\pi : G \to G/N$ the canonical projection. Suppose that $N$ and $G/N$ have property $(T)$, with Kazhdan pairs $(\epsilon_1, K_1)$ and $(\epsilon_2, K_2)$ respectively. Suppose $G$ has a compact subset $K'$ such that $K_2 \subseteq \pi(K')$. Let $K = K_1 \cup K' \subseteq_c G$ and $\epsilon = \frac{1}{2}\min(\epsilon_1, \epsilon_2)$. Then $(\epsilon, K)$ is a Kazhdan pair for $G$. In particular, $G$ has property $(T)$.
\end{lemma*}

\begin{remark}
Every quotient of a locally compact group is locally compact. Every closed subgroup of a locally compact group is locally compact. This is why in \cite[Lemma 1.7.5]{BHV}, the statement is with $N$ a closed normal subgroup. However, this is the only reason we assume $N$ to be closed, as is clear by looking at the proof. So we rather state the theorem in this slightly more general form, which allows us to apply it to direct products. In that case, we want $N$ to be one of the factors, say $G_1 \times \{ e \}$. This is closed in $G_1 \times G_2$ if and only if $\{ e \}$ is closed in $G_2$, which is equivalent to $G_2$ being Hausdorff. On the other hand, local compactness of $G_1 \times \{ e \}$ is immediate, since $G_1 \to G_1 \times \{ e \}$ is an embedding, and $G_1$ is locally compact.
\end{remark}

\begin{proof}
Let $\rho$ be a unitary representation of $G$ on the Hilbert space $H$ with an $(\epsilon, K)$-invariant vector $v_0$ of norm 1. We want to show that $H$ contains a non-zero invariant vector. Let $H^N$ be the subspace of $N$-invariant vectors of $H$. Then $H^N$ is $G$-invariant: if $v \in H^N$, then for all $g \in G$, for all $n \in N$ we have: $\rho(n)(\rho(g)v) = \rho(g)(\rho(g^{-1}ng)v) = \rho(g)v$ since $N$ is normal in $G$. So $\rho(g)v \in H^N$. Let $P : H \to H^N$ be the orthogonal projection of $H$ onto $H^N$. Then by the first lemma $\rho(g) P = P \rho(g) : H \to H^N$ for all $g \in G$.

We claim that $Pv_0 \in H^N$ is $(\epsilon_2, K)$-invariant. Choosing $\delta = \frac{1}{2}$ in the previous proposition, since $v_0$ is $(\frac{1}{2}\epsilon_1, K_1)$-invariant, we have $||v_0 - Pv_0|| \leq \frac{1}{2}$. In particular, $||Pv_0|| \geq ||v_0|| - ||v_0 - Pv_0|| \geq 1 - \frac{1}{2} = \frac{1}{2}$. Thus, using the first lemma, for all $g \in K$ we have:
$$||\rho(g)Pv_0 - Pv_0|| = ||P(\rho(g)v_0 - v_0)|| \leq ||\rho(g)v_0 - v_0|| < \frac{1}{2}\epsilon_2 \leq \epsilon_2 ||Pv_0||.$$
Therefore $H^N$ contains $(\epsilon_2, K)$-invariant vectors. Since $N$ acts trivially on $H^N$, the representation of $G$ on $H^N$ induces a unitary representation of $G/N$ on $H^N$. This is still unitary and it is strongly continuous, because $G/N \times H \to H$ is the map induced on the quotient by $G \times H \to H$, and $G/N$ is equippend with the quotient topology. Since $K_2 \subseteq \pi(K)$, this representation must contain $(\epsilon_2, K_2)$-invariant vectors. But $(\epsilon_2, K_2)$ is a Kazhdan pair for $G/N$, so $H^N$ contains a non-zero $G/N$-invariant vector. Going back to the representation of $G$ on $H^N$, this will also be a non-zero $G$-invariant vector.
\end{proof}

\begin{corollary*}
\begin{enumerate}[(i)]
\item If $G$ is a locally compact group and $N$ is a closed normal subgroup, then if $N$ and $G/N$ have property $(T)$, $G$ has property $(T)$.
\item The direct product of two groups with property $(T)$ has property $(T)$
\end{enumerate}
\end{corollary*}

\begin{proof}
$(i)$. Let $G$ be locally compact, $N$ a closed subgroup, $\pi : G \to G/N$ the canonical projection, $K \subseteq_c G/N$. Let $e \in U \subseteq C \subseteq G$, where $U$ is open and $C$ is compact (this exists by local compactness). Then $(\pi(Ux))_{x \in G}$ is an open cover of $K$, so there exist $x_1, \ldots, x_n \in G$ such that $K \subseteq \pi (\cup_i Ux_i)$. Finally, define $K' = \pi (\cup_i Cx_i) \cap \pi^{-1}(K)$. Then $\pi(K') = K$, $\pi (\cup_i Cx_i)$ is compact, and $\pi^{-1}(K)$ is closed ($G/N$ is Hausdorff since $N$ is closed, so $K$ is closed in $G/N$). Therefore $K'$ is a closed subset of a compact set, so it is compact. \\

$(ii)$. Two groups $G_1, G_2$ with property $(T)$ must be locally compact by definition. In the setting of the previous lemma, given the compact set $K_2 \in G_2$, the set $\{ e \} \times K_2$ is a product of compact sets, so it is compact.
\end{proof}

\end{comment}

\subsection{Explicit construction of expanders using property $(T)$}

\begin{comment}[Proposition 3.3.1, PP. 30-31]
\label{331}

Here we fill in the details of the proof of proposition 3.3.1. Throughout, we denote simply $|| \cdot ||$ for the $L^2$ norm. We start by noting that the statement should read "Let $\Gamma$ be a discrete Kazhdan group", from which finite generation follows (see comment \ref{compact_generation}). Discreteness is needed to apply 3.2.5: if the group is merely finitely generated but not discrete, there may be no Kazhdan pair $(\epsilon, K)$ where $K$ is finite. Indeed, this is the way the statement appears in \cite{T_exp} (lemma 4.8), which is cited by the author. \\

To see that we cannot get around this, notice that any finitely generated group with the trivial topology is locally compact but not discrete. If instead we also suppose that the group is Hausdorff, then local compactness and countability imply discreteness. Indeed, let $G = \{g_1, g_2, \ldots \}$ be a countable locally compact Hausdorff group. Suppose by contradiction that $G$ is not discrete, so that some singleton has empty interior. By translating this singleton, it follows that all singletons have empty interior. Then $G = \cup_{n \geq 1} \{g_n\}$ is a countable union of nowhere-dense sets, which contradicts the Baire category theorem for locally compact Hausdorff spaces. \\

Now we move on to the proof. In order to apply property $(T)$, we need the action of $\Gamma$ on $H$ to be well-defined, strongly continuous and unitary. First: write $\overline{x}$ for the image of $x \in \Gamma$ in $V = \Gamma / N$. If $\overline{x} = \overline{y}$, then there exists $n \in N$ such that $nx = y$, so $(\gamma f)(\overline{y}) = f(\overline{y \gamma}) = f(\overline{n x \gamma}) = f(\overline{x \gamma}) = (\gamma f)(\overline{x})$.

Secondly, to show that the representation is strongly continuous and unitary, we notice that the morphism $\Gamma \to GL(H)$ can be factored through $\Gamma/N$: if $\overline{\gamma} = \overline{\delta}$ there exists $n \in N$ such that $\gamma n = \delta$, so $(\delta f)(\overline{x}) = f(\overline{x \delta}) = f(\overline{x \gamma n}) = f(\overline{x \gamma}) = (\gamma f)(\overline{x})$, so $\gamma f = \delta f$. But then the relative action of $\Gamma / N$ on $H$ is just the right regular representation, which is unitary and strongly continuous. It immediately follows that the action of $\Gamma$ is unitary. Furthermore, $\Gamma \times H \to H$ can be factored through $\Gamma/N \times H$, and becomes the composite of two continuous functions, so it is also continuous. Therefore the action is strongly continuous (see comment \ref{strong_cont}). \\

Next, the author uses the fact that since the action of $\Gamma$ on $H_0$ contains no non-zero invariant vectors, then there must be an $\epsilon$ independent of $N$ such that for every $f \in H_0$ there exists $\gamma \in S$ such that $|| \gamma f - f|| > \epsilon ||f||$. This follows directly from the equivalent definition $(b)$ of property $(T)$ we proved in comment \ref{313}. The independence of $\epsilon$ on $N$ does not, however, follow from the previous chapters of the book: it does follow in the case of irreducible representations, as stated in 3.2.5, but the action of $\Gamma$ on $H_0$ is not necessarily irreducible. Indeed, since the action of $\Gamma$ on $H$ factors through that of $\Gamma / N$, $H_0$ is irreducible for $\Gamma$ if and only if it is irreducible for $\Gamma / N$. But the dimension of $H_0$ is $[\Gamma : N] - 1$, so if, for instance, $\Gamma / N$ is an abelian group of order at least 2, $H_0$ cannot be irreducible. \\

We prove that $||\gamma f - f||^2 = (a + b)^2|E_\gamma(A, B)|$. Indeed, $||\gamma f - f||^2 = \sum\limits_{x \in V} |f(x) - f(x \gamma)|^2$. If $x$ and $x \gamma$ are both in $A$ or both in $B$, then $x$ does not contribute to the sum. In instead they are in different parts, $x$ contributes $(a + b)^2$ to the sum. The number of such $x$ is the size of $\{(x, x \gamma) \in A \times B \} \cup \{(x, x \gamma) \in B \times A \} = E_\gamma(A, B)$. The equality follows. \\

The last point which we make more explicit is the inequality $|\partial A| \geq \frac{1}{2}|E_\gamma(A, B)|$. Now $|E_\gamma(A, B)| = |\{(a, a \gamma) \in A \times B \} \cup \{(b \gamma, b) \in A \times B \}|$. Assuming without loss of generality that the first set of this union is larger, we get $\frac{1}{2}|E_\gamma(A, B)| \leq |\{(a, a \gamma) \in A \times B \}|$. We can map this set to $\partial A$ naturally by $(a, a \gamma) \mapsto a \gamma$. This map is injective: if $a \gamma = a' \gamma$, then $a = a'$. Thus $|\{(a, a \gamma) \in A \times B \}| \leq |\partial A|$ and we conclude. \\

We notice something that is not taken into account in this proof, namely that the graphs obtained this way are not necessarily $k$-regular, where $k = |S|$. Indeed, it is possible that for some $\gamma, \delta \in S$ and some finite-index subgroup $N$, we have $\gamma \delta^{-1} \in N$, in which case, in $X(\Gamma / N, S)$, the edges $(a, a \gamma)$ and $(a, a \delta)$ are identified for all $a \in \Gamma / N$. Also, if $\gamma \in S \cap N$, then the edges $(a, a \gamma)$ are loops, for all $a \in \Gamma / N$. So eliminating loops and double edges, we can say that the family $X(\Gamma / N, S)$ is a family of $k(N)$-regular graphs, where $k(N) = |\overline{S \, \backslash \, N}| \leq |S| = k$. If, however, we want to be consistent with the definitions, we can just add edges to make each of the graphs $k$-regular. This is proven in detail in appendix \ref{special_case}.

\end{comment}

\begin{comment}[P. 31]
\label{quot_graph}

In the proof of proposition 3.3.3, the author states that if the family $X(n, p)$ were a family of expanders, then the family of quotient graphs $Z(n, p)$ would also be a family of expanders. First of all, the $Z(n, p)$ thus defined are not regular, but as we have seen time and again that is not really an issue (see appendix \ref{regularization}). But the way this statement is formulated, it seems that given any family of expanders (of bounded degree), any family of quotient graphs is also a family of expanders (of bounded degree). This is however false in general. \\

The case we are interested in is one where any partition defining the quotients is equitable, that is, each part has the same number of vertices. In that case, the quotient graphs satisfy the expanding property for the same constant $c$. More precisely, suppose we have a family $(X_{n_i})_{i \geq 1}$, where $X_{n_i}$ is an $(n_i, k, c)$-expander, and each $X_{n_i}$ is partitioned into sets of the same size $s(n_i)$. Let $m_i := \frac{n_i}{s(n_i)}$ and $Z_{m_i}$ the quotient graphs of $X_{n_i}$ given by this partition. Then $Z_{m_i}$ has $m_i$ vertices. We claim that the $Z_{m_i}$ satisfy the expanding condition for the same constant $c$. 

Let $\pi : X_{n_i} \to Z_{m_i}$ be the projection, and for $S \subseteq X_{n_i}$ let $\overline{S} := \pi(S)$. Let $\overline{A} \subseteq Z_{m_i}$, where $A = \pi^{-1}(\overline{A})$, so that $|A| = |\overline{A}|s(n_i)$. We need to show that $|\partial \overline{A}| \geq c(1 - \frac{|\overline{A}|}{m_i})|\overline{A}|$, where $\partial \overline{A}$ is the set of neighbours of $\overline{A}$ in $Z_{m_i}$. Let $\partial A$ be the set of neighbours of $A$ in $X_{n_i}$. Then clearly $\overline{\partial A} \subseteq \partial \overline{A}$, so  $|\partial A| \leq |\pi^{-1} (\partial \overline{A})| = s(n_i)|\partial \overline{A}|$. Thus: 
$$|\partial \overline{A}| \geq \frac{1}{s(n_i)}|\partial A| \geq \frac{c}{s(n_i)}(1 - \frac{|A|}{n_i})|A| = \frac{c}{s(n_i)}(1 - \frac{|\overline{A}|s(n_i)}{n_i})|\overline{A}|s(n_i) = c(1 - \frac{|\overline{A}|}{m_i})|\overline{A}|.$$
This gives the desired expanding property, and we conclude. \\

However, even when restricting to such families of quotient graphs, the degree of the quotient graphs is not necessarily bounded. Here is an example, which was given in an answer on StackExchange \cite{stack_graph}. Let $X_n$ be an $(n \times n)$ grid graph wrapping around its sides. More explicitly, the set vertices is $\mathbb{Z}/n \mathbb{Z} \times \mathbb{Z}/n \mathbb{Z}$, and $(i, j)$ is connected to $(i, j \pm 1)$ and $(i \pm 1, j)$. Thus $X_n$ is a 4-regular graph on $n^2$ vertices. We define a partition of $X_n$ where the parts are indexed by $\mathbb{Z}/n \mathbb{Z}$, and $(i, j)$ is in part $(i + \frac{j(j + 1)}{2} \mod n)$. Then if $i \neq i'$, it is clear that $(i, j)$ and $(i', j)$ are in different parts, so it follows that this partition is equitable. Let $Y_n$ be the quotient graphs obtained from this partition.

We claim that the $Y_n$ are complete graphs. Indeed, let $p, q \in \mathbb{Z}/ n \mathbb{Z}$ be distinct parts of the partition. We want to find an edge between $p$ and $q$. Pick $j$ such that $j + 1 = q - p \mod n$, and let $i$ be such that $(i + \frac{j(j + 1)}{2}) = p \mod n$. Then $(i, j)$ is in part $p$ by definition, and $(i, j + 1)$ is in part $q$, because
$$i + \frac{(j+1)(j + 2)}{2} = i + \frac{j(j + 1)}{2} + (j + 1) \equiv p + (q - p) \equiv q \mod n.$$
Since $(i, j)$ is connected to $(i, j + 1)$ in $X_n$ we conclude.

We have thus found an example of a family of quotients of regular graphs of the same degree under equitable partitions, such that the degree of these graphs gets arbitrarily large.

\end{comment}

\begin{comment}[Proposition 3.3.3, P. 31]
\label{333}

Here we fill in the details of the proof of proposition 3.3.3. \\

$SL_n(p)$ acts transitively on $\mathbb{F}_p^n \, \backslash \, \{0\}$, so if $G$ is the stabilizer of a non-zero vector, we get an isomorphism of $SL_n(p)$-sets $SL_n(p) / G \cong \mathbb{F}_p^n \, \backslash \, \{0\}$. Consider the partition induced by the quotient on the set $SL_n(p)$. Since $X(n, p)$ has as set of vertices $SL_n(p)$, this partition gives a quotient graph, which is isomorphic to the graph $Z(n, p)$ on $\mathbb{F}_p^n \, \backslash \, \{0\}$, where $\alpha \in \mathbb{F}_p^n \, \backslash \, \{0\}$ is connected to $A_n^\pm(\alpha)$ and $B_n^\pm(\alpha)$. By comment \ref{quot_graph}, since all cosets have size $|G|$, if $(X(n, p))_{n \geq 1}$ were a family of expanders, then $(Z(n, p))_{n \geq 1}$ would be a family of graphs satisfying the expanding property. Next, we prove that this is not the case; that is, for all $c > 0$ and for $n \geq n(c)$ large enough, there exists $Y_n \subseteq \mathbb{F}_p^n \, \backslash \, \{ 0 \}$ such that $|\partial Y_n| < c(1 - \frac{|Y_n|}{(p^n - 1)})|Y_n|$ in the graph $Z(n, p)$. \\

Let $n \geq 5$, and $\{e_1, \ldots, e_n\}$ the standard basis of $\mathbb{F}_p^n$. Define $Y_n := \{e_3, \ldots, e_{[\frac{n}{2}]} \}$. Notice that in the book $Y_n$ starts at $e_1$, but this does not work (at least, the (in)equalities given below do not hold). In fact, this is the way this set is defined in \cite{Lubotzky} (corollary 4.4), which is cited by the author.

Then $A_n^\pm(Y_n) = Y_n$, while
$$|B_n(Y_n) \Delta Y_n| = |\{e_2, e_{[\frac{n}{2}]} \}| \leq 2 = 2 \frac{|Y_n|}{|Y_n|} \leq \frac{2}{[\frac{n}{2}]} \cdot |Y_n| \leq \frac{5}{n} \cdot |Y_n|.$$
The last inequality follows from the fact that $n \geq 5$. Similarly, $|B_n^{-1}(Y_n) \Delta Y_n| = |B_n(B_n^{-1} (Y_n) \Delta Y_n)| = |Y_n \Delta B_n(Y_n)| \leq \frac{5}{n} \cdot |Y_n|$. Therefore, in $Z(n, p)$, we have:
$$|\partial Y_n| = |(A_n^\pm (Y_n) \cup B_n^\pm (Y_n)) \Delta Y_n| = |B_n^\pm (Y_n) \Delta Y_n| \leq |B_n(Y_n) \Delta Y_n| + |B_n^{-1}(Y_n) \Delta Y_n| \leq \frac{10}{n}|Y_n|.$$
Thus, given $c > 0$, if $n(1 - \frac{|Y_n|}{p^n - 1}) > \frac{10}{c}$ (which can be traslated as $n \geq n(c)$ since the left-hand side is increasing and tends to $\infty$), then:
$$|\partial Y_n| \leq \frac{10}{n}|Y_n| < c(1 - \frac{|Y_n|}{p^n - 1})|Y_n|,$$
which is what we wanted to prove.
 
\end{comment}

\begin{comment}[Remark 3.3.4, P.32]

In order to adapt the proof of 3.3.1 to this case, we need to consider alternatively left and right cosets in order to make everything well-defined. Indeed, looking at how we provided the details in comment \ref{331}, and using the same notation, we see that we need $V = \Gamma / N$ to be the set of right cosets $\{Nx : x \in \Gamma \}$, while $\Gamma \to GL(H)$ is constant on left cosets $\{ xN : x \in \Gamma \}$. We also need to make sure that the action is unitary, which can be done directly, and strongly continuous, for which we use the same argument by factoring $\Gamma \times H \to H$ through the set of left cosets $\Gamma / N \times H$. We skip the details.

\end{comment}

\begin{comment}[Proposition 3.3.5, P. 32]
\label{335}

Here we explain why $\{\sigma_1, \ldots, \sigma_4 \}$ is a generating set for $G := \mathbb{Z}^2 \rtimes SL_2(\mathbb{Z})$. We are considering $G$ as a group of affine transformations of $\mathbb{Z}^2$, and since the $\sigma_i$ are affine transformations of $\mathbb{Z}^2$, we can see them naturally as elements of $\mathbb{Z}^2 \rtimes SL_2(\mathbb{Z})$. Under this identification, $\sigma_1 = ((1, 0)^T, I_2)$, $\sigma_2 = ((0, 1)^T, I_2)$, $\sigma_3 = (0, M_3)$ and $\sigma_4 = (0, M_4)$, where 
$
M_3 =
\left( 
\begin{array}{cc}
1 & 0 \\
1 & 1
\end{array}
\right)
$
 and 
$
M_4 =
\left( 
\begin{array}{cc}
0 & 1 \\
-1 & 0
\end{array}
\right)
$. So the $\sigma_i$ are actually elements of $G$. It is clear that $\sigma_1$ and $\sigma_2$ generate $\mathbb{Z}^2 \leq G$. So we only need to show that $\sigma_3$ and $\sigma_4$ generate $SL_2(\mathbb{Z}) \leq G$, which is equivalent to the fact that $M_3$ and $M_4$ generate $SL_2(\mathbb{Z})$. Indeed, $M_3^T = A_2$ and $M_4 = B_2$, where $A_n, B_n$ are the generators of $SL_n(\mathbb{Z})$ given in example 3.3.2 (in that example it is stated for $n \geq 3$, because it is only then that $SL_n(\mathbb{Z})$ has property $(T)$, but the fact that these two matrices are generators is also true for $n = 2$). This implies that $A_2^T = M_3$ and $B_2^T = M_4^T = M_4^{-1}$ also generate $SL_2(\mathbb{Z})$.

\end{comment}

\begin{comment}[Proposition 3.3.7, P. 33]
\label{337}

Here we fill in the details of the proof of proposition 3.3.7. We start by noticing that also here the author is talking about discrete finitely generated groups (see comment \ref{331}). Also, it is necessary to suppose that the family of finite-index subgroups $\mathcal{L}$ is such that the indices are unbounded. This is used towards the end of the proof, and it is clear why it is a necessary condition: a family of graphs with a bounded number of vertices has finitely many isomorphism classes, and so it is a family of expanders in a trivial way. There are a few imprecisions throughout the proof: there are many typos, some estimations are unnecessary and ignoring them the calculations are cleaner, and it is not proven that the almost-invariant subset found at the end is not the whole of $\Gamma / N$ (if this were the case, the proposition would not tell us anything new); so we rewrite it. \\

Let $\Gamma$ be a discrete amenable group generated by a finite set $S$, and let $\mathcal{L}$ be a family of finite-index subgroups of $\Gamma$ such that $\{ [\Gamma : N] : N \in \mathcal{L} \} \subseteq \mathbb{N}$ is unbounded. Suppose without loss of generality that $S$ is symmetric, i.e., $S = S^{-1}$ (indeed, in order for $X(\Gamma / N, S)$ to be a well-defined undirected graph, we need $S$ to be symmetric: see proposition 3.3.1). Let $\epsilon > 0$. By the definition of a discrete amenable group, there exists $A \subset_f \Gamma$ such that $|A \Delta sA| < \epsilon |A|$ for all $s \in S$. \\

Fix $N \in \mathcal{L}$ and define $\varphi : \Gamma /N \to \mathbb{N} : X \mapsto |A \cap X|$. Then $||\varphi||_1 = \sum\limits_{X \in \Gamma /N} |\varphi(X)| = |A|$ (since all the cosets of $N$ are disjoint, seen as subsets of $\Gamma$) and
$$||\varphi - s\varphi||_1 = \sum\limits_{X \in \Gamma /N} |\varphi(X) - \varphi(s^{-1} X)| = \sum\limits_{X \in \Gamma/N} ||A \cap X| - |A \cap s^{-1} X|| = $$
$$ = \sum\limits_{X \in \Gamma/N} ||A \cap X| - |sA \cap X|| \leq \sum\limits_{X \in \Gamma /N} |(A \cap X) \, \backslash \, (sA \cap X)| + |(sA \cap X) \, \backslash \, (A \cap X)| = $$
$$ = \sum\limits_{X \in \Gamma /N} |(A \cap X) \Delta (sA \cap X)| \leq |A \Delta sA| < \epsilon |A| = \epsilon ||\varphi||_1.$$
Thus, $||\varphi -s\varphi||_1 < \epsilon ||\varphi||_1$. This provides an "almost invariant function", which we want to use to find an "almost invariant subset". \\

For $j \in \mathbb{N}$, define $B_j := \{X \in \Gamma /N : \varphi(X) \geq j\}$. Then if $s\in S$, we have $s B_j  = \{ X \in \Gamma /N : \varphi(s^{-1} X) \geq j \}$, so:
$$\sum\limits_{j \geq 1} |B_j \Delta sB_j| = \sum\limits_{j \geq 1} \sum\limits_{X \in \Gamma /N} \chi_{B_j \Delta sB_j}(X) = \sum\limits_{X \in \Gamma /N} \sum\limits_{j \geq 1} \chi_{B_j \Delta sB_j}(X) = $$
$$ = \sum\limits_{X \in \Gamma /N} \sum\limits_{j \geq 1} \chi \left( \{ \varphi(X) < j \leq \varphi(s^{-1}X) \} \, \cup \, \{ \varphi(s^{-1}X) < j \leq \varphi(X) \} \right) = $$
$$ = \sum\limits_{X \in \Gamma /N} |\varphi(X) - \varphi(s^{-1}X)| = ||\varphi - s\varphi||_1 < \epsilon ||\varphi||_1 = \epsilon |A|.$$
Next, for $s \in S$ define $J_s := \{ j \geq 1 : |B_j \Delta sB_j| \geq \epsilon^{\frac{1}{2}}|S||B_j| \} \subseteq \mathbb{N}$. So, by the previous inequality:
$$|S|\epsilon^{\frac{1}{2}} \sum\limits_{j \in J_s} |B_j| \leq \sum\limits_{j \in J_s} |B_j \Delta sB_j| < \epsilon |A|.$$
Thus:
$$\sum\limits_{j \in J_s} |B_j| < \frac{\epsilon^{\frac{1}{2}}|A|}{|S|}.$$
But
$$\sum\limits_{j \geq 1} |B_j| = \sum\limits_{j \geq 1} \sum\limits_{X \in G/N} \chi_{B_j}(X) = \sum\limits_{X \in G/N} \sum\limits_{j \geq 1} \chi_{B_j}(X) = \sum\limits_{X \in G/N} \varphi(X) = ||\varphi||_1 = |A|.$$
Therefore the two previous equations give:
$$\sum\limits_{j \in J_s} |B_j| < \frac{\epsilon^{\frac{1}{2}}}{|S|} \left(\sum\limits_{j \geq 1} |B_j| \right)$$
and so:
$$\sum\limits_{s \in S} \sum\limits_{j \in J_s} |B_j| < \epsilon^{\frac{1}{2}} \left(\sum\limits_{j \geq 1} |B_j| \right).$$
This means that there exists some $j_0 \geq 1$ such that $|B_{j_0}| \neq 0$ and $j_0 \notin \cup_{s \in S} J_s$. So $B_{j_0} \neq \emptyset$ and $|B_{j_0} \Delta sB_{j_0}| < \epsilon^{\frac{1}{2}} |S| |B_{j_0}|$ for all $s \in S$. If furthermore $N$ is such that $[\Gamma : N] > |A|$, then there exists $X \in \Gamma / N$ such that $A \cap X = \emptyset$, and so for all $j \geq 1$ we have $B_j \neq \Gamma / N$. In particular, $B_{j_0} \neq \Gamma / N$. \\

This shows that for all $N \in \mathcal{L}$ such that $[\Gamma : N] > |A|$ and for all $\delta > 0$ there exists $\emptyset \neq B \subsetneq \Gamma / N$ such that:
$$|(\bigcup\limits_{s \in S} sB) \Delta B| \leq \sum\limits_{s \in S} |B \Delta sB| < \delta |B|,$$
(take $\epsilon$ such that $\epsilon^{\frac{1}{2}} |S|^2 < \delta$ and $B$ the corresponding $B_{j_0} \neq \emptyset$). Now let $c > 0$ and consider the graph $X(\Gamma / N, S)$ for $N \in \mathcal{L}$ such that $[ \Gamma : N] > |A|$. Let $\delta \leq c\frac{1}{[\Gamma : N]}$. Let $B$ as in the previous statement for this $\delta$. Then in this graph we have:
$$|\partial B| = |( \bigcup\limits_{s \in S} sB ) \Delta B| < \delta |B| \leq c\frac{1}{[\Gamma : N]} |B| \leq c(1 - \frac{|B|}{[\Gamma : N]})|B|,$$
where the last equality follows from the fact that $B \neq \Gamma / N$. Therefore $X(\Gamma / N, S)$ is not an expander with respect to this constant $c$. Since $c$ was arbitrary, we conclude that $\{ X(\Gamma / N, S) : N \in \mathcal{L} \}$ is not a family of expanders.
\end{comment}

\subsection{Solution of the Ruziewicz problem for $S^n$, $n \geq 4$}

In this subsection, all finitely generated groups must be discrete in order for the proofs to work. As we have mentioned in comment \ref{331}, if a group is countable, locally compact and Hausdorff, then it must be discrete. This is the case for the finitely generated groups $\Gamma$ we are interested in, that is, subgroups of $SO(n + 1)$. Since $SO(n + 1)$ is Hausdorff, $\Gamma$ is automatically Hausdorff. As for local compactness, this seems to be assumed for all groups for which we talk about unitary representations, as stated at the beginning of this chapter.

\begin{comment}[Proposition 3.4.1, PP. 34-35]

Here we fill in the details of the proof of proposition 3.4.1. Let $\Gamma$ be a discrete subgroup of $SO(n + 1)$ generated by a finite set $S$. $\Gamma$ acts on $S^n$, which leads to a strongly continuous unitary representation $\rho'$ on $L^2(S^n)$ with a subrepresentation $\rho$ on $L_0^2(S^n) = \{ f \in L^2(S^n) : \int_{S^n} f d \lambda = 0 \}$. Suppose that $\rho$ does not weakly contain the trivial representation. We need to show that then the Lebesgue integral is the unique invariant mean on $L^\infty(S^n)$. For the rest of the proof, we will simply write $L^\infty, L^1, L^2$ and $L_0^2$, omitting the $(S^n)$. \\

Suppose that $m \in (L^\infty)^*$ is an invariant mean. We want to show that $m$ is the Lebesgue integral. Recall that there is an isometric isomorphism $L^\infty \cong (L^1)^*$, where the linear functional corresponding to $F \in L^\infty$ is $L^1 \to \mathbb{R} : f \mapsto \int_{S^n} F(x)f(x) d \lambda(x)$. Therefore we can see $m$ as an element of $(L^1)^{**}$. Recall that if $V$ is any vector space, then the canonical linear map $V \to V^{**}$ is a homeomorphism onto its dense image, where $V$ is equipped with the weak topology and $V^{**}$ with the weak-* topology. In our case, $L^1$ with the weak topology is dense in $(L_1)^{**}$ with the weak-* topology, where for an element $f \in L^1$ the associated linear functional in $(L^1)^{**} = (L^\infty)^*$ is $L^\infty \to \mathbb{R} : F \mapsto \int_{S^n} f(x)F(x) d \lambda(x)$. \\

Therefore there exists a net $\{ f_\alpha \} \subset L^1$ converging to $m$ in the weak-* topology. Since $m$ is positive on positive functions, we may choose the $f_\alpha$ so that $\int_{S^n} f_\alpha(x)F(x) d \lambda(x)$ is positive whenever $F$ is. Up to equivalence, this means that we can choose the $f_\alpha$ to be positive. Also, $m(\chi_{S^n}) = 1$, so we may choose the $f_\alpha$ so that the associated linear functional evaluated at $\chi_{S^n}$ equals one. This means that $\int_{S^n} f_\alpha(x) \chi_{S^n}(x) d\lambda(x) = \int_{S^n} f_\alpha(x) d \lambda(x) = 1$. Finally, since $m$ is $\Gamma$-invariant, for all $\gamma \in S$ we have $\gamma m - m = 0$, so $\gamma f_\alpha - f_\alpha$ goes to zero in the weak-* topology of $(L^1)^{**}$. Since $\gamma f_\alpha - f_\alpha$ and 0 are elements of $L^1$, this means that $\gamma f_\alpha - f_\alpha \to 0$ in the weak topology of $L^1$. \\

Now in a locally convex vector space, the weak and strong closures of convex sets coincide \cite[Theorem 3.12]{Rudin}. The above tells us that 0 is in the weak closure of $\{\gamma f_\alpha - f_\alpha \}$ in $L^1$, which is a locally convex vector space. Consider the convex hull $K = \{ \sum_{i = 1}^n \lambda_i f_{\alpha_i} : \lambda \geq 0, \, \sum_{i = 1}^n \lambda_i = 1 \}$. Then if $g_\beta \in K$, we still have $g_\beta \geq 0$ and $\int_{S^n} g_\beta d \lambda = 1$. Now $\gamma g_\beta - g_\beta$ is a convex combination of the $\gamma f_\alpha - f_\alpha$, and we get all possible convex combinations this way, so $\{ \gamma g_\beta - g_\beta \}$ is convex. Since 0 is in the weak closure of $\{ \gamma f_\alpha - f_\alpha \}$, it follows that 0 is in the strong closure of $\{ \gamma g_\beta - g_\beta \}$. Finally, since $f_\alpha \to m$, picking convex combinations with larger and larger $\alpha$, we still get $g_\beta \to m$ in the weak-* topology of $(L^1)^{**}$.

We conclude that there exists a net $\{ g_\beta \} \subset L_1$ such that we still have $g_\beta \to m$ in the weak-* topology of $(L^1)^{**}$, $g_\beta \geq 0$ and $\int_{S^n} g_\beta = 1$; and furthermore $\gamma g_\beta - g_\beta \to 0$ strongly in $L^1$, i.e., with respect to the $L^1$ norm: $||\gamma g_\beta - g_\beta||_1 \to 0$ in $\mathbb{R}$. \\

Let $F_\beta = \sqrt{g_\beta}$, so $F_\beta \in L^2$. Since $F_\beta \geq 0$, we have $F_\beta(\gamma^{-1} x) - F_\beta(x) \leq F_\beta(\gamma^{-1} x) + F_\beta(x)$. Thus, for all $\gamma \in S$ we have:
$$||\gamma F_\beta - F_\beta||_2^2 = \int_{S^n} (F_\beta(\gamma^{-1}x) - F_\beta(x))^2 d \lambda(x) \leq \int_{S^n} (F_\beta(\gamma^{-1}x) - F_\beta(x))(F_\beta(\gamma^{-1}x) + F_\beta(x)) d\lambda(x) =$$
$$= \int_{S_n} (F_\beta(\gamma^{-1} x)^2 - F_\beta(x)^2) d \lambda(x) = \int_{S_n} (g_\beta(\gamma^{-1} x) - g_\beta(x)) d \lambda(x) \leq $$ $$\leq \int_{S_n} |g_\beta(\gamma^{-1} x) - g_\beta(x)| d \lambda(x) = ||\gamma g_\beta - g_\beta||_1 \to 0.$$
Also, note that $||F_\beta||_2^2 = ||g_\beta||_1 = 1$. So the $F_\beta$ are vectors of norm 1 in $L^2$ which are almost-invariant. However $L^2 = \mathbb{R}\chi_{S^n} \oplus L_0^2$, and by hypothesis the latter does not weakly contain the trivial representation, so it does not contain almost invariant vectors for $S$. This implies that $F_\beta$ converges, for the $L^2$ norm, to a vector in $\mathbb{R}\chi_{S^n}$. Since the $F_\beta$ all have norm 1, we conclude that $F_\beta \to \chi_{S_n}$ for the $L^2$ norm. Therefore, using the Hölder inequality:
$$||g_\beta - \chi_{S^n}||_1 = \int_{S^n} |F_\beta^2(x) - 1| d \lambda(x) = \int_{S^n} |F_\beta(x) - 1||F_\beta(x) + 1| d \lambda(x) \leq $$
$$\leq ||F_\beta(x) - \chi_{S^n}||_2 \cdot ||F_\beta(x) + \chi_{S^n}||_2 \leq ||F_\beta - \chi_{S^n}||_2 \cdot 2 \to 0.$$
Therefore $(g_\beta)$ converges to $\chi_{S^n}$ in the $L^1$ norm, so it does so also in the weak topology, and so also in the weak-* topology when we see the $g_\beta$ as elements of $(L^1)^{**}$. However, $(g_\beta)$ converges to $m$ in the weak-* topology as well, and $(L^1)^{**}$ with the weak-* topology is Hausdorff, so $m = \chi_{S^n}$ as elements of $(L^1)^{**}$. But $\chi_{S^n} \in (L^1)^{**} = (L^\infty)^*$ is integration against $\chi_{S^n}$ so it is just the Lebesgue integral. So $m$ is in fact the Lebesgue integral, which is what we wanted to prove.

\end{comment}

\begin{comment}[Theorem 3.4.2, P. 35]
\label{342}

Here we explain why if $\Gamma$ is a finitely generated dense subgroup of $S^n$, then its representation on $L_0^2(S^n)$ does not contain an invariant function, or equivalently, any function in $L^2(S^n)$ which is invariant under $\Gamma$ is constant. In the book, Lubotzky says that it is because the action of $\Gamma$ on $S^n$ is ergodic. So we will define an ergodic action, prove a few equivalent characterizations and prove that the action of $\Gamma$ on $S^n$ is ergodic. \\

Let $G$ be a group acting on a probability measure space $(X, \mu)$. That is, $X$ is a measure space with $\mu(X) = 1$ and the action is measure-preserving: if $A \subseteq X$ is measurable, then $gA$ is measurable and $\mu(gA) = \mu(A)$, for all $g \in G$. The action is \textbf{ergodic} if for any measurable set $A$ such that $\mu(gA \Delta A) = 0$ for all $g \in G$, we have $\mu(A) = 0$ or 1. \\

The following proposition \cite[Proposition 2.5]{ergodic} characterizes ergodic actions of countable groups:

\begin{proposition*}
Let $G$ be a countable group acting on a probability measure space $(X, \mu)$. The following are equivalent:
\begin{enumerate}[(a)]
\item For any measurable set such that $\mu(gA \Delta A) = 0$ for all $g \in G$, we have $\mu(A) = 0$ or 1.
\item For any measurable set $A$ such that $gA = A$ for all $g \in G$, we have $\mu(A) = 0$ or 1.
\item For all sets $A, B \in X$ of positive measure, there exists $g \in G$ such that $\mu(g A \cap B) > 0$
\end{enumerate}
\end{proposition*}

\begin{proof}
$(a) \Rightarrow (b)$ is trivial. \\

$(b) \Rightarrow (c)$. Let $A, B \subseteq X$ be measurable sets of positive measure. Then $A' = \cup_{g \in G} gA$ satisfies $gA' = A'$ for all $g \in G$, so by $(b)$ it has measure 0 or 1. Since it contains $A$, which has positive measure, it must have measure 1. Thus $\mu(A' \cap B) = \mu(B) > 0$. Since $G$ is countable, at least one of the sets $gA \cap B$ must have positive measure. \\

$(c) \Rightarrow (a)$. Let $A \subseteq X$ be a subset with $0 < \mu(A) < 1$. Then by $(c)$ there exists $g \in G$ such that $\mu(gA \cap (X \, \backslash \, A)) > 0$. So $\mu(gA \Delta A) > 0$.
\end{proof}

The following proposition \cite[Proposition 2.7]{ergodic} gives yet another characterization of ergodic actions of countable groups, which is the one we are most interested in:

\begin{proposition*}
Let $G$ be a countable group acting on a probability measure space $(X, \mu)$. Then $G$ also acts on $L^2(X) = \mathbb{R}\chi_X \oplus L_0^2(X)$, where $L_0^2(X) = \{ f \in L^2(X) : \int_X f d \mu = 0 \}$. Then the action  of $G$ on $X$ is ergodic if and only if the action of $G$ on $L_0^2(X)$ has no non-zero invariant vector.
\end{proposition*}

\begin{proof}
Suppose that that the action is not ergodic, so let $A$ be a $G$-invariant measurable set with $0 < \mu(A) < 1$. Then $\chi_A - \mu(A) \chi_X$ is $G$-invariant, non-zero and in $L_0^2(X)$. \\

Suppose that $f$ is a non-zero $G$-invariant function in $L_0^2(X)$. Then there exists $D \subset \mathbb{R}$ such that $0 < \mu(f^{-1} (D)) < 1$. Now for all $g \in G$, if $x \in gf^{-1}(D) \Delta f^{-1}(D)$, then $|f(x) - gf(x)| > 0$, since one of $f(x), gf(x) = f(g^{-1}x)$ is in $D$ and the other is not. Therefore $gf^{-1}(D) \Delta f^{-1}(D) \subseteq \cup_{n \geq 1} \{ x \in X : |f(x) - gf(x)| > \frac{1}{n} \}$. Each set of the union has measure zero since $||f - gf||_2 = 0$, so $\mu(gf^{-1}(D) \Delta f^{-1}(D)) = 0$. This is true for all $g \in G$, so it violates $(c)$ in the previous proposition, so the action is not ergodic.
\end{proof}

Going back to our context, to show that the action of $\Gamma$ on $L^2_0(S^n)$ does not contain a non-zero invariant vector, by virtue of the previous proposition, we need to show that the action of $\Gamma$ on $S^n$ is ergodic. We assume for the moment that the action of $SO(n + 1)$ on $S^n$ is ergodic. Then the following lemma, whose proof was given in an answer on StackExchange \cite{stack_erg}, concludes the proof:

\begin{lemma*}
Let $\Gamma$ be a dense subgroup of $SO(n + 1)$. Then the action of $\Gamma$ on $S^n$ is ergodic.
\end{lemma*}

\begin{proof}

Let $A \subseteq S^n$ be a set such that for all $\gamma \in \Gamma$ we have $\mu(\gamma A \Delta A) = 0$. We need to show that $\mu(A) = 0$ or $1$. Consider $\chi_A \in L^2(X)$. Then $||\gamma \chi_A - \chi_A||_2 = 0$, so $\chi_A$ is $\Gamma$-invariant, since in $L^2$ we only consider functions up to equivalence. Since the action of $SO(n + 1)$ on $L^2(S^n)$ is strongly continuous, it follows by strong continuity that $\chi_A$ is $SO(n + 1)$-invariant, so for all $g \in SO(n + 1)$ we have $\mu(gA \Delta A) = 0$ as well. But we have seen that the action of $SO(n + 1)$ on $S^n$ is ergodic, so $\mu(A) = 0$ or $1$ and we conclude.
\end{proof}

The question of the ergodicity of the action of $SO(n + 1)$ on $S^n$ is a bit more subtle. One might be tempted to say that this follows directly from the transitivity of the action. However, this would be true using the equivalent definition of ergodicity with invariant sets for countable groups (see $(b)$ in the first proposition of this comment), but there is in general no similar statement for uncountable groups. So we need something different. I thank Ofir David for suggesting to me the following argument. We start with a lemma, whose proof is adapted from \cite[Proposition 8.6]{Einsiedler}.

\begin{lemma*}

Let $G$ be a compact metrizable group, and let $\mu$ denote the normalized left Haar measure on $G$. Let $A, B$ be Borel sets of positive measure. Let $O := \{g \in G : \mu(gA \cap B) > 0 \}$. Then $\mu(O) > 0$. Furthermore, the action of $G$ on itself by left translation is ergodic.

\end{lemma*}

\begin{proof}

First,
$$\mu(gA \cap B) = \int_G \chi_{gA}(h)\chi_B(h) d\mu(h).$$
Also, $\chi_{gA}(h) = \chi_{hA^{-1}}(g)$. Therefore, using Fubini's theorem ($G$ is separable):
$$\int_G \mu(gA \cap B) d\mu(g) = \int_G \int_G \chi_{gA}(h)\chi_B(h) d\mu(h) d\mu(g) = $$
$$ = \int_G \chi_B(h) \left( \int_G \chi_{hA^{-1}}(g) d \mu(g) \right) d\mu(h) = \int_G \chi_B(h) \mu(hA^{-1}) d\mu(A) = $$
$$ = \mu(A^{-1}) \int_G \chi_B(h) d\mu(h) = \mu(A^{-1})\mu(B).$$
Setting $B = G$, we obtain $\mu(A) = \mu(A^{-1})$. Then $\int_G \mu(gA \cap B) = \mu(A^{-1})\mu(B) = \mu(A)\mu(B) > 0$. The first statement follows. \\

Now suppose that $A \subseteq G$ is a Borel set such that $0 < \mu(A) < 1$. Then $\mu(G \, \backslash \, A) > 0$, so by the above there exists $g \in G$ such that $\mu(gA \cap (G \, \backslash A)) = \mu(gA \, \backslash \, A) > 0$. Therefore $\mu(gA \Delta A) > 0$. It follows that the action is ergodic.
\end{proof}

This allows to prove the desired statement, which concludes all parts of this comment:

\begin{corollary*}

The action of $SO(n + 1)$ on $S^n$ is ergodic.

\end{corollary*}

\begin{proof}

Let $G = SO(n + 1)$. We start by proving that the action is ergodic with respect to the Lebesgue measure on Borel sets. Since $G$ and $SO(n) < G$ are compact, and $G/SO(n) \cong S^n$, the left Haar measure $\mu$ on $G$ induces the pushforward measure $\mu_*$ on $S^n$ defined by: $\mu_*(A) = \mu(\pi^{-1}(A))$, where $\pi : G \to G/SO(n)$ is the canonical projection. Then $\mu_*$ is a countably additive measure defined on the Borel sets of $S^n$ which is rotation-invariant. By uniqueness of the Lebesgue measure (see appendix \ref{Lebesgue}), we conclude that $\mu_* = \lambda$. Now a few simple calculations show that $\pi^{-1}(A \Delta B) = \pi^{-1}(A) \Delta \pi^{-1}(B)$, and that $\pi^{-1}(gA) = g\pi^{-1}(A)$ for all Borel sets $A, B \subseteq S^n$. Therefore, if $A \subseteq S^n$ is a Borel set such that $\lambda(gA \Delta A) = 0$ for all $g \in G$, it follows that $\mu(g\pi^{-1}(A) \Delta \pi^{-1}(A)) = \mu(\pi^{-1}(gA \Delta A)) = \lambda(gA \Delta A) = 0$, for all $g \in G$. Since the action of $G$ on itself is ergodic, it follows that $\lambda(A) = \mu(\pi^{-1}(A)) = 0$ or 1, which is what we wanted to prove. \\

Now let $A$ be a Lebesgue measurable set such that $\lambda(gA \Delta A) = 0$ for all $g \in G$. Let $B$ be a Borel set such that $\lambda(A \Delta B) = 0$. Then $\lambda(gB \Delta B) \leq \lambda(gB \Delta gA) + \lambda(gA \Delta A) + \lambda(A \Delta B) = 0$, so $\lambda(A) = \lambda(B) = 0$ or 1.
\end{proof}

\end{comment}

\begin{comment}[PP. 36-37]

Here we fill in some details of the paragraph after corollary 3.4.6. \\

$SO(2)$ does not contain a finitely generated dense subgroup with property $(T)$. Indeed, suppose that $\Gamma$ is a finitely generated subgroup of $SO(2)$ with property $(T)$. Then $\Gamma$ is abelian, so amenable, and has property $(T)$, so it is compact. But $SO(2)$ is Hausdorff, so $\Gamma$ is closed. Furthermore, $\Gamma$ is countable and $SO(2)$ is uncountable, so $\overline{\Gamma} = \Gamma \neq SO(2)$. Thus $SO(2)$ does not contain a finitely generated dense subgroup with property $(T)$. \\

It is shown in the book that all finitely generated Kazhdan subgroups of $SU(2)$ are finite. We show how this implies the same statement for $SO(3)$ and $SO(4)$. Let us start with $SO(3)$. We have already seen in chapter 2, when constructing the free subgroup of $SO(3)$, that $H(\mathbb{R})^* / Z(H(\mathbb{R})^*) \cong SO(3)$, where $H(\mathbb{R})^*$ is the multiplicative group of real quaternions. Furthermore, it is well-known that $SU(2)$ is isomorphic to the group of quaternions of norm 1. So $SU(2) / Z(SU(2)) \cong SO(3)$ as well. Since $Z(SU(2)) = \{ \pm I_2 \}$, it follows that $SU(2)$ is a double cover of $SO(3)$. Therefore if $\Gamma$ is a finitely generated Kazhdan subgroup of $SO(3)$, then its lift in $SU(2)$ is still finitely generated, and it is Kazhdan since it is the extension of a finite group by a Kazhdan group (see comment \ref{T_ext}). Therefore it is finite by hypothesis, so $\Gamma$ is finite. \\

The claim for $SO(4)$ follows with the same argument, using the following isomorphism: $(SU(2) \times SU(2)) / \{\pm(I_2, I_2) \} \cong SO(4)$. We explain how this isomorphism is defined; for more details see \cite[1.8]{su2so4}.

Consider the map $\varphi : SU(2) \times SU(2) \to SO(4)$ defined as follows. We identify $SU(2)$ with the group of unit quaternions as usual, and $\mathbb{R}^4$ with the space of quaternions $\mathbb{H}$. If $v, w \in SU(2)$ then $\varphi(v, w)$ is the rotation defined by $\varphi(v, w)(q) = vqw^{-1}$, for all $q \in \mathbb{H}$. One needs to check that this defines a rotation, and all rotations are of this form. Then we get a surjective homomorphism $SU(2) \times SU(2) \to SO(4)$.

If $(v, w) \in \ker(\varphi)$, then in particular $v1w^{-1} = 1$, so $v = w$. But then restricting to imaginary quaternions, $\varphi(v, w) = \varphi(v, v)$ defines a rotation in $\mathbb{R}^3$ in the usual way, and we know that this is the identity if and only if $v = \pm 1$. So $\ker(\varphi) = \pm(I_2, I_2) \in SU(2)$.

\end{comment}

\pagebreak

\section{The Laplacian and its Eigenvalues}

\setcounter{subsection}{1}
\subsection{The combinatorial Laplacian}

\begin{comment}[Definition 4.2.1, P. 44]

Here we show that $\Delta$ does not depend on the choice of orientation of the edges (even though this is a consequence of proposition 4.2.2, the following argument gives another way of interpreting this fact). As in the book, we denote by $V$ the set of $n$ vertices and by $E$ the set of $m$ edges. By induction, it is enough to show that if $D$ and $C$ are the matrices of $d$ with respect to two orientations of the edges differing at the single edge $f \in E$, then $D^*D = C^*C$. In this case, for all $v \in V$ we have: $D_{e, v} = C_{e, v}$ for all $f \neq e \in E$ and $D_{f, v} = -C_{f, v}$. This means that $C = I_f D$, where $I_f$ is the $m \times m$ diagonal matrix with $-1$ in the line corresponding to $f$ and $1$ elsewhere. Then $C^* C = D^* I_f^* I_f D = D^* D$, since $I_f^* = I_f = I_f^{-1}$. \\

Next, we prove that the property $\langle f, \Delta g \rangle = \langle df, dg \rangle$ characterizes $\Delta$. This is used as the starting point of the proof of proposition 4.2.2.

Suppose that $T : L^2(V) \to L^2(V)$ is a linear operator satisfying $\langle f, Tg \rangle = \langle df, dg \rangle$ for all $f, g \in L^2(V)$. Fix $g \in G$. Then for all $f \in L^2(V)$ we have
$$0 = \langle df, dg \rangle - \langle df, dg \rangle = \langle f, \Delta g \rangle - \langle f, Tg \rangle = \langle f, (\Delta - T) g \rangle.$$
This implies that $(\Delta - T)g = 0$. This being true for all $g \in L^2(V)$, we conclude that $\Delta = T$. \\

Note that the identity $\langle f, \Delta g \rangle = \langle df, dg \rangle$ also show that $\Delta$ is self-adjoint and positive. Indeed:
$$\langle \Delta f, g \rangle = \overline{\langle g, \Delta f \rangle} = \overline{\langle dg, df \rangle} = \langle df, dg \rangle = \langle f, \Delta g \rangle.$$
And it is positive because $\langle f, \Delta f \rangle = \langle df, df \rangle = ||df||_2 \geq 0$. This is mentioned after the proof of proposition 4.2.2.

\end{comment}

\begin{comment}[Proposition 4.2.4, PP. 46-47]

There are a few typos in the proof, which are listed in the appendix. Apart from that, we make a few comments on the proof. \\

Here we make more explicit the proof of following equality:
$$\sum\limits_{i = 1}^r \sum\limits_{f(x) = \beta_i} \sum\limits_{f(y) < \beta_i} \delta_{xy} (f^2(x) - f^2(y)) = \sum\limits_{i = 1}^r \sum\limits_{e \in \overline{\partial}L_i} (\beta_i^2 - \beta_{i - 1}^2),$$
using the remark in the book. \\
$$\sum\limits_{i = 1}^r \sum\limits_{e \in \overline{\partial}L_i} (\beta_i^2 - \beta_{i - 1}^2) = \sum\limits_{i = 1}^r \sum\limits_{f(x) \geq \beta_i} \sum\limits_{f(y) < \beta_i} \delta_{xy} (\beta_i^2 - \beta_{i - 1}^2) = $$
$$ = \sum\limits_{i = 1}^r \sum\limits_{k = i}^r \sum\limits_{f(x) = \beta_k} \sum\limits_{j = 0}^{i - 1} \sum\limits_{f(y) = \beta_j} \delta_{xy} (\beta_i^2 - \beta_{i - 1}^2).$$
We then change the order of summation: if $i : 1 \to r$, $k : i \to r$, and $j : 0 \to (i - 1)$; then $k : 1 \to r$, $j : 0 \to (k - 1)$ and $i : (j + 1) \to k$. Thus:
$$ \cdots = \sum\limits_{k = 1}^r \sum\limits_{j = 0}^{k - 1} \sum\limits_{i = j + 1}^k \sum\limits_{f(x) = \beta_k} \sum\limits_{f(y) = \beta_j} \delta_{xy} (\beta_i^2 - \beta_{i - 1}^2) = $$
$$ = \sum\limits_{k = 1}^r \sum\limits_{j = 0}^{k - 1} \sum\limits_{f(x) = \beta_k} \sum\limits_{f(y) = \beta_j} \delta_{xy} \left( \sum\limits_{i = j + 1}^k (\beta_i^2 - \beta_{i - 1}^2) \right) = \sum\limits_{k = 1}^r \sum\limits_{j = 0}^{k - 1} \sum\limits_{f(x) = \beta_k} \sum\limits_{f(y) = \beta_j} \delta_{xy} (\beta_k^2 - \beta_j^2) = $$
$$ = \sum\limits_{k = 1}^r \sum\limits_{f(x) = \beta_k} \sum\limits_{f(y) < \beta_k} \delta_{xy} (f(x)^2 - f(y)^2),$$
which is what we wanted to prove (up to replacing $i$ with $k$).\\

One last comment. In the end, the author says "thus the above gives $A \geq h(X) \langle g, g \rangle \geq h(X) \langle f, f \rangle$". However, the equation above says precisely that $A \geq h(X) \langle f, f \rangle$, which is what we need, and a priori there is no reason why $A$ should be at least $h(X) \langle g, g \rangle$.

\end{comment}

\subsection{Eigenvalues, isoperimetric inequalities and representations}

\begin{comment}[Theorem 4.3.2, PP. 49-50]

Here we fill in a few details of the proof of theorem 4.3.2. \\

In the implication $(iv) \Rightarrow (i)$, the author states that $\Delta$ acts on $L^2(X_i) = \mathbb{C}[\Gamma / N_i]$ as multiplication from the right by $(k \cdot e - A)$, where $A = \sum\limits_{s \in S} s$. We explain this in more detail.

Let us start by making the identification more explicit. For clarity, denote $Q_i := \Gamma / N_i$. Define $\Phi : L^2(X_i) \to \mathbb{C}[Q_i] : f \mapsto \sum_{x \in Q_i} f(x) \cdot x$. This is an isomorphism of vector spaces, with inverse $\Phi^{-1} : \mathbb{C}[Q_i] \to L^2(X_i) : \sum\limits_{x \in Q_i} \alpha_x \cdot x \mapsto (\alpha : x \mapsto \alpha_x)$.

Now on $L^2(X_i)$, by proposition 4.2.2, the action of $\Delta$ is as follows: $\Delta f (x) = k f(x) - \sum\limits_{y \in Q_i} \delta_{xy} f(y)$. Recall that $x, y \in Q_i$ are connected in $X_i$ if and only if there exists $s \in S$ such that $xs = y$ (or equivalently $ys = x$, by symmetry of the set $S$).

On the other hand on $\mathbb{C}[Q_i]$, multiplication by $(k \cdot e - A)$ is as follows:
$$\left( \sum\limits_{x \in Q_i} \alpha_x \cdot x \right) \left(k \cdot e - \sum\limits_{s \in S} s \right) = \sum\limits_{x \in Q_i} \left( k \alpha_x \cdot x - \sum\limits_{s \in S} \alpha_x \cdot xs \right) = \sum\limits_{x \in Q_i} \left( k \alpha_x \cdot x - \sum\limits_{y \in Q_i} \delta_{xy} \alpha_x \cdot y \right) = $$
$$ = \sum\limits_{x \in Q_i} \left( k \alpha_x - \sum\limits_{y \in Q_i} \delta_{xy} \alpha_y \right) \cdot x = \sum\limits_{x \in Q_i} \Delta(\Phi^{-1}(\sum\limits_{z \in Q_i} \alpha_z \cdot z))(x) \cdot x = \Phi(\Delta(\Phi^{-1} ( \sum\limits_{z \in Q_i} \alpha_z \cdot z))).$$
Thus, up to the identification of $L^2(X_i)$ and $\mathbb{C}[Q_i]$ under $\Phi$, the action of $\Delta$ and the right multiplication by $(k \cdot e - A)$ coincide.

\end{comment}

\begin{comment}[Examples 4.3.3, PP. 50-52]
\label{433}

Here we fill in a few details of the examples 4.3.3. \\

B. Here the author gives as an example of group without property $(\tau)$: any infinite finitely generated residually finite amenable group. Now, by 3.3.7, we know that if $\Gamma$ is a finitely generated amenable group and $L$ is a family of normal subgroups of arbitrarily large finite index, then the family of Cayley graphs of the quotients by the subgroups in $L$ is not a family of expanders. So by 4.3.2, finitely generated amenable groups with subgroups of arbitrarily large finite index do not have property $(\tau)$. Therefore to prove the claim in this example, it remains to show that if a group is infinite and residually finite, then it contains normal subgroups of arbitrarily large finite index. 

Let $\Gamma$ be an infinite residually finite group and let $H$ be a finite-index subgroup. Then $H$ is also infinite and residually finite, so it contains some proper finite-index subgroup $K$. This shows that given any $H \leq \Gamma$ of finite index, there exists $K \leq \Gamma$ of larger finite index. \\

Note that the above shows that residual finiteness is not needed to get examples of group without property $(\tau)$: it is enough to have a finitely generated amenable group with subgroups of arbitrarily large finite index. Here is an example of such a group which is not residually finite. Consider the wreath product $G := A_5 \wr \mathbb{Z}$. That is, $H := \oplus_{\mathbb{Z}} A_5 $ and $G = H \rtimes \mathbb{Z}$, where $\mathbb{Z}$ acts by shifting the coordinates.

A direct sum of finite groups is amenable, since every finitely generated subgroup is finite, so amenable (see comment \ref{amenable_discrete}). This shows that $H$ is amenable. Since $\mathbb{Z}$ is amenable, and extensions of amenable-by-amenable groups are amenable, $G$ is amenable. $G$ also contains subgroups of arbitrarily large finite index, because $\mathbb{Z}$ does, and if $F \leq \mathbb{Z}$ has finite index in $\mathbb{Z}$, then $H \rtimes F \leq G$ has the same index in $G$. However, $G$ is a finitely generated group which is not residually finite \cite[Proposition 2.6.5]{res_fin}. \\

E. In this example, the author uses the fact that if $G$ is a locally compact group and $\Gamma \leq G$ is a dense subgroup with property $(T)$, then $G$ also has property $(T)$. Indeed, let $\rho$ be a strongly continuous unitary representation of $G$ and suppose that $\rho$ weakly contains the trivial representation. Then $\rho|_\Gamma$ also weakly contains the trivial representation. Since $\Gamma$ has property $(T)$, this implies that there exists a non-zero vector $v$ that is $\Gamma$-invariant under $\rho$. By the strong continuity of $\rho$, this is also $G$-invariant (see comment \ref{strong_cont}). Since $\rho$ was arbitrary, $G$ has property $(T)$.

\end{comment}

\setcounter{subsection}{4}
\subsection{Random walks on $k$-regular graphs; Ramanujan graphs}

Throughout this subsection, we denote simply $|| \cdot ||$ for the $L^2$-norm.

\begin{comment}[Proposition 4.5.1, P. 55]
\label{451}

Here we present a proof of proposition 4.5.1, following \cite[Proposition 3.1]{rw}. From the proof we will deduce a corollary \cite[Proposition 3.2]{rw} which is used in Lubotzky's proof of proposition 4.5.4. \\

Let $X = (V, E)$ be a connected $k$-regular graph, and $M : L^2(X) \to L^2(X)$ the operator defined by the random walk. We start by noticing a few facts about $M$, which are mentioned in the book in the paragraph preceding the proposition. The fundamental fact is that $M$ is a self-adjoint operator. Indeed, given $f, g \in L^2(X)$:
$$\langle f, M(g) \rangle = \sum\limits_{x \in V} f(x) \overline{M(g)(x)} = \sum\limits_{x \in V} f(x) \sum\limits_{y \sim x} \frac{1}{k} \overline{g(y)} = $$
$$ = \sum\limits_{y \in V} \overline{g(y)} \sum\limits_{x \sim y} \frac{1}{k} f(x) = \sum\limits_{y \in V} \overline{g(y)} M(f)(y) = \langle M(f), g \rangle.$$
From this, we deduce that $||M|| = \rho(M)$ and that $$||M|| \leq \sup_{x \in V} \sum\limits_{y \in V} |M_{xy}| = \sup_{x \in V} \sum\limits_{y \in V} M(\delta_x)(y) = \sup_{x \in V} \sum\limits_{y \in V} \sum\limits_{z \sim y} \frac{1}{k} \delta_x(z) = \sup_{x \in V} \sum\limits_{y \in V} \frac{\delta_{xy}}{k} = 1.$$
($||H|| = \rho(H)$ and the first inequality above are true for any self-adjoint operator $H$). \\

Fix the origin $e \in V$, and let $\delta = \delta_e$ for short. Start the random walk at $e$ at time 0. Now we are ready to prove.

\begin{proposition*}

Let $r_n$ be the probability of the random walk being at the origin at time $n$. Then $||M|| = \limsup_{n \to \infty} (r_n)^\frac{1}{n}$.

\begin{proof}

We start by noticing that $r_n = \langle \delta, M^n(\delta) \rangle$. More generally, $\langle \delta_x, M^n(\delta) \rangle = M^n(\delta)(x)$ is the probability of the random walk being at $x \in V$ at time $n$. We prove this by induction on $n$. For $n = 0$, this is clear. Now suppose that this holds up to $n$. Then for the random walk to be at $x$ at time $(n + 1)$, it must be at some $y \sim x$ at time $n$. Thus the probability of the random walk being at $x$ at time $(n + 1)$ is $\sum\limits_{y \sim x} \frac{1}{k} M^n(\delta)(y) = M^{n + 1}(\delta)(x)$. This yields: $r_n = \langle \delta, M^n \delta \rangle \leq ||M^n|| \leq ||M||^n$, so $||M|| \geq \limsup_{n \to \infty} (r_n)^\frac{1}{n}$. \\

The main part of the proof consists in finding a subsequence of $(r_n)^\frac{1}{n}$ converging to $||M||$, which proves the other inequality. We will show that $(r_{2^k})^{2^{-k}}$ works. Notice that for any self-adjoint operator $H$ and for any $f \in L^2(X)$ of norm at most 1, we have:
$$||H f||^2 = \langle Hf, Hf \rangle = \langle f, H^2f \rangle \leq ||H^2 f||.$$
Picking $H$ to be successive powers of $M$, we get:
$$||M(f)|| \leq ||M^2(f)||^\frac{1}{2} \leq ||M^4(f)||^\frac{1}{4} \leq \cdots \leq ||M^{2^k}(f)||^{2^{-k}} \leq \cdots$$
Now $||M^n(\delta)||^2 = \langle M^n(\delta), M^n(\delta) \rangle = \langle \delta, M^{2n}(\delta) \rangle = r_{2n}$. Therefore, picking $f = \delta$ in the previous inequalities, we deduce that $(r_{2^k})^{2^{-k}}$ is an increasing subsequence. Since it is bounded above by $||M||$, it must converge to a limit $\mu \leq ||M||$. We also deduce: $||M^{2^k}(\delta)|| = \sqrt{r_{2^{k + 1}}} \leq \sqrt{\mu^{2^{k + 1}}} = \mu^{2^k}$. Now our aim is to prove that $\mu \geq ||M||$. \\

Since $X$ is connected, for all $x \in V$ there is a path of finite length from $e$ to $x$, so there exists a time $K_x$ at which the random walk is at $x$ with probability $\beta_x := \langle \delta_x, M^{K_x}(\delta) \rangle > 0$. We claim that $||M^{2^k}(\delta_x)|| \leq \beta_x^{-1}\mu^{2^k}$: this inequality will be used in the end of the proof. To see this, notice that $M(f) \geq 0$ whenever $f \geq 0$, so if $f \leq g$, then $M(f) \leq M(g)$. Now $\delta_x \leq \beta_x^{-1}M^{K_x}(\delta)$, since the two sides are equal to 1 when evaluated at $x$, and otherwise the left-hand side is 0 and the right-hand side is non-negative. Therefore $M^{2^k}(\delta_x) \leq \beta_x^{-1}M^{K_x}M^{2^k}(\delta)$. Thus:
$$||M^{2^k}(\delta_x)|| \leq \beta_x^{-1}||M||^{K_x}||M^{2^k}(\delta)|| \leq \beta_x^{-1} ||M^{2^k}(\delta)|| \leq \beta_x^{-1} \mu^{2^k},$$
where in the second inequality we used that $||M|| \leq 1$. \\

Now we put all of this together. Let $f \in L^2(X)$ be of norm at most 1. We need to show that $||M(f)|| \leq \mu$. For all $\epsilon > 0$, let $f_\epsilon \in L^2(X)$ be of finite support and of norm at most 1 such that $||f - f_\epsilon|| \leq \epsilon$. Then $||M(f)|| \leq ||M(f_\epsilon)|| + ||M(f - f\epsilon)|| \leq ||M(f_\epsilon)|| + \epsilon$, where we used that $||M|| \leq 1$. Therefore it suffices to prove that $||M(f)|| \leq \mu$ for $f$ of finite support and norm at most 1. We can write $f = \sum\limits_{x \in X} f(x) \delta_x$. Let $\beta = \min\limits_{x \in \operatorname{supp}(f)} \beta_x > 0$. Then, for all $x \in \operatorname{supp}(f)$, we have $||M^{2^k}(\delta_x)|| \leq \beta_x^{-1} \mu^{2^k} \leq \beta^{-1} \mu^{2^k}$. Finally:
$$||M(f)|| \leq ||M^{2^k}(f)||^{2^{-k}} \leq \big( \sum\limits_{x \in \operatorname{supp}(f)} |f(x)| \cdot ||M^{2^k}(\delta_x)|| \big)^{2^{-k}} \leq \big( \beta^{-1} \sum\limits_{x \in \operatorname{supp}(f)} |f(x)| \big)^{2^{-k}} \mu.$$
Letting $k \to \infty$, we obtain $||M(f)|| \leq \mu$, and we conclude.
\end{proof}

\end{proposition*}

The proof of the previous proposition actually exhibits a subsequence $(r_n)^\frac{1}{n}$ that converges to $||M||$, and furthermore that converges monotonically to $||M||$: $n = 2^k$. The following corollary uses this to exhibit a larger subsequence. This will be used in the proof of proposition 4.5.4 (see comment \ref{454}):

\begin{corollary*}

$||M|| = \lim_{n \to \infty} (r_{2n})^\frac{1}{2n} = \lim_{n \to \infty} ||M^n(\delta)||^\frac{1}{n}$.

\end{corollary*}

\begin{proof}

We know that $\limsup_{n \to \infty} ||M^n(\delta)||^\frac{1}{n} \leq ||M||$, so we need $\liminf_{n \to \infty} ||M^n(\delta)||^\frac{1}{n} \geq ||M||$.

Let $\lambda < ||M||$. Since $||M^{2^k}(\delta)||^{2^{-k}}$ is increasing to $||M||$, for all $k \geq K$ large enough we have $\lambda^{2^k} \leq ||M^{2^k}(\delta)||$. Now let $n \geq 2^K$, and let $k$ be such that $2^k > n \geq 2^{k-1}$. Then $||M^{2^k}(\delta)|| = ||M^{2^k - n}M^n(\delta)|| \leq ||M||^{2^k - n}||M^n(\delta)||$. Thus:
$$||M^n(\delta)|| \geq ||M||^{n - 2^k}||M^{2^k}(\delta)|| \geq ||M||^{n - 2^k} \lambda^{2^k} \geq ||M||^{n - 2^k} \lambda^{2^k} \left( \frac{\lambda}{||M||} \right)^{2n - 2^k} = \left( \frac{\lambda^2}{||M||} \right)^n.$$
Letting $\lambda \to ||M||$, we obtain $\liminf_{n \to \infty} ||M^n(\delta)||^\frac{1}{n} \geq ||M||$, and we conclude.
\end{proof}

\end{comment}

\begin{comment}[Proposition 4.5.2, PP. 55-56]

Here we fill in the details of the proof of proposition 4.5.2. We want to show that if $X$ is the homogenous $k$-regular tree and $M$ is the operator defined by the random walk, then $||M|| = \frac{2 \sqrt{k - 1}}{k}$. This uses the fact established in the previous proposition that $M$ is the inverse of the radius of convergence of the return generating function, that is, of $R(z) = \sum\limits_{n \geq 0} r_n z^n$, where, for $n \geq 0$, $r_n$ is the probability of the random walk being at the origin $x_0$ at time $n$ having started at the origin at time 0. \\

We define the first return generating function $Q(z) = \sum\limits_{n \geq 0} r_n z^n$, where $q_0 = 0$ and, for $n \geq 1$, $q_n$ is the probability of the random walk being at the origin $x_0$ \textit{for the first time} at time $n$ having started at the origin at time 0. The link between these two function is given by $R(z) = 1 + Q(z) R(z)$. Indeed, if $n \geq 1$, then to go back at the origin at time $n$ we need to go back at the origin for the first time at time $i$, for some $0 < i \leq n$, and then go back to it in $j$ steps, for $i + j = n$. Therefore $r_n = \sum\limits_{i + j = n} q_i r_j$, which is the coefficient of $z^n$ in $1 + Q(z) R(z)$. So $R(z) = \frac{1}{1 - Q(z)}$. \\

Next we fix a vertex $y_0$ adjacent to $x_0$ and define a new generating function $T(z) = \sum\limits_{n \geq 0} t_n z^n$, where, for $n \geq 0$, $t_n$ is the probability of the random walk being at $x_0$ for the first time at time $n$ after having started at $y_0$ at time 0. Now for $n \geq 1$, to go from $x_0$ to $x_0$ in $n$ steps for the first time, we need to first go to a neighbour of $x_0$ in one step, and the from there to $x_0$ in $(n - 1)$ steps for the first time. Thus, by symmetry, $q_n = \sum\limits_{y \sim x_0} \frac{1}{k} t_{n - 1} = t_{n - 1}$, and so $Q(z) = zT(z)$. This gives yet another expression: $R(z) = \frac{1}{1 - zT(z)}$. \\

If instead we defined the first return generating function starting from a point $y_1$ at distance $m$ from $x_0$, instead of starting from $y_0$ which is at distance 1, we would get the generating function $T_m$. Since $X$ is a tree, there is a unique path $x_0 = v_0, v_1, \ldots, v_m = y_1$, so to go from $y_1$ to $x_0$ one must first go from $v_m$ to $v_{m - 1}$, then from $v_{m - 1}$ to $v_{m - 2}$, and so on until $x_0$. Thus, by an argument similar to the one relating $Q$ and $R$, we deduce that $T_m(z) = T(z)^m$.

This allows to find an expression for $T$. To go from $y_0$ to $x_0$, we have one chance out of $k$ to land right away. In the $(k - 1)$ other cases, we land at a neighbour of $y_0$ which is not $x_0$, so we are now at distance $2$ from $x_0$. So then the probability to go back is given by the coefficients of $z T(z)^2$, where the multiplication by $z$ accounts for the step lost in going from $y_0$ to this neighbour. Therefore $T(z) = \frac{z}{k} + \frac{k - 1}{k} z T(z)^2$. Solving for $T$, we obtain $T(z) = \frac{k \pm \sqrt{k^2 - 4(k - 1)z^2}}{2(k - 1) z}$. Notice that this makes sense when $k \geq 2$, which is the case since we are working with a connected graph, and when $z \neq 0$, which for the moment we can assume formally since $T(0) = 0$ (we will soon see that extending $T$ to $z = 0$ this way makes sense). \\

When $0 < z \leq 1$, we can give another interpretation for $T(z)$: it is the probability of ever reaching $x_0$, starting at $y_0$, where at each step we have probability $z$ of taking another step, and probability $(1 - z)$ of stopping. This shows that when $0 < z \leq 1$, since $T(z)$ is a probability, we have $0 \leq T(z) \leq 1$. But, for $0 < z \leq 1$:
$$\frac{k + \sqrt{k^2 - 4(k - 1)z^2}}{2(k - 1)z} \geq \frac{k + \sqrt{k^2}}{2(k - 1)z} = \frac{k}{(k - 1)z} > 1,$$
so $T(z)$ cannot take this value. Therefore we must take the minus sign in the expression for $T$. We conclude that $T(z) = \frac{k - \sqrt{k^2 - 4(k - 1)z^2}}{2(k - 1) z}$. \\

Now at the numerator we have the expression $k - \sqrt{k^2 - 4(k - 1)z^2}$, whose derivative with respect to the complex variable $z$ is $\frac{4(k - 1) z}{\sqrt{k^2 - 4(k - 1) z}}$. It follows that $z = 0$ is a zero of order two of the numerator, and of order one of the denominator, so $T$ has an analytic continuation around 0 up to the branch point, which occurs at the zero of the square root function, which is $z_0 = \frac{k}{2 \sqrt{k - 1}}$. Therefore the radius of convergence of $T$ is exactly $z_0 \in \mathbb{R}_{> 0}$. We have already seen that $R(z) = \frac{1}{1 - zT(z)}$, which is analytic around 0 up to whenever the denominator is non-zero and $T$ is analytic. The analyticity of $T$ has just been discussed. As for the denominator being non-zero, we notice that this occurs if and only if $1 = zT(z) = \frac{k - \sqrt{k^2 - 4(k - 1)z^2}}{2(k - 1)}$. Rearranging the terms yields $\sqrt{k^2 - 4(k - 1)z^2} = k - 2(k - 1) = 2 - k$. If $k \geq 3$, the right hand side is negative while the left-hand side is positive, for $0 \leq z < z_0$; if $k = 2$ then we have equality only at $\pm z_0$. In both cases, we deduce that the radius of convergence of $R$ is also $z_0$. We conclude that $||M|| = \frac{1}{z_0} = \frac{2 \sqrt{k - 1}}{k}$.

\end{comment}

\begin{comment}[Proposition 4.5.4, PP. 56-57]
\label{454}

Here we fill in the details of the proof of proposition 4.5.4. Let $X = (V, E)$ be a connected $k$-regular graph. If $X$ is bipartite, write $V = I \cup O$, else write $V = I = O$. We need to estimate the norm of $N$, which is the restriction of the Markov operator $M$ on the subspace $L^2_0(X) = \{f \in L^2(X) : \sum\limits_{x \in I} f(x) = \sum\limits_{x \in O} f(x) = 0 \}$. Recall that we already know the norm of the Markov operator $\tilde{M}$ of the homogenous $k$-regular tree $\tilde{X} = (\tilde{V}, \tilde{E})$ from proposition 4.5.2. \\

Suppose that the diameter of $X$ is at least $(2r + 2)$. Then we can choose $x_1, x_2 \in V$ of distance $(2r + 2)$. It follows from the fact that this distance is even that, if $X$ is bipartite, $x_1$ and $x_2$ are in the same part. (Notice that if we only ask for the diameter to be at least $(2r + 1)$, as in the book, then in the case where it is exactly $(2r + 1)$ and the graph is bipartite, whatever $x_1$ and $x_2$ we choose, they are going to lie in two different parts). Define $\delta_i = \chi_{\{x_i\}}$ and $f = \delta_1 - \delta_2$. Since $x_1$ and $x_2$ are in the same part, it follows that $f \in L^2_0(X)$. Also, $||f||^2 = 2$.

Since $x_1$ and $x_2$ are further than $(2r + 1)$ apart, if we denote by $B_r(x_i)$ the closed ball of radius $r$ around $x_i$, then we have $B_r(x_1) \cap B_r(x_2) = \emptyset$. By the definition of $M$, if $f \in L^2(X)$ is supported on $B_r(x_i)$, then $M(f)$ is supported on $B_{r + 1}(x_i)$, since $M(f)(x)$ can only be non-zero if $f(y) \neq 0$ for some neighbour $y$ of $x$. It follows by induction that $M^r(\delta_i)$ is supported on $B_r(x_i)$, where $M^r$ is the iterated composition of $M$ with itself. Then $||M^r(f)||^2 = ||M^r(\delta_1) - M^r(\delta_2)||^2 = ||M^r(\delta_1)||^2 + ||M^r(\delta_2)||^2$. \\

Next, we prove that $||\tilde{M}^r(\delta_e)|| \leq ||M^r(\delta_i)||$, where $e$ is the origin of $\tilde{X}$. Recall that since $X$ is regular, $\tilde{X}$ is a covering graph of $X$, meaning that there exists a surjective map $\pi : \tilde{X} \to X$ which is a local isomorphism, that is, if $y, y' \in \tilde{V}$, then $y \sim y'$ if and only if $\pi(y) \sim \pi(y')$. Fix a covering map $\pi$ such that $\pi(e) = x_i$. We will prove that for all $x \in V$:
$$(*): M^r(\delta_i)(x) = \sum\limits_{\pi(y) = x} \tilde{M}^r(\delta_e)(y).$$
We have already seen that $M^r(\delta_i)$ is supported on $B_r(x_i)$. The same argument shows that $\tilde{M}^r(\delta_e)$ is supported on $B_r(e)$, which is finite; so the sum on the right-hand side of ($*$) is finite. Then it follows that:
$$||M^r(\delta_i)||^2 = \sum\limits_{x \in V} M^r(\delta_i)(x)^2 = \sum\limits_{x \in V} \big( \sum\limits_{ \pi(y) = x} \tilde{M}^r(\delta_e)(y) \, \big)^2 \geq $$
$$ \geq \sum\limits_{x \in V} \sum\limits_{\pi(y) = x} \tilde{M}^r(\delta_e)(y)^2 = \sum\limits_{y \in \tilde{V}} \tilde{M}^r(\delta_e)(y)^2 = ||\tilde{M}^r(\delta_e)||^2.$$
So we only have to prove ($*$). We proceed by induction on $r$. The base case $r = 0$ is clear: both sides of the equation are equal to 1 if $x = x_i$ and 0 otherwise. Now suppose that ($*$) holds up to $r$. Then: 
$$M^{r + 1}(\delta_i)(x) = \sum\limits_{x' \sim x} \frac{1}{k} M^r(\delta_i)(x') = \sum\limits_{x' \sim x} \sum\limits_{\pi(y') = x'} \frac{1}{k} \tilde{M}^r(\delta_e)(y').$$
Since $\pi$ is a surjective local isomorphism, $\pi(y') = x' \sim x$ if and only if there exists $y \in \tilde{V}$ such that $\pi(y) = x$ and $y' \sim y$. So:
$$\cdots = \sum\limits_{\pi(y) = x} \sum\limits_{y' \sim y} \frac{1}{k} \tilde{M}^r(\delta_e)(y') = \sum\limits_{\pi(y) = x} \tilde{M}^{r + 1}(\delta_e)(y),$$
so ($*$) holds for $(r + 1)$ and we conclude. \\

Putting all of this together:
$$2||N||^{2r} = ||N||^{2r} \cdot ||f||^2 \geq ||N^r(f)||^2 = ||M^r(f)||^2 = $$
$$ = ||M^r(\delta_1)||^2 + ||M^r(\delta_2)||^2 \geq 2||\tilde{M}^r(\delta_e)||^2.$$
Thus, $||N|| \geq ||\tilde{M}^r(\delta_e)||^{\frac{1}{r}}$. But when $r \to \infty$, the right-hand side of this inequality converges to $||\tilde{M}||$ (comment \ref{451}).

From this, the proposition follows: if $X_{n, k}$ is an infinite family of $k$-regular graphs, then their diameter goes to infinity, so if $N_n$ is, as above, the Markov chain operator associated to the random walk on $X_{n, k}$ restricted to $L^2_0(X_{n, k})$, then $\liminf_{n \to \infty} ||N_n|| \geq ||\tilde{M}|| = 2 \frac{\sqrt{k - 1}}{k}$.

\end{comment}

\pagebreak

\appendix

{\Huge{Appendix}}

\section{From graphs with bounded degree to regular graphs}
\label{regularization}

In the definition of expanding graphs at the beginning of the book (definition 1.1.1), the author states that we restrict ourselves to regular graphs, since they are those which appear in all examples and applications. Although it is true that all expanding graphs appearing in the book are regular \textit{multi}graphs, i.e., with multiple edges and/or loops, they are not always regular as simple graphs. This is not really important for our purposes: the only application we saw concerns superconcentrators and bounded concentrators, for which having expanders with bounded degree is enough. However, in order to be coherent with the definitions, we address how to get around those instances. \\

The two instances we encountered in which the expanders are of bounded degree and not regular are in remark 1.1.2 (ii) and in proposition 3.3.1. In the first case, when passing from a bi-expander to an expander, the regularity is lost. In the second case, we have a family of regular graphs but their degree is not necessarily always the same. We will give a general method for passing from expanders of bounded degree to regular expanders, but in the case of proposition 3.3.1 we can do even better (corollary \ref{k'k}). \\

The main idea is the following: if we have a graph which satisfies the expanding condition for some $c$, then adding edges to it cannot change that. Also, in an asymptotic setting like that of expanding graphs, adding a fixed number of vertices to each graph will not affect the expanding property when the size of the graphs gets large enough. Everything here is done quite explicitely, since all of this has no real interest if it cannot be turned in an algorithm. \\

I thank Dániel Korándi, who was my professor of Graph Theory at EPFL, for suggesting to me the statements of \ref{k_reg} and \ref{k'k}. 

\subsection{Regularization of graphs}

This subsection is devoted to proving that a graph with degree bounded by $k$ can always be made regular by adding at most $(k + 2)$ vertices. \\

The first step is to construct regular and almost-regular graphs. Recall that a $k$-almost-regular graph is one where each vertex has degree $k$ or $(k - 1)$. We will be looking for specific almost-regular graphs, so we define an \textbf{$(a, b, k)$-almost-regular} graph to be a graph on $n = (a + b)$ vertices where $a$ vertices have degree $k$ and the other $b$ have degree $(k - 1)$.

\begin{lemma}
\label{constr_reg}

Let $0 < k < n$ be integers. Then there exists a $k$-regular graph on $n$ vertices if and only if $nk$ is even. Furthermore, this graph can be chosen such that: if $k > 1$ then it is Hamiltonian, and if $k < (n - 1)$ is even, then there is a matching of size $[\frac{n}{2}]$ in the complement.
\end{lemma}

\begin{proof}
Necessity follows from the handshake lemma.

Let $X$ be a set of $n$ vertices arranged in a circle. Suppose that $k$ is even, and connect each vertex to its $k$ neighbours: $\frac{k}{2}$ clockwise and $\frac{k}{2}$ counter-clockwise. This gives the desired $k$-regular graph. Furthermore, suppose that $k < (n - 1)$. If $n$ is even, then each vertex is not connected to its antipode, so we have a perfect matching in the complement. If $n$ is odd, then each vertex is not connected to the two vertices around its antipode, call them left- and right-antipodes. By ignoring one vertex, connecting each vertex to its left-antipode yields a matching in the complement covering $(n - 1)$ vertices.

Suppose that $k$ is odd, then $n$ must be even. Construct a $(k - 1)$-regular graph as before. Then since $k < (n - 1)$ and $n$ is even, there is a perfect matching in the complement. Adding this gives a $k$-regular graph.

In both cases, the Hamilton cycle is around the circle.
\end{proof}

An example of this construction is shown in figure \ref{fig_constr_reg}.

\begin{figure}
\centering
\captionsetup{justification=centering}

\begin{tikzpicture}
\SetVertexNoLabel
\grCycle[RA = 3]{9}

\AssignVertexLabel{a}{$1$, $2$, $3$, $4$, $5$, $6$, $7$, $8$, $9$}

\EdgeFromOneToSel{a}{a}{0}{2, 7}
\EdgeFromOneToSel{a}{a}{2}{4}
\EdgeFromOneToSel{a}{a}{1}{3, 8}
\EdgeFromOneToSel{a}{a}{6}{4, 8}
\EdgeFromOneToSel{a}{a}{5}{3, 7}

\tikzstyle{EdgeStyle}=[dotted]
\EdgeFromOneToSel{a}{a}{0}{5}
\EdgeFromOneToSel{a}{a}{1}{6}
\EdgeFromOneToSel{a}{a}{2}{7}
\EdgeFromOneToSel{a}{a}{3}{8}

\end{tikzpicture}

\caption{A 4-regular graph on 9 vertices. The dotted lines indicate the large matching in the complement.\label{fig_constr_reg}}

\end{figure}
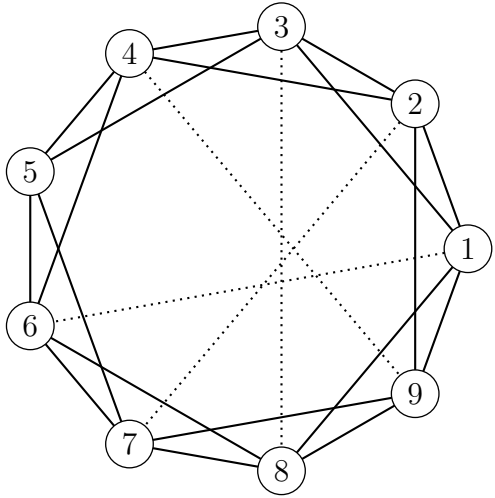

\begin{corollary}
\label{constr_alm_reg}

Let $0 < k < n$ and $a, b \geq 0$, where $(a + b) = n$. Then there exists an $(a, b, k)$-almost-regular graph on $n$ vertices if and only if $ak + b(k - 1)$ is even.
\end{corollary}

\begin{proof}
Once again, necessity follows from the handshake lemma. Also, we may assume that $a, b > 0$, otherwise this is just the previous lemma.

If $k = 1$, then $ak + b(k - 1) = a$ is even, so we can take an $n$-vertex graph with a matching covering $a$ vertices and we are done. So we may assume that $k > 1$. \\

Suppose that $nk$ is even. Since $k > 1$, we can construct a Hamiltonian $k$-regular graph on $n$ vertices as before. Now $nk - b = ak + b(k - 1)$ is even, so $b$ is even. Furthermore, $b < (n - 1)$, so there is a matching covering $b$ vertices in this graph (extracted from the Hamilton cycle). Removing this matching, we remain with $b$ vertices of degree $(k - 1)$ and $a$ vertices of degree $k$.

Suppose that $nk$ is odd. Then $(k - 1)$ and $n(k - 1)$ are even, so we can construct a $(k - 1)$-regular graph on $n$ vertices as before, such that there is a matching covering $(n - 1)$ vertices in the complement. Now $n(k - 1) + a = ak + b(k - 1)$ is even, so $a$ is even. Furthermore, $(k - 1) < (n - 1)$ and $a < (n - 1)$, so there is a matching covering $a$ vertices in the complement of this graph. Adding this matching, we remain with $a$ vertices of degree $k$ and $b$ vertices of degree $(k - 1)$.
\end{proof}

This allows us to prove the main result of this subsection:

\begin{proposition}
\label{k_reg}

Let $X = (V, E)$ be an $n$-vertex graph with degree bounded by $k < n$. Then there exists a $k$-regular graph $X'$ with at most $(n + k + 2)$ vertices that contains $X$ as a subgraph. In other words, we can make $X$ regular by adding at most $(k + 2)$ vertices and some edges.

If $k$ is even, then $X'$ may be chosen to have at most $(n + k + 1)$ vertices.

If $X$ is connected, then $X'$ may be chosen to be connected.
\end{proposition}

\begin{proof}
First, if the maximum degree is not $k$, we can just pick a vertex with maximum degree and connect it to some other vertices, so that now $X$ has maximum degree $k$. Then, whenever two vertices are not connected and have degree smaller than $k$, we connect them. This reduces to the following setting: $X$ has maximum degree $k$, and the non-empty subset $K$ of vertices whose degree is smaller than $k$ induces a complete subgraph. (This does not mean that all the vertices in $K$ have the same degree, since they are probably connected differently to the rest of $V$). Since the degree of each vertex is smaller than $k$, we conclude that $K$ has at most $k$ vertices. If $K = \{ v \}$ is an isolated singleton, we can add a complete graph $K_k$ and connect each vertex of it to $v$, and we are done. Therefore we may assume that $|K| \geq 2$, so each vertex in $K$ has degree at least 1. \\

Now let $l$ be the total degree missing, and $\delta$ the maximal degree missing. That is, $l = \sum\limits_{v \in K} (k - d(v))$ and $\delta$ is the maximal term of this sum. Suppose that there exists an integer $m$ such that:
\begin{enumerate}
\item $m \geq \delta$;
\item $m > k - [\frac{l}{m}] > 0$
\item $(n + m)k$ is even.
\end{enumerate}
We claim that then we can make $X$ regular by adding $m$ vertices. \\

Indeed, add a set $M$ of $m$ vertices to $X$. Number $K = \{ v_1, v_2, \ldots, v_t \}$, where we ordered the vertices from lowest to highest degree, and identify $M$ with $\mathbb{Z} / m \mathbb{Z}$. We need to add $l$ edges from $K$ to $M$, so let $L = \{e_1, \ldots, e_l\}$ be the set of edges we will add. We can partition $L$ as $L = L_1 \cup L_2 \cup \cdots \cup L_t$, where $L_j$ is the set of edges that will leave from $v_j$, so that $|L_j| = k - d(v_j)$. Up to reordering, we may assume that this partition is increasing, i.e., $L_1 = \{ e_1, \ldots, e_{|L_1|} \}$; $L_2 = \{e_{|L_1| + 1}, \ldots, e_{|L_1| + |L_2|} \}$, and so on. Now if $e_i \in L_j$, we add the edge $e_i$ between $v_j$ and $i \mod m \in M$. This does not create multiple edges, since $|L_j| \leq \delta \leq m$, and in the end all the $v_j$ have degree $k$. As for the vertices in $M$, each one is adjacent to all the edges whose index is in the corresponding class mod $m$, so each vertex is adjacent to $[\frac{l}{m}]$ vertices, while some are adjacent to $[\frac{l}{m}] + 1$. let $d := [\frac{l}{m}] + 1$ and say we have $a$ vertices of degree $d$ and $b$ vertices of degree $(d - 1)$. To make the degrees of the vertices in $M$ equal $k$, we need to add a $(b, a, k - (d - 1))$-almost-regular graph on $M$. \\

By corollary \ref{constr_alm_reg}, we can do this provided that $0 < k - (d - 1) = k - [\frac{l}{m}] < m$, and that $b(k - (d - 1)) + a(k - d)$ is even. The inequality has been assumed. So it remains to prove the congruence condition. Consider the graph we had while there were no edges between the vertices of $M$. In that graph, there were $n$ vertices of degree $k$, $a$ of degree $d$ and $b$ of degree $(d - 1)$, so $nk + ad + b(d - 1)$ is even. By hypothesis, $nk + mk$ is even, so $mk + ad + b(d - 1) = (a + b)k + ad + b(d - 1) = b(k + d - 1) + a(k + d)$ is also even. Switching some signs (which does not change the parity), we conclude. \\

Finally, we need to show that such an integer $m$ exists. We claim that $m = (k + 1), (k + 2)$ satisfy conditions 1 and 2 above. 1 follows from the definition of $\delta$. As for 2, first $k - [\frac{l}{m}] \leq k < (k + 1) < (k + 2)$. Secondly, $l \leq k(k - 1)$, since each vertex in $K$ has degree at least 1, and $m > k$; so:
$$k - [\frac{l}{m}] \geq k - \frac{l}{m} \geq k - \frac{k(k - 1)}{m} > k - \frac{k(k - 1)}{k} = k - (k - 1) = 1 > 0.$$
Then we choose whichever value makes $(n + m)k$ even, so that 3 is also satisfied. Notice that if $k$ is even, $(n + m)k$ is always even, so we can choose $m = (k + 1)$. \\

Now suppose that $X$ is connected. Then all the vertices of $X$ are contained in a connected component of $X'$. Call $X''$ the subgraph of $X'$ induced by this connected component. Then $X''$ is connected, $k$-regular, and contains $X$ as a subgraph.
\end{proof}

An example of this construction is shown in figure \ref{fig_reg}. \\

\begin{figure}
\centering
\captionsetup{justification=centering}

\begin{tikzpicture}

\begin{scope}[xshift = -5cm, yshift = 2cm]
\grCycle[Math, prefix=x, RA = 2]{5}
\end{scope}

\EdgeFromOneToSel{x}{x}{2}{4}
\EdgeFromOneToSel{x}{x}{1}{3}

\begin{scope}[yshift = 0.5cm, rotate=90]
\grPath[Math, prefix=v, RA=3, RS=0]{2}
\end{scope}

\EdgeFromOneToSel{x}{v}{0}{1}

\tikzstyle{EdgeStyle}=[dotted]

\begin{scope}[xshift = 4cm, rotate=90]
\grPath[Math, prefix=n, RA=2, RS=0]{3}
\end{scope}

\EdgeFromOneToSel{v}{n}{0}{0, 1}
\EdgeFromOneToSel{v}{n}{1}{2}

\tikzstyle{EdgeStyle}=[dotted, bend right]

\EdgeFromOneToSel{n}{n}{0}{2}

\end{tikzpicture}

\caption{We start with the graph on the $x_i, v_i$ with the black edges, where the maximal degree is $k = 3$. The $x_i$ have degree 3, while the $v_i$ do not. We add three new vertices $n_i$, and the dotted edges. The resulting graph is 3-regular.\label{fig_reg}}

\end{figure}
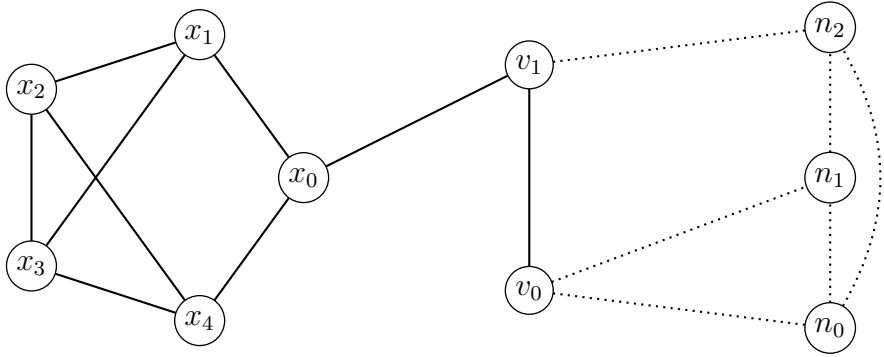

We close this section by showing that, in a general setting, we cannot do better.

\begin{lemma}
The $(k + 1), (k + 2)$ in proposition \ref{k_reg} are sharp. More precisely, there exist arbitrarily large graphs $X$ with degree bounded by $k > 0$ that cannot be made regular by adding less than $(k + 1)$ vertices, if $k$ is even, or $(k + 2)$ vertices, if $k$ is odd.
\end{lemma}

\begin{proof}
Suppose that $k$ is even. Then for all $n \geq 2$, we have that $0 < nk - 2 = (n - 1)k + (k - 2)$ is even. Let $X$ be an $n$-vertex graph with $(n - 1)$ vertices of degree $k$ and one of degree $(k - 2)$. We want to add vertices and edges to make $X$ $k$-regular. We must start by adding two new vertices to connect to the vertex of degree $(k - 2)$. This leaves us with $n$ vertices of degree $k$ and two of degree one or two (depending on whether we connect them or not). To get the degree of these two vertices to $k$, we must add at least $(k - 2)$ new vertices, which leaves us with at least $n + 2 + (k - 2) = n + k$ vertices. Suppose we add exactly $(k - 2)$ vertices. Then they must all be connected to the first two new vertices, which raises their degree to 2, and after that the most we can do is put a complete graph on these $(k - 2)$ vertices, which raises their degree to $2 + (k - 3) = (k - 1)$. Therefore we need at least one more vertex.

This shows that in this setting we need at least $(k + 1)$ new vertices. To show that such examples can get arbitrarily large, we need to check the existence of graphs with $n$ vertices of degree $k$ and one of degree $(k - 2)$, when $nk$ is even. This can be done the following way: pick an $k$-regular graph with $n$ vertices. If we find a "cherry", i.e., a triple of vertices $u, v, w$ such that $v$ is connected to both $u$ and $w$ and $u$ and $w$ are not connected, then we can eliminate the edges leaving from $v$ and add an edge between $u$ and $w$. This does not change the degree of $u$ and $w$, and it lowers the degree of $v$ to $(k - 2)$. So now it is enough to show that any connected graph that is not complete contains a cherry. Indeed, suppose that $X$ is a connected graph with no cherry and let $v, w \in V(X)$. By connectedness, there is a shortest path $v = v_1 v_2 \cdots v_l = w$. Suppose that $v$ and $w$ are not adjacent, so $l \geq 3$. Since $X$ has no cherry and $v_2$ is connected to both $v_1$ and $v_3$, there must be an edge connecting $v_1$ and $v_3$, which leads to a shorter path. This is a contradiction, so $v$ and $w$ must be connected and $X$ is complete. \\

Suppose that $k$ is odd. Then if $n \geq 3$ is odd, $nk - 1 = (n - 1)k + (k - 1)$ is even, so by corollary \ref{constr_alm_reg} there exists an $(n - 1, 1, k)$-almost-regular graph $X$. That is, $X$ is an $n$-vertex graph with $(n - 1)$ vertices of degree $k$ and one of degree $(k - 1)$. We want to add vertices and edges to $X$ to make it $k$-regular. A similar argumment as before shows that we need at least $(k + 1)$ new vertices. However, since $n$ and $k$ are odd, $(n + k + 1)k$ is also odd, so there cannot be a $k$-regular graph on $(n + k + 1)$ vertices. So we need at least $(k + 2)$ new vertices.
\end{proof}

\begin{remark}
By a result attributed to König, any graph with degree bounded by $k$ is an induced subgraph of a $k$-regular graph. Our result is weaker, since we started the construction by adding edges in the graph. However, the classic proof of König's theorem works by adding copies of the graph to raise the minimal degree by 1, which is not suitable for our situation.
\end{remark}

\subsection{Application to expanders}

We will use the definition of fixed-expander introduced at the beginning of section 1 of these comments. This is equivalent to that of expander up to a change of constant, and makes the calculations here less tedious.

\begin{proposition}
\label{k_exp}

Let $(X_n)_{n \geq 1}$ be a family of fixed-expanders for some constant $1 \geq c > 0$ with degree bounded by $k$. Let $X'_n$ be the connected $k$-regular graphs obtained from the $X_n$ as in proposition \ref{k_reg}, so that $X'_n$ has $n'$ vertices, where $n \leq n' \leq (n + k + 2)$. Then $\{ X'_n : n \geq 2\frac{(k + 3)^2}{c(k + 2)} \}$ is a family of $(n', k, \frac{c}{k + 3})$-expanders.
\end{proposition}

\begin{proof}
Fix $n$, and let $n'$ be the number of vertices in $X'_n$, so $0 \leq (n' - n) \leq (k + 2)$. Denote by $\partial'$ and $\partial$ the neighbours operators in $X'_n$ and $X_n$ respectively. Let $A$ be a set of vertices of $X'_n$ such that $|A| \leq \frac{n'}{2}$ and let $A_0 = A \cap X_n$. We will have to treat three situations differently.

First, if $A_0 = \emptyset$, then all the vertices of $A$ are new and $|A| \leq (k + 2)$. Since $X'_n$ is connected by hypothesis, $A$ has at least one neighbour. Thus, $|\partial' A| \geq 1 = \frac{1}{|A|}|A| \geq \frac{1}{k + 2}|A| \geq \frac{c}{k + 3}|A|.$ \\

Now suppose that $|A_0| \geq 1$. Suppose further that $|A_0| \leq \frac{n}{2}$, so $|\partial A_0| \geq c|A_0|$. We want to find $d > 0$ such that for all such $A$ we have $c|A_0| \geq d|A|$. Writing $|A| = |A_0| + t$, where $t \leq (n' - n) \leq (k + 2)$, this translates to $\frac{c}{d} \geq \frac{|A|}{|A_0|} = 1 + \frac{t}{|A_0|}$. But the right hand side is at most $1 + (k + 2) = (k + 3)$, so letting $d = \frac{c}{k + 3}$ works and we get
$|\partial' A| \geq |\partial A_0| \geq c|A_0| \geq \frac{c}{k + 3}|A|.$ \\

Finally, suppose that $|A_0| > \frac{n}{2}$. Let $A_1 \subseteq A_0$ be a subset of size $[\frac{n}{2}]$. Then
$$|A_0 \, \backslash \, A_1| = |A_0| - |A_1| \leq |A| - [\frac{n}{2}] \leq \frac{n + k + 2}{2} - \frac{n - 1}{2} = \frac{k + 3}{2}.$$
So we get:
$$|\partial' A| \geq |\partial A_0| \geq |(\partial A_1) \, \backslash \, A_0| = |(\partial A_1) \, \backslash \, (A_0 \, \backslash \, A_1)| \geq |\partial A_1| - |A_0 \, \backslash \, A_1| \geq c|A_1| - \frac{k + 3}{2} \geq $$
$$ \geq c\frac{n - 1}{2} - \frac{k + 3}{2} = c\frac{n + k + 2}{2} - \frac{(c + 1)(k + 3)}{2} \geq c|A| - (k + 3).$$
Now we want to find $d > 0$ such that $c|A| - (k + 3) \geq d|A|$. This translates to $d \leq c - \frac{k + 3}{|A|}$. But $|A| \geq |A_0| \geq \frac{n}{2}$, so
$$c - \frac{k + 3}{|A|} \geq c - \frac{k + 3}{\frac{n}{2}} \geq c - c\frac{k + 2}{k + 3} = \frac{c}{k + 3},$$
since by hypothesis $n \geq 2 \frac{(k + 3)^2}{c(k + 2)}$. Therefore choosing $d = \frac{c}{2}$, we get $|\partial' A| \geq \frac{c}{k + 3}|A|$. \\

In all cases, $|\partial' A| \geq \frac{c}{k + 3}|A|$, so we conclude.
\end{proof}

Notice that although these estimations are probably not optimal, we cannot hope to do much better for general subsets, since we have no control over the expanding properties of the edges that we have added. Still, we get a new family of expanders. \\

This result works also for bi-expanders, although the argument that follows does a bit worse for the degree and the constant of expansion. Start with a family of bi-expanders for a constant $c > 0$, with degree bounded by $k$. Add edges in order to ensure that this graph has a perfect matching, raising the bound on the degree to $(k + 1)$ (see comment \ref{match}). Apply the second direction of remark 1.1.2 (ii) (comment \ref{112}) to get a family of fixed-expanders for the same constant $c$ and degree bounded by $2k$. Then apply proposition \ref{k_exp} to get a family of $2k$-regular fixed-expanders for $c' = \frac{c}{2k + 3}$. Finally apply the first direction of remark 1.1.2 (ii), which preserves regularity, to get back a family of $(2k + 1)$-regular bi-expanders for the same $c'$. Thus we proved:

\begin{corollary}
Let $(X_n)_{n \geq 1}$ be a family of bi-expanders for some constant $1 \geq c > 0$ with degree bounded by $k$. Let $X'_n$ be the bipartite $(2k + 1)$-regular graphs obtained from the $X_n$ as above, so that each part of $X'_n$ has $n'$ vertices, where $n' \leq n \leq (n + (2k + 1) + 2) = (n + 2k + 3)$. Then $\{X'_n : n \geq 2\frac{((2k + 1) + 3)^2}{c((2k + 1) + 2)} = 8\frac{(k + 2)^2}{c(2k + 3)} \}$ is a family of $(n', 2k + 1, \frac{c}{2k + 3})$-bi-expanders.
\end{corollary}

\subsection{A special case}
\label{special_case}

In this subsection we address a special case that we encounter in the book, in which we can do better than in the general case. \\

The expander graphs encountered in proposition 3.3.1 are regular but not of the same degree (comment \ref{331}). In this case, to make them all of the same degree, we do not need to add any new vertex, which leaves the constant of expansion unchanged.

\begin{lemma}
\label{k'k}
 
Let $X = (V, E)$ be a $k'$-regular graph on $n$ vertices, and let $k > k'$ be such that $kn$ is even and $k \leq \frac{n}{2}$. Then we can add edges to $X$ to make it $k$-regular. If $k' = (k - 2)$, we can do this so that the added edges form a Hamilton cycle.
\end{lemma}

\begin{proof}
Let $\overline{X}$ be the complement of $X$. Then for every $v \in V$, its degree in $\overline{X}$ is $(n - 1) - k' \geq (n - 1) - (k - 1) = (n - k) \geq \frac{n}{2}$. Recall Dirac's theorem: if in an $n$-vertex graph each vertex has degree at least $\frac{n}{2}$, then it contains a Hamiltonian cycle. So $\overline{X}$ contains a Hamiltonian cycle $C$. If furthermore $n$ is even, then we can choose one out of two edges of $C$ to get a matching $M$ in $\overline{X}$. \\

Now adding $C$ to $X$ increases the degree of each vertex by 2, so if $k' \equiv k \mod 2$ we are done by induction, and otherwise we can make $X$ $k' = (k - 1)$-regular. But then $k'n$ is even, since $X$ has $n$ vertices and is $k'$ regular; so since $kn$ is even by hypothesis, $n$ must be even. Therefore we can add the matching $M$ to $X$ to make it $k$-regular.
\end{proof}

\begin{corollary}
\label{k'k_exp}

Let $(X_n)_{n \geq 1}$ be a family of $k_n$-regular expanders for some constant $c > 0$, and suppose that $k$ is such that $k \geq k_n$ for all $n$. Let $X'_n$ be the $k$-regular graphs obtained by adding edges to the $X_n$ as in the previous lemma, whenever possible. Then $\{ X'_n : n \geq 2k, \, kn \text{ even} \}$ is a family of expanders, for the same constant $c$.
\end{corollary}

\begin{remark}
Looking at the classic proof of Dirac's theorem \cite[Theorem 3.4.5]{Ore}, it is easy to see that it can be turned into an algorithm.
\end{remark}

\pagebreak

\section{The arithmetic of quaternions}
\label{quaternions}

This section is devoted to proving some facts about (Hurwitz) integral quaternions that are used in the proof of proposition 2.1.7 (PP. 9-11).

\subsection{Basic facts about $\tilde{H}(\mathbb{Z})$}

For a commutative ring $R$, define $H(R) := \{ \alpha = a_0 + a_1 i + a_2 j + a_3 k : a_i \in R \}$. This is a ring for component-wise addition and multiplication defined by $i^2 = j^2 = k^2 = ijk = -1$, and the rest determined by distributivity. Thus, the product of two elements is:
$$(a_0 + a_1 i + a_2 j + a_3 k)(b_0 + b_1 i + b_2 j + b_3 k) =$$
$$= (a_0b_0 - a_1b_1 - a_2b_2 - a_3b_3) + (a_0b_1 + a_1b_0 + a_2b_3 - a_3b_2)i +$$
$$+ (a_0b_2 + a_2b_0 + a_3b_1 - a_1b_3)j + (a_0b_3 + a_3b_0 + a_1b_2 - a_2b_1)k.$$
The conjugate of $\alpha = (a_0 + a_1 i + a_2 j + a_3 k)$ is $\overline{\alpha} = (a_0 - a_1 i - a_2 j - a_3 k)$. The norm of $\alpha$ is $N(\alpha) = \alpha \overline{\alpha} = \overline{\alpha}\alpha = (a_0^2 + a_1^2 + a_2^2 + a_3^3) \in R$. Generally the norm in the usual quaternions refers to the square root of this quantity, but we want to do arithmetic on integral quaternions so this will be more appropriate. We recall that $H(\mathbb{R})$ is a division ring, that the norm is multiplicative, and that $\overline{\alpha \beta} = \overline{\beta}\overline{\alpha}$.
 
The ring of integral quaternions is $H(\mathbb{Z})$. Define the element $f = \frac{1}{2}(1 + i + j + k)$. The ring of Hurwitz integral quaternions is $\tilde{H}(\mathbb{Z}) := \{a_0f + a_1 i + a_2 j + a_3 k : a_i \in \mathbb{Z} \}$. Notice that $\tilde{H}(\mathbb{Z})$ is the set of all quaternions whose coordinates are either all integers (if $a_0$ is even) or all half an odd integer (if $a_0$ is odd). In other words, $\tilde{H}(\mathbb{Z}) = H(\mathbb{Z}) \cup H(\mathbb{Z} + \frac{1}{2})$, where the second element of the union is only a set.

\begin{lemma}
\begin{enumerate}
\item $\tilde{H}(\mathbb{Z})$ is a ring.
\item The norm of an element of $\tilde{H}(\mathbb{Z})$ is always a positive integer.
\item $\tilde{H}(\mathbb{Z})$ has 24 units, which are exactly the elements with norm 1: the 8 integral units \sloppy
$\{ \pm 1, \pm i, \pm j, \pm k \}$ and the 16 non-integral ones, which are the ones of the form $\frac{1}{2}((\pm 1) + (\pm i) + (\pm j) + (\pm k))$.
\end{enumerate}
\end{lemma}

\begin{proof}
1. $\tilde{H}(\mathbb{Z})$ is clearly a group for addition. To prove that it is a ring, since we already know that $H(\mathbb{R})$ is a ring, we only have to show that the product of two elements in $\tilde{H}(\mathbb{Z})$ is still in $\tilde{H}(\mathbb{Z})$. By distributivity, it is enough to show it for the elements $f, i, j, k$. The products not involving $f$ are in $H(\mathbb{Z}) \subset \tilde{H}(\mathbb{Z})$. Since multiplication by $i, j, k$ just permutes the set $\{ \pm 1, \pm i, \pm j, \pm k \}$, all products of the form $if, fi, \ldots$ are in $\tilde{H}(\mathbb{Z})$. Finally, $f^2 = \frac{1}{2}(1 - i - j - k) \in \tilde{H}(\mathbb{Z})$. \\ 

2. The norm of an element of $H(\mathbb{Z})$ is a positive integer. Suppose that $\alpha \in \tilde{H}(\mathbb{Z})$ is not integral, say $\alpha = \frac{1}{2}(a_0 + a_1 i + a_2 j + a_3 k)$, where the $a_i$ are odd integers. Then $N(\alpha) = \frac{1}{4}(a_0^2 + a_1^2 + a_2^2 + a_3^3)$, which is a positive integer since the square of an odd number is congruent to $1 \mod 4$. \\

3. Let $\alpha \in \tilde{H}(\mathbb{Z})$. If $N(\alpha) = 1$, then $\overline{\alpha} = \alpha^{-1}$. If $\alpha$ is a unit, then since the norm is multiplicative, we have $N(\alpha^{\pm 1}) = N(\alpha)^{\pm 1} \in \mathbb{Z}_{\geq 0}$, so $N(\alpha) = 1$. If $\alpha \in H(\mathbb{Z})$, then $\alpha$ must have all coordinates equal to 0 but one which is $\pm 1$. Else, $\alpha$ must have all coordinates equal to $\pm \frac{1}{2}$.
\end{proof}

\begin{remark}
When working in $\tilde{H}(\mathbb{Z})$, we sometimes need to determine whether an element is integral. This is quite easy to do, since we only have to look at one coordinate, and it will be an integer if and only if all of them are, by the way this ring is defined.
\end{remark}

\subsection{Factorization in $\tilde{H}(\mathbb{Z})$}

The results and proofs in this subsection are taken from \cite{herstein}. \\

The reason we are interested in $\tilde{H}(\mathbb{Z})$ is that it allows to do a sort Euclidean division.

\begin{lemma}[Left division algorithm]

Let $\alpha, \beta \in \tilde{H}(\mathbb{Z})$, with $\beta \neq 0$. Then there exist $\gamma, \delta \in \tilde{H}(\mathbb{Z})$ such that $\alpha = \gamma \beta + \delta$ and $N(\delta) < N(\beta)$.

\end{lemma}

\begin{proof}
We start by showing the lemma for the special case in which $\beta = n \in \mathbb{Z}_{> 0}$. Let $\alpha = a_0f + a_1 i + a_2 j + a_3 k$ and $\gamma = x_0f + x_1 i + x_2 j + x_3 k$. We want to choose the $x_i$ so that $N(\alpha - \gamma n) < N(n) = n^2$. Now:
$$\alpha - \gamma n = \frac{1}{2} [ (a_0 - nx_0) + (a_0 + 2a_1 - n(x_0 + 2x_1))i + (a_0 + 2a_2 - n(x_0 + 2x_2))j + (a_0 + 2a_3 - n(x_0 + 2x_3))k ].$$

Let $nx_0$ be the multiple of $n$ that is closest to $a_0$. Then $|a_0 - nx_0| \leq \frac{1}{2}n$. For $i > 0$, let $2nx_i$ be the multiple of $2n$ that is closest to $a_0 + 2a_i - nx_0$. Then $|a_0 + 2a_i - nx_0 - 2nx_i| \leq n$. We conclude:
$$N(\alpha - \gamma n) = \frac{1}{4}\left( (a_0 - nx_0)^2 + \sum\limits_{i = 1}^3 (a_0 + 2a_i - n(x_0 + 2x_i))^2 \right) \leq \frac{1}{4} \left( \frac{n^2}{4} + 3n^2 \right) < n^2.$$

For the general case, let $\alpha$ and $\beta$ be as in the statement. Since $\beta \overline{\beta}$ is a positive integer, we use the previous part to find a $\gamma$ such that $N(\alpha \overline{\beta} - \gamma \beta \overline{\beta}) < N(\beta \overline{\beta})$. By multiplicativity $N(\alpha - \gamma \beta)N(\overline{\beta}) < N(\beta)N(\overline{\beta})$, and, since $\beta \neq 0$, it follows that $N(\alpha - \gamma \beta) < N(\beta)$.
\end{proof}

\begin{corollary}
$\tilde{H}(\mathbb{Z})$ is a left PID, meaning that all left ideals are principal.
\end{corollary}

\begin{proof}
The argument is the same as when proving that every Euclidean domain is a PID. Let $I$ be a left ideal of $\tilde{H}(\mathbb{Z})$, and choose $0 \neq x \in I$ of minimal norm. Let $\alpha \in I$, and choose $\gamma$ such that $\alpha = \gamma x + \delta$ with $N(\delta) < N(x)$. Since $\delta = \alpha - \gamma x \in I$, by the choice of $x$ we must have $\delta = 0$, so $\alpha = \gamma x \in \tilde{H}(\mathbb{Z})x$. Therefore $I = \tilde{H}(\mathbb{Z})x$.
\end{proof}

\subsection{The isomorphism $H(\mathbb{F}_p) \cong M_2(\mathbb{F}_p)$}
\label{iso}

There are many ways to prove the existence of such an isomorphism, most commonly by using classification theorems of algebras over finite fields, or of quaternion algebras. But, for our purposes, a ring isomorphism is enough, so here we present a proof which needs less general theory. \\

Note, however, that there is a much easier proof in the case $p \equiv 1 \mod 4$, which is actually the only case that is used in this book (proof of theorem 2.1.7). In that case, there exists a square root of -1 in $\mathbb{F}_p$, so we can represent $H(\mathbb{F}_p)$ in $M_2(\mathbb{F}_p)$ in the same way we represent $H(\mathbb{R})$ in $U(2)$. Then the fact that this is an isomorphism follows from the cardinality of these two rings being the same. This does not work in the general case, so we present a proof that covers that as well. \\

Recall that a ring $R$ is simple if its only two-sided ideals are $0$ and $R$. We will use two well-known theorems of Wedderburn:

\begin{theorem}[Wedderburn's theorem on simple rings]
\label{Wsr}
If $R$ is a simple ring with identity 1 and a minimal left ideal $M \neq 0$, then $R$ is isomorphic to the ring of $n \times n$ matrices over a division ring.
\end{theorem}

\begin{theorem}[Wedderburn's theorem on finite division rings]
\label{Wfdr}
If $R$ is a finite division ring, then it is a field (i.e., it is commutative).
\end{theorem}

A simple proof of theorem \ref{Wsr} by Henderson can be found in \cite{Wsr}. A simple proof of theorem \ref{Wfdr} by Witt can be found in \cite[Chapter 5]{Wfdr}. We will now use these theorems to prove the desired isomorphism.

\begin{theorem}
If $p$ is an odd prime, then $H(\mathbb{F}_p) \cong M_2(\mathbb{F}_p)$ as rings.
\end{theorem}

\begin{proof}
Assume for the moment that $H(\mathbb{F}_p)$ is a simple ring. Since it is finite, it has a minimal left ideal $M \neq 0$. So we can apply Wedderburn's theorem on simple rings to get that $H(\mathbb{F}_p) \cong M_n(D)$, for some division ring $D$. Now $p^4 = |H(\mathbb{F}_p)| = |M_n(D)| = |D|^{n^2}$, so necessarily $|D| = p^j$ for some integer $j \geq 1$. Then $4 = j n^2$, which is only possible if either $j = 4$ and $n = 1$ or $j = 1$ and $n = 2$. In the first case, $H(\mathbb{F}_p) \cong M_1(D) = D$, but $H(\mathbb{F}_p)$ is non-commutative so it cannot be a division ring by Wedderburn's theorem on finite division rings. In the second case, $|D| = p$, so $D \cong \mathbb{F}_p$ and we conclude. \\

Now we prove that $H(\mathbb{F}_p)$ is a simple ring. Let $0 \neq \alpha \in H(\mathbb{F}_p)$, and let $(\alpha)$ be the two-sided ideal generated by $\alpha$. We can naturally see $\alpha$ as an integral quaternion as well. We want to find a unit in $(\alpha)$. If $p \nmid N(\alpha)$, then $N(\alpha) \in (\alpha)$ is a unit and we are done, so suppose that $p | N(\alpha)$.

Notice that $N(x + y) = (x + y)(\overline{x} + \overline{y}) = N(x) + N(y) + x\overline{y} + \overline{x\overline{y}} = N(x) + N(y) + 2 Re(x\overline{y})$. In particular, $N(x \alpha y + z \alpha w) \equiv 2 Re((x \alpha y) \, (\overline{z \alpha w} )) \mod p$. We do some calculations. For $x \in H(\mathbb{F}_p)$, denote by $x_i$ its $i$ coordinate. Then writing $\alpha = (a_0 + a_1 i + a_2 j + a_3 k)$ we have:
$$(\alpha i \overline{\alpha})_i = ((a_0 + a_1 i + a_2 j + a_3 k)(a_1 + a_0 i + a_3 j - a_2 k))_i = (a_0^2 + a_1^2 - a_2^2 - a_3^2) = 2(a_0^2 + a_1^2) - N(\alpha).$$
Therefore, in $H(\mathbb{F}_p)$, we have $-N(i \alpha - \alpha i) = -2Re(i\alpha i \overline{\alpha}) = 4(a_0^2 + a_1^2)$, so $(a_0^2 + a_1^2), (a_2^2 + a_3^2) \in (\alpha)$. Similar calculations with $(\alpha j \overline{\alpha})_j$ and $(\alpha k \overline{\alpha})_k$ show that the sum of the squares of any two coordinates of $\alpha$ are in $(\alpha)$. Then $a_0^2 = \frac{1}{2}((a_0^2 + a_1^2) + (a_0^2 + a_2^2) - (a_1^2 + a_2^2)) \in (\alpha)$, and similarly for all other coordinates. Since $\alpha \neq 0$, at least one coordinate is a unit in $\mathbb{F}_p$, and so its square is as well. Therefore, $(\alpha)$ contains a unit, which concludes the proof.
\end{proof}

We conclude with a fact that is used in the proof of theorem 2.1.8.

\begin{corollary}
If $p$ is an odd prime, then $\tilde{H}(\mathbb{Z}/p\mathbb{Z}) \cong M_2(\mathbb{F}_p)$ as rings.
\end{corollary}

\begin{proof}
Since $p$ is odd, it makes sense to divide by 2, so $\tilde{H}(\mathbb{Z}/p\mathbb{Z})$ is a well-defined ring. Furthermore, there is a natural embedding $M_2(\mathbb{F}_p) \cong H(\mathbb{F}_p) = H(\mathbb{Z}/p\mathbb{Z}) \hookrightarrow \tilde{H}(\mathbb{Z}/p\mathbb{Z})$. Since these two rings have cardinality $p^4$, this is an isomorphism.
\end{proof}

\pagebreak

\section{Amenable actions and Tarski's theorem}
\label{Tarski}

This section is devoted to proving Tarski's theorem, and getting a better understanding of amenable actions on the way. Indeed, we will not only prove theorem 2.2.2, but also other equivalences. This is the approach that was taken in Monod's class \cite{Monod}, and we will follow it closely. The only difference being that in that lecture the matching problem was treated in a way that allowed to not mention graphs at all. Since graphs are one of the central topics of the book we are commenting, it seemed suitable to take a more graph-theoretic approach. \\

In all that follows, groups will always be discrete.

\subsection{Invariant means and amenable actions}

\begin{definition}
Let $X$ be a set. A \textbf{mean} on $X$ is a map $\mu : \mathcal{P}(X) \to [0, 1]$ such that:
\begin{enumerate}
\item $\mu(X) = 1$;
\item If $A, B \subseteq X$ are disjoint, then $\mu(A \sqcup B) = \mu(A) + \mu(B)$.
\end{enumerate}
We denote by $\mathcal{M}(X)$ the set of means on $X$.
\end{definition}

The reason we do not call this a measure is that measures are usually thought of as countably additive and come with a collection of measurable sets which is usually not the whole power set.

\begin{example}
\label{l1_mean}

If $\nu \in \ell^1(X), \, \nu \geq 0$ and $||\nu||_1 = 1$, then we can see $\nu$ as a mean by setting $\nu (A) = \sum\limits_{x \in A} \nu(x)$.
\end{example}

Before we define amenability, we prove an important fact about the space of means.

\begin{lemma}
\label{M_compact}

Let $\mathcal{M}(X)$ be equipped with the pointwise topology, i.e., the subspace topology of $[0, 1]^X$. Then $\mathcal{M}(X)$ is compact.
\end{lemma}

\begin{proof}
By Tychonoff's theorem, $[0, 1]^X$ is compact, so we only need to show that $\mathcal{M}(X)$ is a closed subset. But this is clear since we are only imposing closed conditions (equalities), thus $\mathcal{M}(X)$ will be an intersection of closed sets, so closed.
\end{proof}

We now move on to the key concept of this section.

\begin{definition}
Let $G$ be a group acting on a set $X$. A mean on $X$ is \textbf{$G$-invariant} if for all $g \in G$ and all $A \subseteq X$, we have $\mu(g^{-1} A) = \mu(A)$. If there exists a $G$-invariant mean, we say that the action is \textbf{amenable}.
\end{definition}

\begin{remark}
In this section we will only talk about amenability of actions, not of groups, but we mention here that a discrete group $G$ is said to be amenable if the action of $G$ on itself by left translation is amenable. Once we have proved that an action is amenable if and only if it satisfies the Følner condition, it will follow directly that this definition of amenable discrete groups is equivalent to the one given in Lubotzky's book.
\end{remark}

\subsection{Følner and Reiter conditions}

Here we introduce two analytic properties that will turn out to be equivalent to amenability: the Følner condition (F) and the Reiter condition (R). We will prove that (F) implies (R) and that (R) implies amenability.

\begin{definition}
Let $G$ be a group acting on a set $X$. The action satisfies the \textbf{Følner condition} if: \\
(F): For all $K \subseteq_f G$ and for all $\epsilon > 0$, there exists some $A \subseteq_f X$ such that for all $x \in K$: $|xA \Delta A| < \epsilon |A|$.
\end{definition}

Notice that this is the same as the definition of amenability of discrete groups we have in chapter 2, when $G$ acts on itself. \\

The next lemma proves a few equivalent conditions to (F). The proof that (F) and (F') are equivalent completes comment \ref{(F)}. The fact that (F'') is again equivalent will be used in the proof of Tarski's theorem.

\begin{lemma}
\label{(F)app}

Let $G$ be a group acting on a set $X$. Then the following are equivalent: \\

\textup{(F)} For all $K \subseteq_f G$ and for all $\epsilon > 0$, there exists some $A \subseteq_f X$ such that: $|xA \Delta A| < \epsilon |A|$ for all $x \in K$. \\

\textup{(F')} For all $K \subseteq_f G$ and for all $\epsilon > 0$, there exists some $A \subseteq_f X$ such that: $|KA \Delta A| < \epsilon |A|$. \\

\textup{(F'')} For all $K \subseteq_f G$ and for all $\epsilon > 0$, there exists some $A \subseteq_f X$ such that: $|KA| < (1 + \epsilon) |A|$. \\
\end{lemma}

\begin{proof}
(F) $\Rightarrow$ (F'). Let $K \subseteq_f G, \, \epsilon > 0$. Let $A \subseteq_f X$ be such that $|xA \Delta A| < \delta |A|$ for all $x \in K$, for some $\delta > 0$. Then
$$|KA \Delta A| = |\bigcup\limits_{x \in K} xA \Delta A| \leq \sum\limits_{x \in K} |xA \Delta A| < |K| \delta |A|.$$
Choosing $\delta = \frac{\epsilon}{|K|}$, we conclude. \\

(F') $\Rightarrow$ (F''). Let $K \subseteq_f G, \, \epsilon > 0$. Let $A \subseteq_f X$ be such that $|KA \Delta A| < \epsilon |A|$. Then
$$|KA| \leq |A| + |KA \Delta A| < (1 + \epsilon)|A|.$$

(F'') $\Rightarrow$ (F). Let $K \subseteq_f G, \, \epsilon > 0$. Without loss of generality, let $e \in K$. Let $A \subseteq_f X$ be such that $|KA| < (1 + \delta) |A|$ for some $\delta > 0$. Then for all $x \in K$:
$$|xA \cup A| = |(x \cup e) A| \leq |KA| < (1 + \delta) |A|;$$
$$|xA \cap A| = |xA| + |A| - |xA \cup A| > 2|A| - (1 + \delta)|A| = (1 - \delta)|A|;$$
$$|xA \Delta A| \leq |xA \cup A| - |xA \cap A| < (1 + \delta)|A| - (1 - \delta)|A| = 2\delta|A|.$$
Choosing $\delta = \frac{\epsilon}{2}$, we conclude.
\end{proof}

Now we move on to the second property discussed in this subsection. We will simply note $|| \cdot ||$ instead of $|| \cdot ||_1$ for the $\ell^1$-norm, since it is the only one that appears so there is no room for confusion.

\begin{definition}
Let $G$ be a group acting on a set $X$. Then $G$ acts naturally on $\ell^1(X)$ by permuting the coordinates: $(g \nu) (a) = \nu (g^{-1} a)$ for all $a \in X$. The action of $G$ on $X$ satisfies the \textbf{Reiter condition} if: \\
(R): For all $K \subseteq_f G$ and for all $\epsilon > 0$, there exists some $\nu \in \ell^1(X)$ such that for all $x \in K$: $||x \nu - \nu|| < \epsilon ||\nu||$.
\end{definition}

\begin{remark}
This condition is generally noted (R\textsubscript{1}), and we have the analogous (R\textsubscript{p}) for $1 \leq p < \infty$, by letting $G$ act on $\ell^p(X)$ and considering the $p$-norm. It turns out that these are all equivalent.
\end{remark}

Also in this case, we will need an equivalent condition.

\begin{lemma}
\label{(R)}
Let $G$ be a group acting on a set $X$. Then the following are equivalent: \\

\textup{(R)} For all $K \subseteq_f G$ and for all $\epsilon > 0$, there exists some $\nu \in \ell^1(X)$ such that: $||x\nu - \nu|| < \epsilon ||\nu||$ for all $x \in K$. \\

\textup{(R')} For all $K \subseteq_f G$ and for all $\epsilon > 0$, there exists some $\nu \in \ell^1(X)$ such that $\nu \geq 0, \, ||\nu|| = 1$ and: $||x\nu - \nu|| < \epsilon ||\nu||$ for all $x \in K$.
\end{lemma}

\begin{proof}
(R) $\Rightarrow$ (R'). Let $K \subseteq_f G, \, \epsilon > 0$. Let $\nu \in \ell^1(X)$ be such that $||x \nu - \nu|| < \epsilon ||\nu||$ for all $x \in K$. Since this inequality is strict, $||\nu|| \neq 0$, so we can normalize $\nu$ to get a vector of norm 1 satisfying the same condition. Then we take $|\nu| \geq 0$ and we have: $||x|\nu| - |\nu||| = |||x\nu| - |\nu||| \leq ||x\nu - \nu|| < \epsilon ||\nu||$ by the reverse triangle inequality.

The other direction is trivial.
\end{proof}

Finally we get to the key proposition of this subsection, whose statement was announced at the beginning.

\begin{proposition}
\label{equiv}

Let $G$ be a group acting on a set $X$. Then \textup{(F)} implies \textup{(R)} which implies amenability.
\end{proposition}

\begin{proof}
For the first implication, let $K \subseteq_f G, \, \epsilon > 0$. Let $A \subseteq_f X$ be such that $|xA \Delta A| < \epsilon |A|$ for all $x \in G$. Then it is easy to see that $||x \chi_A - \chi_A || = |xA \Delta A|$. \\

Now suppose that the action satisfies (R). Then by lemma \ref{(R)}, it satisfies (R'). So for all $K \subseteq_f G, \, \epsilon > 0$, let $\nu_{(K, \epsilon)} \in \ell^1(X)$ be a positive vector of norm 1 such that $||x\nu_{(K, \epsilon)} - \nu_{(K, \epsilon)}|| < \epsilon||\nu||$ for all $x \in K$. Then by example \ref{l1_mean}, we can consider the $\nu_{(K, \epsilon)}$ as elements of $\mathcal{M}(X)$.

We define a relation on $D := \{ (K, \epsilon) : K \subseteq_f G, \, \epsilon > 0 \}$ by $(K, \epsilon) \preceq (K', \epsilon')$ if $K \subseteq K'$ and $\epsilon \geq \epsilon'$. Then $(D, \preceq)$ is a directed set, so $(\nu_{(K, \epsilon)})_{(K, \epsilon) \in D}$ is a net in $\mathcal{M}(X)$. But by lemma \ref{M_compact}, $\mathcal{M}(X)$ is compact, so this net admits an accumulation point $\mu \in \mathcal{M}(X)$. We claim that this is an invariant measure. \\

Fix $x \in G$. Then the map $\mathbb{R}^X \to \mathbb{R}^X : \nu \mapsto (x \nu - \nu)$ is continuous. This is because both the subtraction and the action of $G$ on $\mathbb{R}^X$ are. Therefore, since $\mu$ is an accumulation point of $(\nu_{(K, \epsilon)})_{(K, \epsilon) \in D}$, we deduce that $(x\mu - \mu)$ is an accumulation point of $(x\nu_{(K, \epsilon)} - \nu_{(K, \epsilon)})_{(K, \epsilon) \in D} \subseteq \ell^1(X)$. Now by uniqueness of the limit in Hausdorff spaces (such as $\mathbb{R}^X$) and by the fact that if a net has a unique limit then any accumulation point is equal to it; it suffices to show that this last net converges to $0 \in \ell^1(X) \subseteq \mathbb{R}^X$. Let $\epsilon > 0$. Then if $(\{x\}, \epsilon) \preceq (K, \epsilon')$, we have $||x\nu_{(K, \epsilon')} - \nu_{(K, \epsilon')}|| < \epsilon' ||\nu_{(K, \epsilon')}|| = \epsilon' \leq \epsilon$. This concludes the proof.
\end{proof}

The rest of this section will be devoted to proving that all of these properties are equivalent to each other, and to the non-existence of a paradoxical decomposition.

\subsection{Marriage lemmas}

Here we prove the infinite version of Hall's marriage lemma (which is also used in theorem 2.1.17), as well as the bigamist lemma, which will be the key in finding the paradoxical decomposition. \\

For a bipartite graph with parts $I$ and $O$, and any subset $A \subseteq I$ we will denote, as usual, $\partial A$ for the neighbours of $A$. We will simply note $\partial v$ for $\partial \{ v \}$.

\begin{theorem}[Hall Marriage lemma]
\label{Hall}

Consider a bipartite graph $I \sqcup O$, such that each vertex of $I$ has finite degree. Then there exists a matching covering $I$ if and only if for all $A \subseteq_f I$ we have $|\partial A| \geq |A|$.
\end{theorem}

\begin{remark}
This generalization relies on the finite version of the marriage lemma.
\end{remark}

\begin{proof}
As in the finite case, the first direction is trivial. So consider a bipartite graph satisfying the expanding condition. Define $K := \prod\limits_{v \in I} \partial v \subseteq O^I$. Equip $O$ with the discrete topology, and $O^I$ with the product topology. Since each $\partial v$ is finite, by Tychonoff's theorem $K$ is a compact subset of $O^I$. Now for each $F \subseteq_f I$, let $K_F$ denote the set of elements of $K$ that define a matching covering $F$. Then to say that there exists a matching covering $I$ is equivalent to say that $\bigcap\limits_{F \subseteq_f I} K_F \neq \emptyset$. By the finite version of Hall's marriage lemma, each $K_F \neq \emptyset$. Also, since we are imposing conditions on finitely many coordinates of $O^I$, and $O$ is discrete, $K_F \subseteq K$ is closed. Finally, if $F_1, \ldots, F_n \subseteq_f I$, then $K_{F_1} \cap \cdots \cap K_{F_n} = K_{F_1 \cup \cdots \cup F_n} \neq \emptyset$. Therefore the collection $(K_F)_{F \subseteq_f I}$ is a collection of closed sets satisfying the finite intersection property, so by the compactness of $K$, we conclude that $\bigcap\limits_{F \subseteq_f I} K_F \neq \emptyset$.
\end{proof}

\begin{definition}
Let $(I \sqcup O, E)$ be a bipartite graph. A \textbf{bigamist matching} is a pair of matchings $M_\pm$ covering $I$ and touching two disjoint sets of vertices in $O$.
\end{definition}

\begin{corollary}[Bigamist lemma]
Consider a bipartite graph $I \sqcup O$, such that each vertex of $I$ has finite degree. Then there exists a bigamist matching if and only if for all $A \subseteq_f I$ we have $|\partial A| \geq 2|A|$.
\end{corollary}

\begin{proof}
Once again, the first direction is trivial. Define $\tilde{I} := I \times \{ \pm \}$. Define a new bipartite graph on $\tilde{I} \sqcup O$ by connecting $(v, \pm)$ to $\partial v$ for all $v \in I$. Then this new graph satisfies the usual expanding condition, so there is a matching covering $\tilde{I}$. We get the two matchings we were looking for by identifying $I$ first with $I \times \{ + \}$, then with $I \times \{ - \}$.
\end{proof}

\subsection{Tarski's theorem}

We will prove that an action that does not satisfy the Følner condition is paradoxical. Then this, together with proposition \ref{equiv}, implies that (F), (R), amenability and non-paradoxicality are all equivalent properties for a group action.

\begin{lemma}
\label{non_F}

Let $G$ act on $X$, and suppose that the action does not satisfy (F). Then there is exists some $K \subseteq_f G$ such that for all $A \subseteq_f X$ we have $|KA| \geq 2|A|$.
\end{lemma}

\begin{proof}
By lemma \ref{(F)app}, if the action does not satisfy (F), then it does not satisfy (F''). That is, there exists some $K_0 \subseteq G$ and some $\epsilon > 0$ such that for all $A \subseteq_f X$ we have $|KA| \geq (1 + \epsilon)|A|$. Now let $n$ be such that $(1 + \epsilon)^n \geq 2$. Let $K := K_0^n$. Then for all $A \subseteq_f X$:
$$|KA| = |K_0(K_0^{n-1}A)| \geq (1 + \epsilon) |K_0^{n-1}A| \geq \cdots \geq (1 + \epsilon)^n |A| \geq 2|A|.$$
\end{proof}

This result looks very promising to apply the bigamist lemma. But to do this, we must first change how we think about realizations, as they are defined in 2.1.1:

\begin{definition}
Let $G$ be a group acting on a set $X$, and let $A, B \subseteq X$. A \textbf{piecewise-$G$} map is a map $f : A \to B$ such that there exists a partition $A = \bigsqcup\limits_{i = 1}^n A_i$ and elements $x_1, \ldots, x_n \in G$ such that $f(a) = x_i a$ whenever $a \in A_i$.
\end{definition}

Then a realization is just a bijective piecewise-$G$ map. With this language, corollary 2.1.3 becomes:

\begin{proposition}
\label{paradox}

Let $G$ be a group acting on a set $X$. Then the action is paradoxical if and only if there exist two piecewise-$G$ injections $f_\pm : X \to X$ with disjoint image.
\end{proposition}

However, this definition of piecewise-$G$ maps is a bit painful to work with. The next lemma solves that problem:

\begin{lemma}
Let $G$ be a group acting on a set $X$, and let $A, B \subseteq X$. Let $f : A \to B$. Then $f$ is piecewise-$G$ if and only if there exists some $K \subseteq_f G$ such that for all $a \in A$ we have $f(a) \in Ka$.
\end{lemma}

\begin{proof}
$\Rightarrow$. Let $f : A \to B$ be a piecewise-$G$ map, and let $A_i, x_i$ be as in the definition. Let $K := \{ x_1, \ldots, x_n \}$. Then for all $a \in A_i, f(a) = x_i a \in K a$, so for all $a \in A, f(a) \in Ka$. \\

$\Leftarrow$. Let $K \subseteq_f G$ be such that $f : A \to B$ satisfies: for all $a \in A, f(a) \in Ka$. Denote $K := \{ x_1, \ldots, x_n \}$. Define $A_i := \{a \in A : f(a) = x_i a$ and $i$ is minimal for this property$\}$.  Then the $A_i$ form a partition of $A$, and by definition $f(a) = x_i a$ whenever $a \in A_i$.
\end{proof}

We are now ready to prove:

\begin{theorem}[Tarski's theorem]
Let $G$ be a group acting on a set $X$. Then the action is amenable if and only if it is non-paradoxical.
\end{theorem}

\begin{proof}
$\Rightarrow$. Suppose that the action is amenable, so let $\mu$ be a $G$-invariant mean. Suppose by contradiction that there is a paradoxical decomposition $X = A \sqcup B$, with $A \sim X \sim B$. Since piecewise-$G$ maps do not affect $\mu$, we have $1 = \mu(X) = \mu(A) + \mu(B) = \mu(X) + \mu(X) = 2$, a contradiction. \\

$\Leftarrow$. By proposition \ref{equiv}, it is enough to prove that if the action does not satisfy the Følner condition, then it is paradoxical. So suppose that this is the case. By lemma \ref{non_F}, there exists some $K \subseteq_f G$ such that for all $A \subseteq_f X$ we have $|KA| \geq 2|A|$. Consider the bipartite graph on $X \sqcup X$, where every input $a \in X$ is connected to $Ka$. Then this graph satisfies the condition of the bigamist lemma, so there exists a bigamist matching. Let $f_+ : X \to X$ be the map assigning each $a \in X$ to its first match, and $f_-$ to its second match. Then for all $a \in X, \, f_\pm(a) \in Ka$, so $f_\pm$ are piecewise-$G$ injections with disjoint image. By proposition \ref{paradox}, the action is paradoxical.
\end{proof}

\pagebreak

\section{Uniqueness of measures}
\label{Lebesgue}

The aim of this section is to give a detailed proof of the uniqueness of the Lebesgue measure, as stated in 2.2.9. The two theorems below and respective proofs are taken from \cite[Theorems 1.10 and 3.4]{meas}; though a little simplified since we are only interested in the case of finite measures, as for the sphere.

Throughout this section, $X$ will be a metric space, and $\mathcal{B}$ its Borel $\sigma$-algebra.

\begin{lemma}
\label{gdelta}

Let $\mu$ be a countably additive measure on $(X, \mathcal{B})$. Then
\begin{enumerate}
\item For any $\epsilon > 0$ and any $G_\delta$ set $A$, there exists some open set $V$ such that $\mu(V \, \backslash \, A) < \epsilon$.
\item Any closed set is a $G_\delta$ set.
\end{enumerate}
\end{lemma}

\begin{proof}
1. Let $A = \bigcap\limits_{n \geq 1} V_n$ be a $G_\delta$ set. Up to replacing $V_n$ with $\bigcap\limits_{m = 1}^n V_m$, we may assume that $V_{n+1} \subseteq V_n$ for all $n \geq 1$. Then:
$$0 = \mu(\bigcap\limits_{n \geq 1} V_n \, \backslash \, A) = \mu(\bigcap\limits_{n \geq 1} (V_n \, \backslash \, A)) = \lim\limits_{n \rightarrow \infty} \mu(V_n \, \backslash \, A).$$
Therefore, for all $\epsilon$, we might choose $n$ sufficiently large to get $\mu(V_n \, \backslash \, A) < \epsilon$. \\

2. Let $C \subseteq X$ be closed, and for any $\epsilon > 0$ define $C_\epsilon := \bigcup\limits_{x \in C} B(x, \epsilon)$, which is open. Then $C \subseteq \bigcap\limits_{n \geq 1} C_{\frac{1}{n}}$. Also, if $y \in \bigcap\limits_{n \geq 1} C_{\frac{1}{n}}$, then for all $\epsilon > 0$ there exists some $x \in C$ such that $d(x, y) < \epsilon$. So $y$ is a limit point of $C$, and since $C$ is closed, $y \in C$. So $C = \bigcap\limits_{n \geq 1} C_{\frac{1}{n}}$.
\end{proof}

\begin{theorem}
\label{110}

Let $\mu$ be a countably additive measure on $(X, \mathcal{B})$ such that $\mu(X) < \infty$. Then for all $A \in \mathcal{B}$, for all $\epsilon > 0$, there exist $C$ closed, $V$ open such that $C \subseteq A \subseteq V$ and $\mu(V \, \backslash \, C) < \epsilon$. In particular, for the same sets, we have $\mu(V \, \backslash \, A), \, \mu(A \, \backslash \, C) < \epsilon$.
\end{theorem}

\begin{proof}
Define $\mathcal{A} \subseteq \mathcal{B}$ to be the collection of Borel sets satisfying the hypothesis. We want to show that $\mathcal{A} = \mathcal{B}$. First note that by lemma \ref{gdelta}, all closed sets are in $\mathcal{A}$, so we only need to show that $\mathcal{A}$ is a $\sigma$-algebra. Clearly $X \in \mathcal{A}$. If $A \in \mathcal{A}$, then for all $\epsilon > 0$, taking the respective $C$ and $V$, we have that $V^c \subseteq A^c \subseteq C^c$ and $\mu(C^c \, \backslash \, V^c) = \mu(V \, \backslash \, C) < \epsilon$. So $A^c \in \mathcal{A}$.

Finally, let $(A_n)_{n \geq 1} \subseteq \mathcal{A}$ and fix $\epsilon > 0$. For all $n \geq 1$, let $C_n \subseteq A_n \subseteq V_n$ be as in the definition, with $\frac{\epsilon}{2^{n+1}}$. Define $C := \bigcap\limits_{n \geq 1} C_n$, and similarly $A$ and $V$. Then $C$ is closed and $V$ is a $G_\delta$-set. By lemma \ref{gdelta}, there exists some $W$ open such that $V \subseteq W$ and $\mu(W \, \backslash \, V) < \frac{\epsilon}{2}$. So we have: $C \subseteq A \subseteq V \subseteq W$, with $C$ closed, $W$ open, and:
$$\mu(V \, \backslash \, C) \leq \mu(\bigcap\limits_{n \geq 1} (V_n \, \backslash \, C_n)) \leq \sum\limits_{n \geq 1} \mu(V_n \, \backslash \, C_n) < \sum\limits_{n \geq 1} \frac{\epsilon}{2^{n+1}} = \frac{\epsilon}{2}.$$
So $\mu(W \, \backslash \, C) = \mu(W \, \backslash \, V) + \mu(V \, \backslash \, C) < \epsilon$ and $A \in \mathcal{A}$. This concludes the proof.
\end{proof}

\begin{corollary}
\label{un_meas}

Let $\mu, \nu$ be countably additive measures on $(X, \mathcal{B})$ such that $\mu(X), \nu(X) < \infty$. Suppose that $\mu$ and $\nu$ agree on open sets. Then $\mu = \nu$.
\end{corollary}

\begin{proof}
Let $A \in \mathcal{B}$ and let $\epsilon > 0$. Then by theorem \ref{110}, for all $\epsilon > 0$ there exists $V$ open such that $A \subseteq V$ and $\mu(V \, \backslash \, A), \nu(V \, \backslash \, A) < \epsilon$ (just take the intersection of the two open sets whose existence is guaranteed by the theorem). Then:
$$|\mu(A) - \nu(A)| = |(\mu(V) - \mu(V \, \backslash \, A)) - (\nu(V) - \nu(V \, \backslash \, A)| = |- \mu(V \, \backslash \, A) + \nu(V \, \backslash \, A)| < 2 \epsilon.$$
This being true for all $\epsilon > 0$, we conclude that $\mu(A) = \nu(A)$.
\end{proof}

\begin{theorem}
\label{34}

Let $\mu, \nu$ be countably additive measures on $(X, \mathcal{B})$, where $X$ is separable. Suppose that $\mu(X) = \nu(X) < \infty$ and that $\mu(B(x, r)) = g(r)$ and $\nu(B(x, r)) = h(r)$ are positive and independent of $x$. Then $\mu = \nu$.
\end{theorem}

\begin{proof}
Let $U \subseteq X$ be open. Then $\lim\limits_{r \to 0} h(r)^{-1} \nu(U \cap B(x, r)) = 1$ for all $x \in U$. Therefore:
$$\mu(U) = \int_U \lim\limits_{r \to 0} h(r)^{-1} \nu(U \cap B(x, r)) d \mu (x) \leq \liminf\limits_{r \to 0} h(r)^{-1} \int_U \nu(U \cap B(x, r)) d \mu (x)$$
by Fatou's lemma. Now since $X$ is separable and $\mu, \nu$ are finite, we can apply Fubini's theorem. Also, note that for $x, y \in U$, we have $x \in B(y, r)$ if and only if $y \in B(x, r)$. Thus:
$$\int_U \nu(U \cap B(x, r)) d \mu (x) = \int_U \int_U \chi_{B(x, r)} (y) d \nu (y) d \mu (x) = $$
$$ = \int_U \int_U \chi_{B(y, r)} (x) d \mu (x) d \nu (y) = \int_U \mu(U \cap B(y, r)) d \nu (y) \leq g(r) \nu(U).$$
Therefore
$$\mu(U) \leq \liminf\limits_{r \to 0} \frac{g(r)}{h(r)} \nu(U),$$
and by the same argument
$$\nu(U) \leq \liminf\limits_{r \to 0} \frac{h(r)}{g(r)} \mu(U).$$
So $\lim\limits_{r \to 0} \frac{g(r)}{h(r)} =: c$ exists and $\mu(U) = c \nu(U)$. This being true for all open sets, $\mu = c \nu$ by corollary \ref{un_meas}. But $\mu$ and $\nu$ agree on $X$, so $c = 1$ and $\mu = \nu$.
\end{proof}

\begin{corollary}
The Lebesgue measure $\lambda$ is the unique countably additive measure of total measure 1 on the Lebesgue sets of $S^n$ that is invariant under rotation.
\end{corollary}

\begin{proof}
Let $\mathcal{B}$ be the Borel $\sigma$-algebra and $\mathcal{L}$ the Lebesgue $\sigma$-algebra on $S^n$. Note that $S^n$ is a compact metric space, so it is separable. Let $\mu$ be a measure satisfying the hypotheses. Then for all $r > 0$, we have $\mu(B(x, r)) = \mu(B(y, r))$ for all $x, y \in S^n$. Indeed, $SO(n+1)$ acts transitively on $S^n$, so there exists $g \in SO(n+1)$ such that $g(x) = y$. Then $\mu(B(y, r)) = \mu(g(B(x, r))) = \mu(B(x, r))$. Also, by compactness, for any $r > 0$ there exists a finite set $F$ such that $S^n = \bigcup\limits_{x \in F} B(x, r)$. This implies that $\mu(B(x, r)) > 0$. So $\mu$ satisfies the hypotheses of theorem \ref{34}, and we conclude that $\lambda = \mu$ on $\mathcal{B}$. \\

We are left to show that if $A \in \mathcal{L}$ is a null set, then $\mu(A) = 0$. By the definition of $\lambda$, if $\lambda(A) = 0$, then for all $\epsilon > 0$ there exists an open set $U$ such that $A \subseteq U$ and $\lambda(U) < \epsilon$. Therefore $\mu(A) \leq \mu(U) = \lambda(U) < \epsilon$. This being true for all $\epsilon$, we conclude that $\mu(A) = 0$.
\end{proof}

\pagebreak

\section{Typos}

Here is a list of all the typos that I have found while reading the book. \\

P. 10: In the proof of lemma 2.1.9, the right-hand-side of the next-to-last equation should be $8 \sum\limits_{j = 0}^k p^j $. \\

P. 12: In the proof of corollary 2.1.12, $D = \{ x \in S^2 \, | \, \exists 1 \neq \gamma \in F, \, \gamma(x) = x \}$, with $\gamma$ instead of $r$. \\

P.12: In the proof of proposition 2.1.13, $n > 0$, not $n \geq 0$, since clearly $\rho^0(D) \cap D = D \neq \emptyset$. \\

P. 13: At the end of the proof of theorem 2.1.17, $kA \lesssim kX$, not $kA \leq kX$. \\

P. 15: At the end of the proof of lemma 2.2.6, the last equation should be with $\max \{| \, || f_i ||_\infty | \}$, with the absolute value. \\

P. 16: In the statement at the beginning of the page, it should be $\sup\limits_{x \in A} h(x)$ instead of $||h||_\infty$. This is discussed in more detail in comment \ref{227}. \\

P. 16: At the end of the proof of proposition 2.2.5, it should be $m(f) \leq ||f||_\infty$, not $|m(f)|$. \\

P. 21: In  the second paragraph of the proof, the measure in the integral should be $d \lambda$, not $dg$. \\

P. 23: At the end of the paragraph starting with "Now we use property $(F^*)$..." it should be "$\lambda(n U \Delta U) < \frac{\epsilon}{2} \lambda(U)$ for all $n \in N$", with $n$ instead of $k$. \\

P. 25: At the beginning of the page, the definition of $L^2(G, W)$ should be with $||f(g)||^2$, not with $||f(g)||$. Also, the right definition of $\theta$ is: $\theta(f \otimes v)(g) = f(g)(\pi(g^{-1})v)$. This is discussed in more detail in comment \ref{3111}. \\

P. 26: At the end of the page, it should read "every discrete Kazhdan group is finitely generated", not "every countable Kazhdan group is finitely generated". See comment \ref{compact_generation} for more details. \\

P. 30: The statement should read: "Let $\Gamma$ be a discrete finitely generated Kazhdan group", instead of "finitely generated Kazhdan group". See comment \ref{331} for more details. The same holds for proposition 3.3.7 at page 33. \\

P. 31: The set $Y_n$ should be defined as $\{e_3, \ldots, e_{[\frac{n}{2}]} \}$, starting at $e_3$ instead of $e_1$. See comment \ref{333} for more details. \\

P. 33: The sum should run over $X \in \Gamma / N$, not $X \in G/H$. The typo appears four times: in all of the sums running over $G/H$ and in the definition of $B_j$. Also, in the definition of $\varphi$, it should once again be $\Gamma / N$, not $G/N$. There are other imprecisions in this proof: see comment \ref{337} for more details. \\

P. 46: In the equation expressing $\langle df, df \rangle$, the second line of the equation should be
$$\sum\limits_{v \in V^+} g(v) \sum\limits_{u \in V} \delta_{vu} (g(v) - g(u)) + \cdots$$
with the second sum running over $V$ and not over $V^+$. \\

P. 47: In the expansion of $(f^2(x) - f^2(y))$, the last term should be $(\beta_{i - j + 1}^2 - \beta_{i - j}^2)$, with $(i - j + 1)$ instead of $(i - j - 1)$. \\

P. 51: In the last part of example C., the paragraph starting with "As a corollary..", it should always be $\gamma$ instead of $\tau$. There are two instances when this typo appears. First: "... with respect to the generators $\gamma$ and $\sigma$...". Second: "Then $\gamma \cdot A_m = A_m$ and...". \\

P. 55: In proposition 4.5.1, the definition of the return generating function should be $R(z) = \sum\limits_{n = 0}^\infty r_n z^n$, with the sum running over $n$ instead of $r$. \\

P. 55: The proof of proposition 4.5.2 should start with: "Fix the vertex $x_0$ in $X$", with $x_0$ instead of $v_0$. \\

P. 57: At the beginning of the proof of proposition 4.5.4, it should be: "If diameter($X$) $\geq 2r + 2$". See comment \ref{454} for more details. \\

\pagebreak

\end{document}